\documentclass[11pt,reqno]{amsart}
\usepackage[utf8]{inputenc}	
\usepackage{amsthm,amsfonts,amscd}
\usepackage{multirow,booktabs,verbatim}
\usepackage
{hyperref}

\usepackage{fullpage}
\usepackage{lastpage}
\usepackage{enumitem}
\usepackage{fancyhdr}
\usepackage{mathrsfs}
\usepackage{wrapfig}
\usepackage{setspace}
\usepackage{calc}
\usepackage{multicol}
\usepackage{cancel}
\usepackage{bbm}
\usepackage{amsmath}
\usepackage{tikz}
\usepackage{ulem}
\usepackage{cleveref}
\usepackage{cases}
\usepackage{natbib}

\usetikzlibrary{shapes.geometric, arrows}

\DeclareMathOperator*{\argmin}{arg\,min}
\DeclareMathOperator*{\argmax}{arg\,max}

\usepackage{amssymb}
\newcommand{\vertiii}[1]{{\left\vert\kern-0.25ex\left\vert\kern-0.25ex\left\vert #1\right\vert\kern-0.25ex\right\vert\kern-0.25ex\right\vert}}

\usepackage{empheq}
\usepackage{framed}
\usepackage[most]{tcolorbox}
\usepackage{xcolor}
\colorlet{shadecolor}{orange!15}

\usepackage{soul}
\usepackage[margin=1in]{geometry}

\numberwithin{equation}{section}

\newtheorem{theorem}{Theorem}[section]
\newtheorem{lemma}{Lemma}[section]
\newtheorem{assumption}{Assumption}[section]
\newtheorem{proposition}{Proposition}[section]
\newtheorem{corollary}{Corollary}[section]

\newtheorem{definition}{Definition}[section]
\newtheorem{remark}{Remark}[section]

\makeatletter

\newcommand{\Rmnum}[1]{\expandafter\@slowromancap\romannumeral #1@}
\makeatother
\title{A Principal-Agent Mean-Field Game Model of Insurance with Risk Interdependence}

\author{Asaf Cohen}
\address{Department of Mathematics, University of Michigan, Ann Arbor, MI 48109}
\email{shloshim@gmail.com}
\thanks{* Asaf Cohen and Mingyan Liu are supported in part by the IUCRC Planning Grant, University of Michigan: Cyber and Terrorism Insurance Studies (CATIS) Center (Grant No. 2514919). Asaf Cohen's research is also supported in part by the National Science Foundation under Grant No. DMS-2505998.}

\author{Ruolan He}
\address{Department of Mathematics, University of Michigan, Ann Arbor, MI 48109}
\email{ruolan@umich.edu}

\author{Mingyan Liu}
\address{Department of Engineering, University of Michigan, Ann Arbor, MI 48109}
\email{mingyan@umich.edu}

\begin{document}

\begin{abstract}
We study an insurance contract-design problem under moral hazard, endogenous participation, and strategic risk interdependence. Because the resulting $N$-agent game suffers from the curse of dimensionality, we approximate the strategic interactions via a heterogeneous mean-field game. We rigorously establish the existence of a lower-level mean-field Nash equilibrium using measurable selection arguments and the Kakutani fixed-point theorem. By proving the $L^1$-Lipschitz continuity of the aggregate participation threshold, we further establish equilibrium uniqueness via a contraction mapping. We then embed this mean-field response into the insurer’s upper-level Stackelberg optimization problem. We formulate the objective through general performance envelopes to accommodate potential equilibrium multiplicity, proving the existence of upper-level $\varepsilon$-optimal contracts, and demonstrating the existence of an exact Stackelberg equilibrium under the uniqueness regime. We conclude by extending the model to finite contract menus, providing numerical evidence that multi-contract screening improves the principal's expected payoff in interdependent risk environments.
\end{abstract}

\maketitle

\noindent{\bf Keywords:} mean-field games, principal-agent problems, insurance contract design, moral hazard, risk interdependence, Stackelberg equilibrium

\noindent{\bf AMS Classification:} 91A16, 49N80, 91B43, 91B41, 91A65



\tableofcontents

\section{Introduction}

We study a principal-agent model of insurance with risk interdependence, motivated fundamentally by the challenges of cyber-risk insurance. In such environments, an insurer (principal) offers a contract to heterogeneous, risk-averse policyholders (agents) who choose costly prevention effort and whether to participate in the contract. Because effort is not directly observable, the insurer faces a classic moral hazard problem. However, the insurer observes a noisy pre-screening signal, and the contract uses screening-based incentives and partial coverage to shape participation and induce optimal safety investments.

We first formulate an $N$-agent game, in which each agent's loss depends on both individual effort and the aggregate behavior of others, but the combination of strategic dependence between all agents' responses and endogenous participation makes the agents' Nash equilibrium analytically intractable due to the curse of dimensionality and the complex fixed-point structure of the participation decisions. We therefore turn to a heterogeneous mean-field model, which allows us to analyze the Nash equilibrium among agents, establish existence of a mean-field Nash equilibrium, and, under additional conditions, uniqueness. This leads naturally to a two-level structure: a lower-level mean-field equilibrium problem for the agents, and an upper-level contract-design problem for the insurer. We then use the lower-level agents' equilibrium characterization to study the insurer's upper-level contract-design problem. Under equilibrium multiplicity, we formulate the insurer's objective through general performance criteria and establish the existence of upper-level $\varepsilon$-optimal contracts; under a unique lower-level equilibrium response, we recover the existence of an exact Stackelberg equilibrium. Building on this framework, we further extend the model from a single contract to a contract menu, illustrating how menu design can screen across risk-aversion types and, in our numerical examples, improve insurer's profit relative to a single-contract benchmark.

{\bf Literature review.} Principal-agent models provide a fundamental framework for studying contractual relationships under moral hazard and asymmetric information. In these models, a principal designs a contract to manage risk and provide incentives for an agent whose effort is costly and not directly observable, with outcomes depending instead on noisy signals of the agent's actions. This formulation originates in the classical work of \citet{holmstrom1979moral, 9ab814a1-71bf-3750-8dda-9705acaa1dda}, and has been extensively developed in both static and dynamic settings (see, e.g., \citet{sannikov2008continuous, cvitanic2006optimal}). These models are typically solved via a Stackelberg equilibrium, in which the principal acts as a leader by committing to a contract, and the agent responds optimally given the contract.

Insurance markets provide a canonical application of the principal-agent framework. The Stackelberg structure naturally reflects the timing of decisions: the insurer first offers a contract, and insured agents then choose their effort levels in response. The agents' prevention efforts affect loss outcomes but cannot be directly monitored. Instead, insurers rely on imperfect and noisy signals from pre-screening, such as audits or third-party assessments, to evaluate agents' underlying risk and effort levels. These pre-screening mechanisms play a central role in improving incentive compatibility and reducing moral hazard, and have been studied extensively in the insurance literature, for example in \citet{boyer2004overcompensation, dohertysmetters2005}. In many insurance settings, however, an agent's loss exposure depends not only on its own effort but also on the actions of other agents in the same environment. Such risk interdependence introduces strategic externalities among the insured population, so that the insurer's contract design must account not only for individual moral hazard but also for the equilibrium interaction among agents.

Related questions have also been studied in the cyber-insurance literature. Most closely related to our setting, \citet{khalili2018designing} analyze contract design with pre-screening, security interdependence and voluntary participation under a profit-maximizing insurer. Their model captures the interaction between screening and security interdependence in contract design, but is limited to homogeneous agents when the number exceeds 2.
Our paper extends this line of work by introducing a heterogeneous mean-field model with endogenous participation, establishing equilibrium existence and uniqueness results at the lower level, and then studying the insurer's upper-level contract design problem within this framework. Earlier work studies how cyber-insurance affects security investment and network outcomes under different market structures and interdependence assumptions. For example, \citet{lelarge2009economic} examine how insurance can provide incentives for self-protection in the presence of interdependent security risks, while \citet{shetty2010competitive} study cyber-insurance in a competitive environment and analyze how market structure affects security and welfare outcomes.

Within the principal-agent literature, a large strand of work focuses on single-agent models, where the principal designs a contract to induce effort from a single strategic agent (see, e.g., \citet{holmstrom1979moral, holmstrom1991multitask}). There are also extensions to multi-agent settings, where the principal interacts with several agents simultaneously, for example, in \citet{attar2010multiple} and \citet{castiglioni2023multi}. When multiple agents are present, the analysis involves two layers of equilibrium: a Nash equilibrium among agents' effort and participation decisions conditional on the principal's policy, and a Stackelberg equilibrium between the principal and the population of agents. A key challenge in such settings is to characterize the agents' Nash equilibrium for a given contract. When agents are heterogeneous and their risks are interdependent, each agent's optimal response depends on the aggregate behavior of others, making the equilibrium analysis high-dimensional and analytically intractable as the number of agents grows.

Mean-field game formulations provide a natural approach to this difficulty by approximating strategic interactions through population-level aggregates; see, for example, \citet{lasry1, lasry2} and \citet{Huang2006}. Several recent works combine principal-agent problems with mean-field interactions, including \citet{elie2019tale} and \citet{elie2021mean}. In our setting, the mean-field interaction arises through risk interdependence: each agent's loss distribution depends on both their own effort and the population-average effort, so that individual decisions and aggregate outcomes are jointly determined. A distinguishing feature of our model is that agents are heterogeneous in multiple dimensions (risk aversion, effort cost, and pre-screening noise) and make an endogenous participation decision in addition to choosing effort. The mean-field equilibrium is therefore characterized by a fixed point in the aggregate effort that simultaneously determines individual best responses, participation thresholds, and the equilibrium composition of the insured population.

{\bf Our contribution.} We make three main contributions. First, for a fixed contract, we formulate the lower-level agents' heterogeneous mean-field game with endogenous participation and establish the existence of a mean-field Nash equilibrium via the Kakutani fixed-point theorem, see \Cref{thm_NE}. Under additional regularity and a small-interaction condition on the interdependence parameter, we show that the aggregate mean-field map is a contraction, thereby yielding uniqueness of the equilibrium, see \Cref{thm_NEuniq}. 

Second, we analyze the upper-level insurer's contract-design problem. When the lower-level mean-field Nash equilibrium may be non-unique, a contract can induce a set of equilibrium payoffs rather than a single payoff value. We therefore formulate the insurer's objective through a general performance criterion and prove the existence of upper-level $\varepsilon$-optimal contracts, see \Cref{thm: insurer}. We then discuss worst- and best-equilibrium payoff envelopes as economically meaningful criteria under equilibrium multiplicity, and show that under a unique lower-level equilibrium response the insurer's problem admits an exact Stackelberg equilibrium.

Third, we extend the framework to a contract menu and study how differential pricing and coverage can be used to screen across risk-aversion types. On the numerical side, we illustrate the mean-field fixed point and comparative statics of participation, examine how the insurer's optimal contract varies with population risk aversion and pre-screening noise, and show that an optimized two-contract menu can outperform the optimized one-contract benchmark in insurer's per-capita profit. We also explore larger menus numerically and find that moving beyond two contracts yields only modest additional gains.

{\bf Economic implications of the pre-screening mechanism.} Beyond the mathematical analysis, a core focus of this paper is explaining the real-world economic value of pre-screening when agents' risks are interconnected. Traditional models usually manage unobservable effort by penalizing agents after a loss occurs (for example, through deductibles). In our setting, however, insurers use imperfect pre-screening signals to estimate an agent's effort before any damage happens. This allows them to adjust coverage and premiums proactively. The accuracy of this screening technology directly determines how well the insurer can control the ``domino effect'' of risk spreading between agents. Ultimately, our framework provides a theoretical foundation for understanding how better underwriting tools, such as advanced cyber-risk audits or real-time threat monitoring and detection,
can expand the limits of what insurers can safely cover when risks are highly correlated.

{\bf Methodological challenges and proof techniques.} The primary analytical hurdle in our framework stems from the intricate coupling of endogenous participation, population heterogeneity, and strategic risk interdependence. Because agents selectively opt into the contract based on their idiosyncratic types—spanning risk aversion, effort cost, and pre-screening noise—the aggregate mean field is driven by a complex, shifting composition of participating agents. To rigorously formalize the lower-level equilibrium, we must bridge pointwise individual optimization with population-level aggregation. Unlike standard applications of the Kakutani fixed-point theorem where topological continuity suffices, our aggregation requires integrating individual best responses across the heterogeneous population measure. For this Aumann integral to be mathematically well-defined, we must establish the weak measurability of the agents' best-response correspondences. We achieve this by deploying measurable selection theorems (specifically the Kuratowski--Ryll-Nardzewski theorem) to construct rigorous expected-effort mappings, which then allows us to apply the Kakutani fixed-point theorem to the aggregate mean field. Furthermore, establishing equilibrium uniqueness presents a distinct topological challenge: we must tightly bound the mass of marginal agents who alter their participation decisions in response to mean-field perturbations. We overcome this by leveraging the implicit function theorem on the participation boundaries to prove that the aggregate participation probability is Lipschitz continuous in the $L^1$ norm, thereby guaranteeing that the mean-field mapping is a strict contraction.

{\bf Organization.} The remainder of the paper is organized as follows. \Cref{sec:N-agent} provides the $N$-agent model as a motivation to study the mean-field model. We formulate the finite-population model and define the mixed Nash equilibrium among agents for a given contract. \Cref{sec:MFG} introduces the lower-level heterogeneous mean-field model, establishes the existence of the mean-field Nash equilibrium via the Kakutani fixed-point theorem, and proves uniqueness under a contraction condition on the interdependence parameter. \Cref{sec:insurer} returns to the upper-level insurer's contract-design problem, formulates the problem through general performance criteria under possible equilibrium multiplicity, proves the existence of upper-level $\varepsilon$-optimal contracts, and discusses worst- and best-equilibrium performance criteria as well as the unique-equilibrium case in which an exact Stackelberg equilibrium exists. It then presents numerical illustrations of the optimal contract in a unique-equilibrium regime. \Cref{sec: menu} extends the framework to a menu of contracts, examines how differential coverage and pricing enable screening across risk-aversion groups, and presents numerical comparisons between one-contract, two-contract, and larger-menu designs. The appendix contains the proofs of the main theoretical results.

\section{Motivational Model: $N$-Agent Model} \label{sec:N-agent}
We begin with a finite population of $N$ (risk-averse) agents. Each agent $i\in\{1,\dots,N\}$ is characterized by a {\it type} $\theta_i = (\gamma_i, c_i, \sigma_i)$, where $\gamma_i > 0$ denotes the {\it risk-aversion parameter}, $c_i > 0$ denotes the {\it cost of effort}, $\sigma_i > 0$ denotes the {\it pre-screening noise}. Let $\mathbf{e} = (e_1, e_2, \ldots, e_N)$, where each $e_i \in [0, \infty)$, denote the vector of efforts, and let $\overline{\mathbf{e}}_{-i} := \frac{1}{N-1} \sum_{j \neq i} e_j$ denote the average effort of all agents other than agent $i$. The loss of agent $i$ is given by
\[
L_{\mathbf{e}}^{(i)} \sim \mathcal{N}\left( \mu \Big(e_i + x \overline{\mathbf{e}}_{-i}\Big),\; \lambda\Big(e_i + x \overline{\mathbf{e}}_{-i}\Big) \right),
\]
where $\mathcal N(a, b)$ denotes the normal distribution with mean $a$ and variance $b$, the functions $\mu: [0, \infty) \to \mathbb R$ and $\lambda: [0, \infty) \to [0, \infty)$ stand for the mean and variance of the loss respectively, and $x$ is an {\it interdependence parameter} that captures how the average effort of the other agents affects agent $i$'s loss distribution. Economically, the parameter $x$ captures the intensity of the network's risk interdependence. In a cyber-insurance context, $x$ represents the ``contagion vulnerability'' of the system. A high $x$ implies that an agent's individual loss is highly susceptible to the aggregate security failures of the network—triggering a systemic domino effect of cascading losses. Conversely, as $x \to 0$, the model collapses to a standard, independent-risk insurance market. Therefore, the parameter $x$ serves as the critical mathematical link between individual moral hazard and systemic risk.\footnote{A natural question is whether the uniform interdependence parameter $x$ could be replaced by heterogeneous, pairwise interaction weights $x_{i,j}$ between any two agents $i$ and $j$. We believe that this is possible and that it would lead to a graphon mean-field game in the limit as the number of players goes to infinity. We discuss it further in Section \ref{sec:6} and leave the formal development of this extension for future research.}

Each agent decides whether to participate in the contract. An agent who remains uninsured bears the full loss exposure and chooses effort to maximize her utility. An agent who enters the contract pays the premium, receives partial coverage according to the contract, and chooses effort accordingly. Hence, for each agent, there are two possible effort responses: one conditional on being outside the contract and one conditional on entering it.

The agent's expected utility outside the contract is,
\begin{align*}
U^{\mathrm{out}}_i(e_i; \overline{\mathbf{e}}_{-i})
&= \mathbb{E}\big[ -\exp\{-\gamma_i\,(-L_{\mathbf{e}}^{(i)}-c_i\,e_i)\}\big] \\
&= -\exp\Big\{\gamma_i \Big(
\mu(e_i+ x \overline{\mathbf{e}}_{-i}) \;+\; c_i\,e_i \;+\; \tfrac{1}{2}\,\gamma_i\,\lambda(e_i+ x \overline{\mathbf{e}}_{-i})
\Big)\Big\}.
\end{align*}
Define the {\it certainty equivalent} as ${\rm CE}_i^{\mathrm{out}}(e;\overline{\mathbf{e}}_{-i})
:= \mu(e+x \overline{\mathbf{e}}_{-i})+c_i e+\tfrac12\gamma_i\lambda(e+x\overline{\mathbf{e}}_{-i})$. Because $U_i^{\mathrm{out}}=-\exp\{\gamma_i {\rm CE}_i^{\mathrm{out}}\}$ and $\gamma_i>0$, maximizing expected utility is equivalent to minimizing ${\rm CE}_i^{\mathrm{out}}$. The best response set is given by
\[
E^{\mathrm{out}}_i(\overline{\mathbf{e}}_{-i}) := \argmin_{e \geq 0} {\rm CE}_i^{\mathrm{out}}(e;\overline{\mathbf{e}}_{-i}).
\]
A best-response effort outside the contract is any measurable 
selection $e^{\mathrm{out}}_i(\overline{\mathbf{e}}_{-i}) \in E^{\mathrm{out}}_i(\overline{\mathbf{e}}_{-i})$.

The insurer offers a contract $(p, \alpha, \beta)$, where $p \in [0, \infty)$ is the {\it base premium}, $\alpha \in [0, \infty)$ is the {\it screening-based discount coefficient}, and $\beta \in [0, 1]$ is the {\it coverage ratio}. Economically, $p$ determines the agent's fixed payment for coverage, $\alpha$ measures the strength of discount based on the insurer's pre-screening assessment of the agent's effort, and $\beta$ determines the fraction of the realized loss reimbursed by the insurer. Under this contract, an insured agent $i$ pays a net premium $p - \alpha S(e_i)$ and bears only the uninsured portion $(1-\beta) L_{\boldsymbol{e}}^{(i)}$ of the realized loss. Thus, his expected utility will be given by
\begin{align*}
U^{\mathrm{in}}_i(e_i; \overline{\mathbf{e}}_{-i})
&= \mathbb{E}\!\left[ - \exp\!\left( -\gamma_i\,\big\{ -p + \alpha S(e_i) - L_{\mathbf{e}}^{(i)} + \beta L_{\mathbf{e}}^{(i)} - c_i e_i \big\} \right) \right] \\
&= - \exp\!\Big( \gamma_i\Big[
p + (c_i - \alpha)\, e_i
+ \tfrac{1}{2}\,\gamma_i\,\alpha^2\,\sigma_i^2 \\
&\qquad\qquad\qquad\quad
+ (1 - \beta)\,\mu\!\big(e_i + x \overline{\mathbf{e}}_{-i}\big)
+ \tfrac{1}{2}\,\gamma_i\,(1 - \beta)^2\,\lambda\!\big(e_i + x \overline{\mathbf{e}}_{-i}\big)
\Big] \Big),
\end{align*}
Define the certainty equivalent inside the contract as
\[
{\rm CE}_i^{\mathrm{in}}(e; \overline{\mathbf{e}}_{-i}):=
p+(c_i-\alpha)e+\tfrac12\,\gamma_i\alpha^2\sigma_i^2
+(1-\beta)\mu(e+x\overline{\mathbf{e}}_{-i})+\tfrac12\,\gamma_i(1-\beta)^2\lambda(e+x\overline{\mathbf{e}}_{-i}).
\]
The best response set of agent $i$ inside the contract is
\[
E^{\mathrm{in}}_i(\overline{\mathbf{e}}_{-i}) := \argmin_{e \geq 0} {\rm CE}_i^{\mathrm{in}}(e; \overline{\mathbf{e}}_{-i}).
\]
A best-response effort inside the contract is any measurable selection $e^{\mathrm{in}}_i(\overline{\mathbf{e}}_{-i}) \in E^{\mathrm{in}}_i(\overline{\mathbf{e}}_{-i})$.

Finally, participation is determined by comparing the maximal expected utilities. An agent participates in the contract only if the {\it individual rationality} (IR) condition holds, and the effort choices satisfy the {\it incentive compatibility} (IC) condition:
\begin{align*}
&\textbf{IR:} \qquad
\max_{e \geq 0} U_i^{\mathrm{in}}\big(e ;\overline{\mathbf e}_{-i}\big) \ge \max_{e \geq 0} U_i^{\mathrm{out}}\big(e ;\overline{\mathbf e}_{-i}\big).\\
&\textbf{IC:} \qquad \text{Given participation, the agent chooses effort } e_i^{\mathrm{in}}(\overline{\mathbf e}_{-i})
\in E^{\mathrm{in}}_i(\overline{\mathbf{e}}_{-i}).
\end{align*}

Given a contract $(p, \alpha, \beta)$, agents simultaneously decide whether to participate and how much effort to exert. To allow for indifference in equilibrium, we permit mixed participation: an equilibrium assigns to each agent $i$ probability $q_i \in [0,1]$ of entering the contract. With probability $q_i$, the agent enters and chooses an effort $e_i^{\mathrm{in}}(\overline{\boldsymbol{e}}_{-i}) \in E^{\mathrm{in}}_i(\overline{\mathbf{e}}_{-i})$; with probability $1-q_i$, the agent stays outside and chooses $e_i^{\mathrm{out}}(\overline{\boldsymbol{e}}_{-i}) \in E^{\mathrm{out}}_i(\overline{\mathbf{e}}_{-i})$.

Since participation is randomized, the agent’s expected utility is the weighted average of the utilities inside and outside the contract,
\[
U_i(e_i; \overline{\mathbf{e}}_{-i}) = q_i U_i^{\mathrm{in}}(e_i; \overline{\mathbf{e}}_{-i}) + (1-q_i) U_i^{\mathrm{out}}(e_i; \overline{\mathbf{e}}_{-i}).
\]
This formulation corresponds to a mixed Nash equilibrium among the agents.

\begin{definition}
A (mixed) Nash equilibrium in this $N$-agent setting is a collection of strategies $\{(e_i^{\mathrm{in}}, e_i^{\mathrm{out}}), q_i: i\in\{1,\dots,N\}\}$, where $e_i^{\mathrm{in/out}}$ is an optimal effort in/out of the contract, and $q_i \in [0, 1]$ is an equilibrium probability with which agent $i$ enters the contract, such that the following conditions hold:

For each agent $i$, define the expected average effort of all other agents induced by the mixed participation profile $q$ as $\overline{\mathbf e}_{-i}^{\,q} := \frac{1}{N-1} \sum_{j \neq i} \big(q_j \cdot e_j^{\mathrm{in}} + (1 - q_j) \cdot e_j^{\mathrm{out}}\big)$. Then
\begin{enumerate}
    \item 
    \begin{align*}
    &e_i^{\mathrm{out}} \in \argmax_{e_i\ge 0} U_i^{\mathrm{out}}(e_i;\overline{\mathbf e}_{-i}^{\,q}), \qquad e_i^{\mathrm{in}} \in \argmax_{e_i\ge 0} U_i^{\mathrm{in}}(e_i;\overline{\mathbf e}_{-i}^{\,q}).
    \end{align*}
    \item
    \begin{alignat*}{2}
    \text{if } & 0 < q_i < 1, \qquad & U_i^{\mathrm{in}}(e_i^{\mathrm{in}}; \overline{\mathbf e}_{-i}^{\,q}) = U_i^{\mathrm{out}}(e_i^{\mathrm{out}}; \overline{\mathbf e}_{-i}^{\,q}),\\[1.5ex]
    \text{if } & q_i = 1, \qquad & U_i^{\mathrm{in}}(e_i^{\mathrm{in}}; \overline{\mathbf e}_{-i}^{\,q}) \geq U_i^{\mathrm{out}}(e_i^{\mathrm{out}}; \overline{\mathbf e}_{-i}^{\,q}),\\[1.5ex]
    \text{if } & q_i = 0, \qquad & U_i^{\mathrm{in}}(e_i^{\mathrm{in}}; \overline{\mathbf e}_{-i}^{\,q}) \leq U_i^{\mathrm{out}}(e_i^{\mathrm{out}}; \overline{\mathbf e}_{-i}^{\,q}).
\end{alignat*}
\end{enumerate}
\end{definition}
The first condition states that, given the average equilibrium effort of the other agents, agent $i$ chooses optimal efforts $e_i^{\mathrm{out}}$ and $e_i^{\mathrm{in}}$ conditioning on staying outside and entering the contract. The second condition describes the participation decision: $q_i = 0$ and $q_i = 1$ correspond to the cases in which staying outside or entering the contract is preferred, while $q_i \in (0, 1)$ is the mixed case, which follows from the indifference principle.

Although the $N$-agent model is conceptually natural, its equilibrium analysis (both theoretical and computational) becomes difficult in the presence of heterogeneity. Individual effort and participation decisions depend on the average effort of others, which in turn depends on the optimal responses of all agents. This mutual dependence, together with endogenous participation, makes the equilibrium mapping hard to characterize. Motivated by these challenges, we study a heterogeneous {\it mean-field} version of the model.

\section{Lower-Level Mean-Field Model of Agents} \label{sec:MFG}

We consider a mean-field game formulation for the agents. For this, we start with a formal setup of continuum of agents indexed by {\it type} $\theta = (\gamma, c, \sigma)$. Each agent chooses an {\it effort} $e(\theta) \in [0, \infty)$. Throughout, we write $\gamma(\theta)$, $c(\theta)$, and $\sigma(\theta)$ for the corresponding components of $\theta$, which denote as before, respectively, the \textit{risk-aversion parameter}, the \textit{cost of effort}, and the \textit{pre-screening noise}. We assume that $\theta$ has distribution $\pi$ in set $\Theta$. The agent's loss $L(\theta)$ is modeled as a Gaussian random variable whose mean and variance depend on both the agent's own effort and the average effort in the population:
\[
L(\theta) \sim \mathcal{N}\left( \mu(e(\theta) + x \tilde{m}),\; \lambda(e(\theta) + x \tilde{m}) \right),
\]
where $\tilde{m} = \int_{\Theta} e(\theta) \, d\pi(\theta)$ is the mean-field effort, and $x \in [0,1)$ is an {\it interdependence parameter} that captures the degree of interdependence among agents.

Through pre-screening, the insurer obtains a pre-screening assessment of each agent's effort, $S(e(\theta)) = e(\theta) + W(\theta)$, where $W(\theta)$ is a zero mean Gaussian noise with variance $\sigma(\theta)^2$.

As in the finite-population model, each agent faces a choice between remaining outside the contract or opting into it. In the following subsections, we set up the mean-field Nash equilibrium (formally defined in Section \ref{sec:3.3}, with existence and uniqueness proved in Sections \ref{sec:3.4} and \ref{sec:3.5}, respectively). To formulate this equilibrium, we use a fixed-point approach: we fix the aggregate mean-field effort $\tilde{m}$, determine the agents' optimal individual responses conditional on $\tilde{m}$, and then require that these individual responses aggregate back to $\tilde{m}$. Therefore, in the immediate analysis of the inside and outside options, we treat the mean-field effort $\tilde{m}$ as a fixed parameter. The last part of this section (Section \ref{sec:3.6}) is dedicated to numerical study.

\subsection{Outside Contract}

Each agent may choose whether to accept an insurance contract $(p, \alpha, \beta)$, consisting of a {\it base premium} $p$, a {\it screening-based discount coefficient} $\alpha$, and a {\it coverage ratio} $\beta \in [0,1]$. Given the mean-field effort $\tilde m$, the expected utility function of a risk-averse agent (with {\it risk-aversion parameter} $\gamma(\theta) > 0$) outside the contract is
\begin{align*}
U^{\mathrm{out}}(e(\theta);\theta, \tilde{m})
&= \mathbb{E}\big[ -\exp\{-\gamma(\theta)\,(-L(e(\theta), \tilde{m})-c(\theta)\,e(\theta))\}\big] \\
&= -\exp\Big\{\gamma(\theta)\Big(
\mu(e(\theta)+ x \tilde{m}) \;+\; c(\theta)\,e(\theta) \;+\; \tfrac{1}{2}\,\gamma(\theta)\,\lambda(e(\theta)+ x \tilde{m})
\Big)\Big\}.
\end{align*}
Define the certainty equivalent as ${\rm CE}^{\mathrm{out}}(e;\theta,\tilde m)
:= \mu(e+x\tilde m)+c(\theta)e+\tfrac12\gamma(\theta)\lambda(e+x\tilde m)$. Since expected utility is a strictly decreasing transformation of this certainty-equivalent cost, maximizing expected utility is equivalent to minimizing ${\rm CE}^{\mathrm{out}}$. The optimal response set is therefore defined by
\[
E^{\mathrm{out}}(\theta;\tilde m) := \argmin_{e\geq 0} {\rm CE}^{\mathrm{out}}(e;\theta,\tilde m).
\]

\subsection{Inside Contract}

The contract offered by the insurer is parameterized by $(p, \alpha, \beta)$, where $p$ denotes the base premium, $\alpha$ is the screening-based discount coefficient applied to the pre-screening signal $S(e(\theta))$, and $\beta \in [0,1]$ is the coverage ratio. Given the mean-field effort $\tilde m$, the expected utility of a risk-averse agent who accepts the contract is:
\begin{align*}
U^{\mathrm{in}}(e(\theta);\theta,\tilde{m})
&= \mathbb{E}\!\left[ - \exp\!\left( -\gamma(\theta)\,\big\{ -p + \alpha S(e(\theta)) - L(e(\theta),\tilde{m}) + \beta L(e(\theta),\tilde{m}) - c(\theta) e(\theta) \big\} \right) \right] \\
&= - \exp\!\Big( \gamma(\theta)\Big[
p + (c(\theta) - \alpha)\, e(\theta)
+ \tfrac{1}{2}\,\gamma(\theta)\,\alpha^2\,\sigma^2(\theta) \\
&\qquad\qquad\qquad\quad
+ (1 - \beta)\,\mu\!\big(e(\theta) + x \tilde{m}\big)
+ \tfrac{1}{2}\,\gamma(\theta)\,(1 - \beta)^2\,\lambda\!\big(e(\theta) + x \tilde{m}\big)
\Big] \Big),
\end{align*}

Define the certainty equivalent as the expression in the brackets above with a generic choice variable $e$.
\[
{\rm CE}^{\mathrm{in}}(e; \theta, \tilde{m}):=
p+(c(\theta)-\alpha)e+\tfrac12\,\gamma(\theta)\alpha^2\sigma^2(\theta)
+(1-\beta)\mu(e+x\tilde{m})+\tfrac12\,\gamma(\theta)(1-\beta)^2\lambda(e+x\tilde{m}).
\]
Because the certainty equivalent perfectly ranks the agent's preferences, maximizing expected utility is equivalent to minimizing this deterministic cost.
The set of optimal effort responses inside the contract is then defined by
\[
E^{\mathrm{in}}(\theta;\tilde m) := \argmin_{e\geq 0} {\rm CE}^{\mathrm{in}}(e;\theta,\tilde m).
\]

\subsection{Lower-Level Mean-Field Nash Equilibrium}\label{sec:3.3}
With the optimal responses for both the outside and inside options established, we now formalize the agent's endogenous participation decision and the resulting mean-field equilibrium. Because the population is heterogeneous, agents will self-select into or out of the contract based on their individual risk and cost types. An agent evaluates the contract by comparing the minimized certainty equivalents under each option. For a given mean-field effort $\tilde m$, let $e^{\mathrm{in}}(\theta; \tilde m)\in E^{\mathrm{in}}(\theta;\tilde m)$ and $e^{\mathrm{out}}(\theta; \tilde m)\in E^{\mathrm{out}}(\theta;\tilde m)$ denote selected optimal efforts inside and outside the contract. Just as in the $N$-player model, an agent of type $\theta$ participates in the contract only if the \textit{individual rationality} (IR) condition holds, and their corresponding effort choice $e^{\mathrm{in}}(\theta; \tilde{m})$ satisfies the \textit{incentive compatibility} (IC) condition.

\textbf{IR:}
\[
U^{\mathrm{in}}(e^{\mathrm{in}}(\theta; \tilde{m});\theta, \tilde{m}) \geq U^{\mathrm{out}}(e^{\mathrm{out}}(\theta;\tilde m);\theta,\tilde m)=\max_{e\geq 0}U^{\mathrm{out}}(e;\theta,\tilde m).
\]

Since the exponential function is monotone, the IR constraint can be equivalently written in terms of the certainty equivalents as follows:
\begin{align*}
&p + (c(\theta) - \alpha)\, e^{\mathrm{in}}(\theta; \tilde{m})
+ \tfrac{1}{2}\,\gamma(\theta)\,\alpha^2\,\sigma^2(\theta) + (1 - \beta)\,\mu\!\big(e^{\mathrm{in}}(\theta; \tilde{m}) + x \tilde{m}\big)\\
&\quad+ \tfrac{1}{2}\,\gamma(\theta)\,(1 - \beta)^2\,\lambda\!\big(e^{\mathrm{in}}(\theta; \tilde{m}) + x \tilde{m}\big)\\
&\quad \leq \mu(e^{\mathrm{out}}(\theta;\tilde m)+ x \tilde{m}) \;+\; c(\theta)\,e^{\mathrm{out}}(\theta;\tilde m) \;+\; \tfrac{1}{2}\,\gamma(\theta)\,\lambda(e^{\mathrm{out}}(\theta;\tilde m)+ x \tilde{m}).
\end{align*}

The IC condition is

\textbf{IC:}
\[
e^{\mathrm{in}}(\theta; \tilde{m})\ \in\ \argmin_{e\ge 0}\ {\rm CE}^{\mathrm{in}}(e; \theta, \tilde{m}).
\]

For a fixed mean-field effort $\tilde m$, an agent first decides its optimal effort inside the contract, which must satisfy the IC condition minimize ${\rm CE}^{\mathrm{in}}$, and its optimal effort outside the contract, which must minimize ${\rm CE}^{\mathrm{out}}$. It then compares the two optimized values through the IR condition and decide whether to enter the contract or remain outside.
When the two optimized values coincide, the agent may mix between participation and non-participation.

A mean-field Nash equilibrium may involve both participating and non-participating agents. Given the equilibrium mean-field $\tilde m$, for each type $\theta$, the agent chooses optimal effort conditional on accepting or rejecting the contract. The equilibrium also includes a measurable participation probability $q(\theta; \tilde m) \in [0, 1]$, which is the probability that type $\theta$ accepts the contract given the mean-field $\tilde m$, while $1 - q(\theta; \tilde m)$ is the probability of staying out.

\begin{definition} \label{defn: equilibrium}
Given a contract $(p, \alpha, \beta)$ and a population distribution $\pi$ over types $\theta = (\gamma, c, \sigma)$, a tuple $(\tilde{m}, \{e^{\mathrm{in}}(\theta; \tilde{m}), e^{\mathrm{out}}(\theta;\tilde m), q(\theta; \tilde{m})\}_{\theta \in \Theta})$ forms a {\rm mean-field Nash equilibrium} if the following conditions hold:
\begin{enumerate}
    \item[(1)] For every type $\theta \in \Theta$, the best response functions satisfy
    \[
    e^{\mathrm{out}}(\theta;\tilde m) \in E^{\mathrm{out}}(\theta;\tilde m), \qquad e^{\mathrm{in}}(\theta; \tilde{m}) \in E^{\mathrm{in}}(\theta;\tilde m).
    \]
    
    \item[(2)] The participation probability $q(\theta; \tilde m)$ is measurable and satisfies
    \begin{equation} \label{eq: participation probability}
    q(\theta; \tilde{m}) = \begin{cases}
        1 & \text{ if }\;\; U^{\mathrm{in}}(e^{\mathrm{in}}(\theta; \tilde{m}); \theta, \tilde{m}) > U^{\mathrm{out}}(e^{\mathrm{out}}(\theta;\tilde m); \theta, \tilde{m}),\\
        0 & \text{ if }\;\; U^{\mathrm{in}}(e^{\mathrm{in}}(\theta; \tilde{m}); \theta, \tilde{m}) < U^{\mathrm{out}}(e^{\mathrm{out}}(\theta;\tilde m); \theta, \tilde{m}),\\
        \text{any value in } [0, 1] & \text{ if }\;\; U^{\mathrm{in}}(e^{\mathrm{in}}(\theta; \tilde{m}); \theta, \tilde{m}) = U^{\mathrm{out}}(e^{\mathrm{out}}(\theta;\tilde m); \theta, \tilde{m}).
    \end{cases}
    \end{equation}

    \item[(3)] Given these individual best responses, the effort contributions among inside and outside agents are
    \begin{align*}
        &m^{\mathrm{in}}(\tilde{m}) = \int_{\Theta} q(\theta; \tilde{m}) e^{\mathrm{in}}(\theta; \tilde{m}) d\pi(\theta), \\
        &m^{\mathrm{out}}(\tilde{m}) = \int_{\Theta} (1-q(\theta; \tilde{m})) e^{\mathrm{out}}(\theta;\tilde m) d\pi(\theta).
    \end{align*}
    The equilibrium mean effort satisfies the fixed-point condition
    \[
    \tilde{m} = m^{\mathrm{in}}(\tilde{m}) + m^{\mathrm{out}}(\tilde{m}),
    \]
    or equivalently,
    \[
    \tilde{m} = \int_{\Theta} \big[q(\theta; \tilde{m}) e^{\mathrm{in}}(\theta; \tilde{m}) + (1-q(\theta; \tilde{m})) e^{\mathrm{out}}(\theta;\tilde m)\big] d\pi(\theta).
    \]
\end{enumerate}
\end{definition}

\color{black}

\subsection{Existence of a Mean-Field Nash Equilibrium}\label{sec:3.4}

To establish the existence of equilibrium, we formulate the mean-field equilibrium condition as a fixed-point problem in the aggregate effort $\tilde m$.  The proof relies on the Kakutani fixed-point theorem stated below. Under suitable compactness, convexity, and measurability conditions, we construct a correspondence mapping aggregate effort into the set of feasible population-average best responses, and verify that it satisfies the assumptions of the fixed-point theorem.

\begin{lemma}[Kakutani Fixed-Point Theorem]
Let $K$ be a non-empty, compact and convex subset of $\mathbb R$. 
Let $\Gamma : K \to 2^K$ be a set-valued function with closed graph and nonempty convex values. 
Then there exists a point $y$ in $K$ such that $y \in \Gamma(y)$.
\end{lemma}

\begin{assumption} \label{assum_NE}
    \begin{enumerate}
        \item[(A1)] There is a maximal effort $E_{\max}<\infty$; hence, for all $\theta\in\Theta$, $e(\theta), \tilde m\in[0,E_{\max}]$.
        \item[(A2)] The functions $\mu(z), \lambda(z)$ are continuous and convex on $[0, (1+x)E_{\max}]$.
        \item[(A3)] $\gamma(\theta), c(\theta), \sigma(\theta)$ are measurable and uniformly bounded. Namely, there are $\gamma_{\max},c_{\max},\sigma_{\max}\in(0,\infty)$, such that,
        \[
        0 < \gamma(\theta) \le \gamma_{\max}, \ 0 \leq c(\theta) \le c_{\max}, \ 0 \leq \sigma(\theta) \le \sigma_{\max} , \qquad \forall \theta \in \Theta.
        \]
    \end{enumerate}
\end{assumption}

Under Assumption (A1), the feasible effort set is restricted to $[0,E_{\max}]$. Thus, in the existence and uniqueness analysis below, we use the constrained optimal-effort correspondences
\begin{equation} \label{eq: optimal effort correspondences}
E^{\mathrm{out}}(\theta;\tilde m)
:=
\argmin_{e\in[0,E_{\max}]}
{\rm CE}^{\mathrm{out}}(e;\theta,\tilde m),
\qquad
E^{\mathrm{in}}(\theta;\tilde m)
:=
\argmin_{e\in[0,E_{\max}]}
{\rm CE}^{\mathrm{in}}(e;\theta,\tilde m).
\end{equation}

\begin{remark}
    The assumption above ensures the existence of equilibrium from the mathematical perspective and has economic meaning.  Assumption (A1) ensures the compactness of the decision space, which is necessary for applying the Kakutani fixed-point theorem. It reflects finite resource constraints in economic sense, such as a physical limit or a budget ceiling for security investments. Assumption (A2) assumes the continuity and convexity of $\mu$ and $\lambda$, which ensures the upper hemicontinuity and convexity of best-response correspondences, which are essential for the Berge's Maximum Theorem. It also represents the diminishing marginal returns of security effort. Economically, one may also impose $\mu'(z) < 0$ and $\lambda'(z) \le 0$ when derivatives exist, reflecting the notion that prevention effort reduces expected loss and risk exposure. These monotonicity conditions are not needed for the existence proof. Assumption (A3) ensures that individual best responses are measurable selections, making the population-level aggregate effort well-defined. In reality, it accounts for the diversity of risk preferences and costs in the market while ensuring the population behavior remains statistically tractable.
\end{remark}

\begin{theorem} \label{thm_NE}
    Given a contract $(p, \alpha, \beta)$ and a population distribution $\pi$ over types $\theta = (\gamma, c, \sigma)$, under \Cref{assum_NE}, there exists a mean-field Nash equilibrium 
    \[
    (\tilde{m}, \{e^{\mathrm{in}}(\theta; \tilde{m}), e^{\mathrm{out}}(\theta;\tilde m), q(\theta; \tilde m)\}_{\theta \in \Theta}).
    \]
\end{theorem}

\begin{proof}[Sketch of the proof.] The proof relies on a fixed-point argument. For each candidate mean field $\tilde m \in [0,E_{\max}]$, we characterize the individual inside and outside best responses. A primary technical challenge arises when aggregating these heterogeneous responses: taking the population average requires evaluating an Aumann integral of a correspondence. To ensure this integral is mathematically well-defined, we rigorously establish the weak measurability of the agents' piecewise best-response correspondences and deploy measurable selection theorems (e.g., the Kuratowski--Ryll-Nardzewski theorem) to construct the set-valued aggregate mean-field map. We then verify that this aggregate correspondence has nonempty, convex values and a closed graph, allowing us to apply the Kakutani fixed-point theorem to obtain a fixed point $\tilde{m}$. Finally, we reverse-engineer this aggregated expected effort to reconstruct explicit, measurable individual effort choices and a mixed participation probability. This yields a complete mean-field Nash equilibrium. The full proof is deferred to \Cref{sec: proof of existence}.
\end{proof}

\color{black}

\subsection{Uniqueness of Equilibrium}\label{sec:3.5}
Kakutani fixed-point theorem guarantees existence of a fixed point for a set-valued best-response correspondence, but it does not imply uniqueness. In general, multiple fixed points may arise because the individual inside and outside best responses may be set-valued, and because the participation decision may remain non-unique on the indifference set.

In this section, we impose additional regularity and non-degeneracy assumptions under which the aggregate mean-field map becomes single-valued and strictly contractive on $[0,E_{\max}]$. The unique mean-field equilibrium then follows from the Banach fixed point theorem. The conditions in the assumption below are stronger than the conditions in \Cref{assum_NE} and require more compactness, regularity, and convexity of the functions involved.

\begin{assumption}[Assumptions for Uniqueness]\label{assum_unique}
\leavevmode
\begin{enumerate}
\item[(U1)] There is a maximal effort $E_{\max}<\infty$; hence, for all $\theta\in\Theta$, $e(\theta), \tilde m\in[0,E_{\max}]$.

\item[(U2)] The functions $\mu,\lambda$ are continuous and strictly convex on $[0, (1+x) E_{\max}]$. Moreover, they are Lipschitz continuous on this interval; namely, there exist $L_\mu, L_\lambda>0$, such that:
\[
|\mu(z_1)-\mu(z_2)|\le L_\mu|z_1-z_2|,\qquad
|\lambda(z_1)-\lambda(z_2)|\le L_\lambda|z_1-z_2|.
\]

\item[(U3)] There exist constants $0<\gamma_{\min}<\gamma_{\max}<\infty$, $0<c_{\max}<\infty,$ and $0<\sigma_{\min}<\sigma_{\max}<\infty$ such that
\[
\Theta=[\gamma_{\min},\gamma_{\max}]
\times[0,c_{\max}]
\times[\sigma_{\min},\sigma_{\max}].
\]

\item[(U4)] The screening-based discount coefficient and the coverage ratio, respectively,  satisfy: $\alpha>0$ and $\beta\in[0,1)$.

\item[(U5)] The type distribution $\pi$ admits a density $f$ with respect to Lebesgue measure on
$\Theta$
and\footnote{The global uniform bound on the joint density function, $\|f\|_{\infty} \le C_f$, is mathematically stronger than strictly necessary to guarantee uniqueness. Because the contraction property relies exclusively on the mass of marginal agents shifting their participation decisions, this assumption could be relaxed to a bound on the conditional density of the idiosyncratic shock $\sigma$. Formally, it suffices to define the marginalized supremum $f^{\sup}(\gamma, c) := \sup_{\sigma} f(\sigma \mid \gamma, c) f(\gamma, c)$ and assume only that this function is finite and integrable over the parameter space. We maintain the stronger global uniform bound in our primary assumptions because it significantly streamlines the subsequent proof and is not overly restrictive for standard distributional choices.} $\|f\|_{\infty}\le C_f$.

\item[(U6)] Define
\begin{align*}
    &c_\sigma:=\gamma_{\min}\alpha^2\sigma_{\min},\\
    &C_m:=x\Big(2c_{\max}+\alpha+(2-\beta)L_\mu +\tfrac12\gamma_{\max}\big((1-\beta)^2+1\big)L_\lambda\Big),\\
    &C_q:=\frac{C_f}{c_\sigma} (\gamma_{\max}-\gamma_{\min})\,c_{\max}\,C_m.
\end{align*}
Assume
\[
3x+E_{\max}C_q<1.
\]
\end{enumerate}
\end{assumption}

\begin{remark}
    Assumptions (U1)--(U6) strengthen (A1)-(A3) to obtain uniqueness rather than only existence. The additional conditions ensure single-valued best responses and a contraction property of the mean-field map. 
    
    Assumption (U2) replaces convexity by strict convexity and adds Lipschitz continuity of $\mu$ and $\lambda$. Strict convexity guarantees that individual optimization problems admit unique minimizers, so the best-response correspondences become singletons. The Lipschitz property is used to control how optimal efforts vary with the mean-field. 
    
    Assumptions (U3)--(U5) are used to control the participation decision. In particular, they imply that the indifference set $\{\Delta {\rm CE} = 0\}$ is $\pi$-null. Hence, although the mixed participation rule allows $q(\theta; \tilde m) \in [0, 1]$ on this set, this freedom is irrelevant up to $\pi$-a.e.~equality and does not affect the aggregate equilibrium objects.
    
    Finally, (U6) ensures that the Lipschitz constant of the aggregate equilibrium map is strictly smaller than 1, which is essential for the Banach fixed point theorem. Note that both $C_m$ and hence $C_q$ are proportional to $x$, which is the interdependence parameter. Therefore, the condition $3x+E_{\max}C_q<1$ is effectively a small-interaction requirement: it imposes that the interdependence parameter $x$ be sufficiently small relative to the model primitives so that the overall Lipschitz constant of the mean-field map remains strictly below one.
\end{remark}

\begin{theorem}\label{thm_NEuniq}
    Fix a contract $(p, \alpha, \beta)$. Under \Cref{assum_unique}, the individual optimal effort levels $e^{\mathrm{in}}(\theta; \tilde{m})$ and $e^{\mathrm{out}}(\theta;\tilde m)$ are uniquely defined for each type $\theta$ and each $\tilde m\in[0,E_{\max}]$. 
    
    Moreover, define the mean-field aggregation function $\Phi: [0, E_{\max}] \to [0, E_{\max}]$ as
    \begin{equation} \label{eq: Phi}
        \Phi(\tilde{m}) = \int_{\Theta} \Big[ q(\theta; \tilde m) e^{\mathrm{in}}(\theta; \tilde{m}) + (1-q(\theta; \tilde m)) e^{\mathrm{out}}(\theta; \tilde{m}) \big] d\pi(\theta).
    \end{equation}
    Then $\Phi$ is well defined and Lipschitz continuous on $[0,E_{\max}]$ with Lipschitz constant $<1$. In particular, $\Phi$ is a contraction mapping on the complete metric space $[0,E_{\max}]$. Consequently, there exists a unique fixed point $\tilde m^* \in [0,E_{\max}]$ such that
    \[
    \tilde m^* = \Phi(\tilde m^*).
    \]
    The associated mixed-participation mean-field Nash equilibrium is unique up to $\pi$-a.e.~equality. Hence all aggregate equilibrium quantities are unique.
\end{theorem}

\begin{proof}[Sketch of the proof.] Under \Cref{assum_unique}, the inside and outside optimal effort choices are uniquely defined for each type $\theta$ and each $\tilde m \in [0,E_{\max}]$. The first step is to show that these optimal effort mappings are Lipschitz continuous in $\tilde m$. The second step is to prove that for each fixed $\tilde m$, the participation probability $q(\theta; \tilde m)$ is unique for $\pi$-a.e.~$\theta$. The third step is to establish an $L^1(\pi)$-Lipschitz bound for $q(\cdot; \tilde m)$. Combining these estimates yields that the aggregate mean-field map $\Phi$ is Lipschitz continuous on $[0,E_{\max}]$, with Lipschitz constant strictly smaller than $1$ by assumption (U6). Therefore, $\Phi$ is a contraction on $[0,E_{\max}]$, and the Banach fixed-point theorem implies the existence and uniqueness of a fixed point $\tilde m^*$. The associated mixed-participation equilibrium is therefore unique up to $\pi$-a.e.~equality. The full proof is deferred to \Cref{sec: proof of uniqueness}.
\end{proof}

\subsection{Numerics}
\label{sec:3.6}

In this subsection, we illustrate the lower-level mean-field equilibrium under a fixed contract and examine how the equilibrium responds to changes in key parameters. Throughout the numerical experiments below, the parameters are not chosen to satisfy the sufficient uniqueness condition $3x+E_{\max}C_q<1$. Instead, for each experiment, we numerically check the fixed-point residual $\Phi(\tilde m) - \tilde m$ over the compact feasible interval for $\tilde m$ and count the distinct numerical solutions detected by this residual check. In all reported experiments below, the check return a single fixed point, and we report the equilibrium outcome associated with that solution.

We first characterize the equilibrium fixed point, then study comparative statics of participation rate, and finally examine threshold and risk-aversion heterogeneity patterns that help motivate the later menu-design analysis.

In all numerical experiments in this subsection, we use the fixed contract $p=4.0$, $\alpha=0.2$, and $\beta=0.6$. Unless otherwise stated, the interdependence parameter is set to $x=1.0$. The theoretical effort bound $E_{\max}$ is chosen sufficiently large in every experiment so that it does not bind in the reported computations. Because different figures are designed to illustrate different aspects of the lower-level equilibrium, we report the distributional assumptions for $(\gamma,c,\sigma)$ separately for each experiment.

\subsubsection{Fixed-Point Characterization of the Equilibrium}
We first analyze the agents' mean-field Nash equilibrium under a fixed contract $(p,\alpha,\beta)$. For the fixed-point illustration, we use an illustrative calibration with $c=2.0$, $\sigma=1.0$, and risk-aversion values drawn from a Gaussian distribution centered around $1.0$.

For each parameter configuration, the equilibrium mean effort $\tilde m^*$ is determined by the fixed-point condition $\tilde m = \Phi(\tilde m)$, where $\Phi$ aggregates individual best responses. 
We verify that the fixed-point equation admits a solution, and illustrate the equilibrium as the intersection between $\Phi(\tilde m)$ and $y = \tilde{m}$ in Figure~\ref{fig:fixed_point_mapping}.
\begin{figure}[htbp]
\centering
\includegraphics[width=0.6\textwidth]{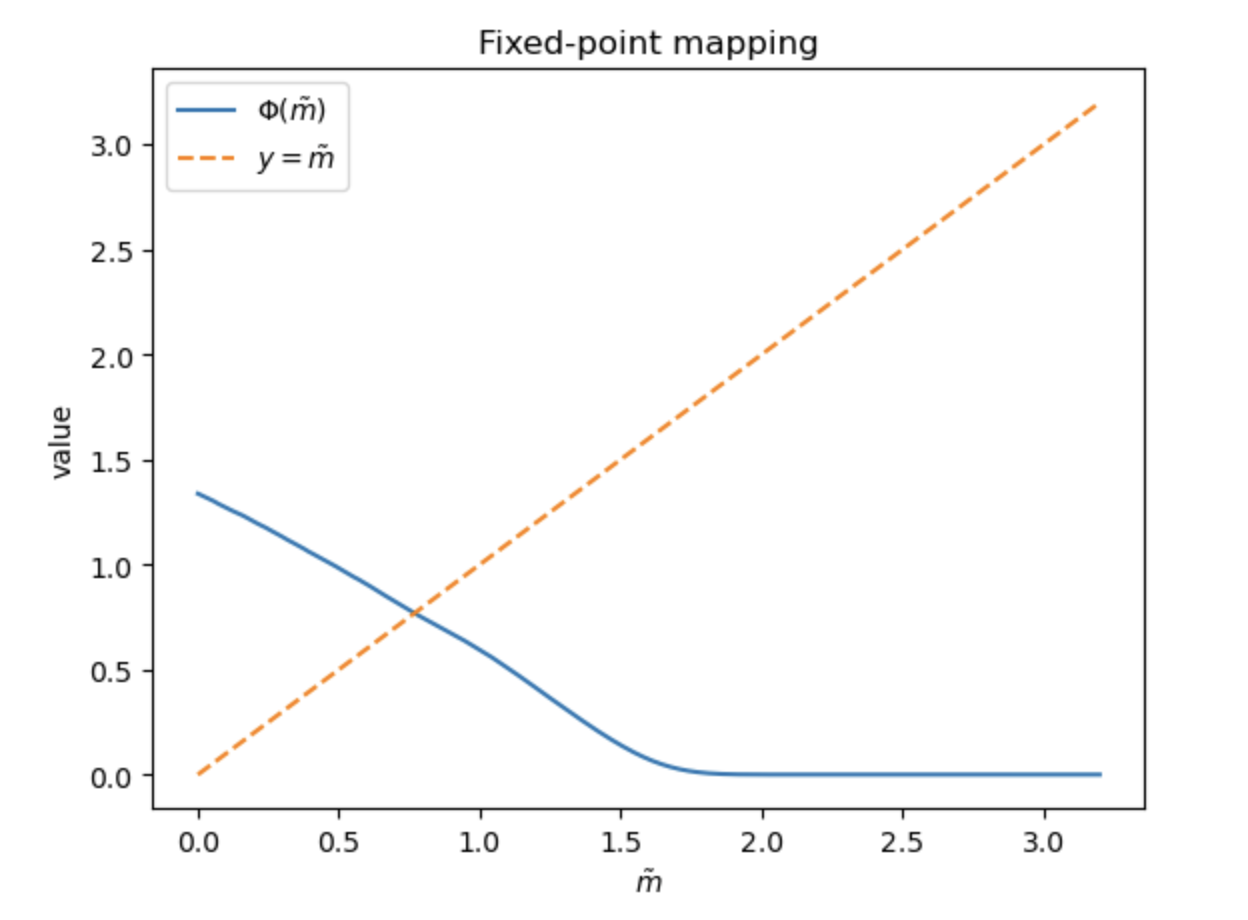}
\caption{Fixed-point Mapping.}
\label{fig:fixed_point_mapping}
\end{figure}

\subsubsection{Comparative Statics of the Participation Rate}
We next study how the equilibrium participation rate responds to model parameters. For the unique lower-level equilibrium computed in the numerical experiments, let $\tilde m^*$ denote the equilibrium mean effort and let $q(\theta;\tilde m^*)$ denote the associated participation probability. We define the equilibrium participation rate by
\[
\rho^* := \int_\Theta q(\theta;\tilde m^*)\,d\pi(\theta),
\]
which endogenously depends on both individual type distributions and the mean-field feedback. 
We next study how $\rho^*$ responds to key parameters, including the pre-screening noise variance $\sigma^2$, that is, the variance of the pre-screening noise, the average risk aversion $\bar{\gamma}$, the average effort cost $\bar{c}$, and the interdependence parameter $x$.

In this experiment, the baseline population has $N=5000$, with $\gamma=\max{10^{-4},\mathcal N(3.0,0.35^2)}$, $c=\max{0,\mathcal N(3.0,0.35^2)}$, and $\sigma=1.0$. 
\begin{figure}[htbp]
\centering
\includegraphics[width=0.9\textwidth]{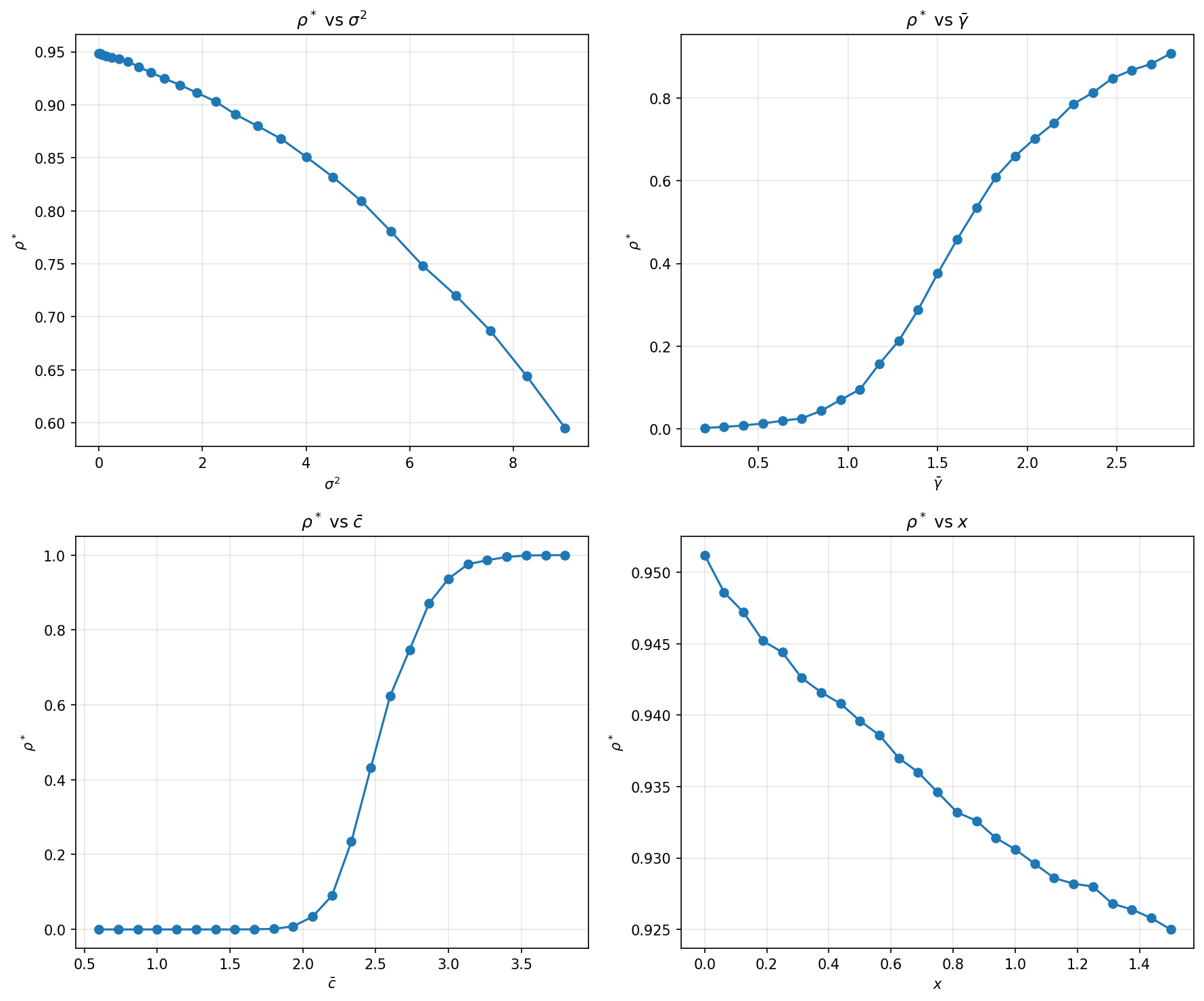}
\caption{Equilibrium participation rate $\rho^*$ as a function of key model parameters:
top-left: pre-screening noise variance $\sigma^2$; 
top-right: mean risk aversion $\bar{\gamma}$; 
bottom-left: mean effort cost $\bar{c}$; 
bottom-right: interdependence parameter $x$.}
\label{fig:rho_comparative_statics}
\end{figure}

Figure~\ref{fig:rho_comparative_statics} illustrates how the equilibrium participation rate $\rho^*$ responds to economically meaningful parameters. 

Participation decreases in the noisy risk variance $\sigma^2$, reflecting the intuitive effect that higher outcome uncertainty makes the contract less attractive when agents bear residual risk.

Participation increases in the average risk aversion $\bar{\gamma}$ and in the average effort cost $\bar{c}$. A higher average level of risk aversion implies that agents value insurance more strongly, making the contract relatively more appealing. 
Similarly, when effort becomes more costly on average, the outside option deteriorates, increasing the relative benefit of contractual protection.

Finally, participation decreases in the interdependence parameter $x$, but the magnitude of the effect is comparatively small. Economically, a larger $x$ means that each agent's loss exposure depends more strongly on the aggregate prevention effort in the population. This creates a spillover benefit that is also available to agents who stay outside the contract. Hence, for marginal types, the outside option becomes relatively more attractive, which reduces the incentive to participate. The magnitude of this effect is small because $x$ affects participation indirectly through the equilibrium mean-field adjustment $\tilde m^*$, rather than through the premium or effort-incentive terms directly.

\subsubsection{Participation Threshold in Pre-Screening Noise}

To further understand the participation threshold in equilibrium, we characterize the critical pre-screening noise level that makes an agent of type $(\gamma,c)$ indifferent between participating in the contract and remaining outside.

\begin{remark}[Threshold Structure in Pre-Screening Noise]
Recall that an agent of type $\theta=(\gamma,c,\sigma)$ participates in the contract if and only if
\[
\Delta {\rm CE}(\theta;\tilde m)
:=
{\rm CE}^{\mathrm{in}}(e^{\mathrm{in}}(\theta; \tilde{m});\theta,\tilde m)
-
{\rm CE}^{\mathrm{out}}(e^{\mathrm{out}}(\theta; \tilde{m});\theta,\tilde m)
\le 0.
\]
For fixed $(\gamma,c,\tilde m)$, the quantity $\Delta {\rm CE}(\theta;\tilde m)$ is increasing in $\sigma$, since the only $\sigma$-dependent term is $\tfrac12 \gamma \alpha^2 \sigma^2$ in ${\rm CE}^{\mathrm{in}}$. Therefore, whenever the participation region is nonempty, there exists a critical threshold $\sigma^*(\gamma,c;\tilde m)$ such that
\[
\sigma \le \sigma^*(\gamma,c;\tilde m)
\quad\Longrightarrow\quad
\Delta {\rm CE}(\theta;\tilde m)\le 0,
\]
and
\[
\sigma > \sigma^*(\gamma,c;\tilde m)
\quad\Longrightarrow\quad
\Delta {\rm CE}(\theta;\tilde m)>0.
\]
Thus, conditional on $(\gamma,c,\tilde m)$, participation is monotone in the pre-screening noise level. By contrast, no comparable monotonicity in $\gamma$ or $c$ is immediate, because the optimal effort choices depend endogenously on these parameters.
\end{remark}
For the numerical illustration below, if the indifference equation has no solution in the feasible range of $\sigma$, we define the threshold by the relevant boundary value. In particular, when $\Delta{\rm CE}(0;\gamma,c,\tilde m)>0$, the agent does not participate even under perfect screening, and we set $\sigma^*(\gamma,c;\tilde m)=0$.

For the numerical illustration below, we evaluate this threshold at the equilibrium mean effort $\tilde m^*$ and define $\sigma^*(\gamma, c)$ through the indifference condition
\[
\Delta{\rm CE}(\sigma^*(\gamma, c); \gamma, c, \tilde{m}^*) = 0.
\]
Figure~\ref{fig:sigmastar} illustrates the resulting surface $\sigma^*(\gamma,c)$.
\begin{figure}[htbp]
\centering
\includegraphics[width=0.75\textwidth]{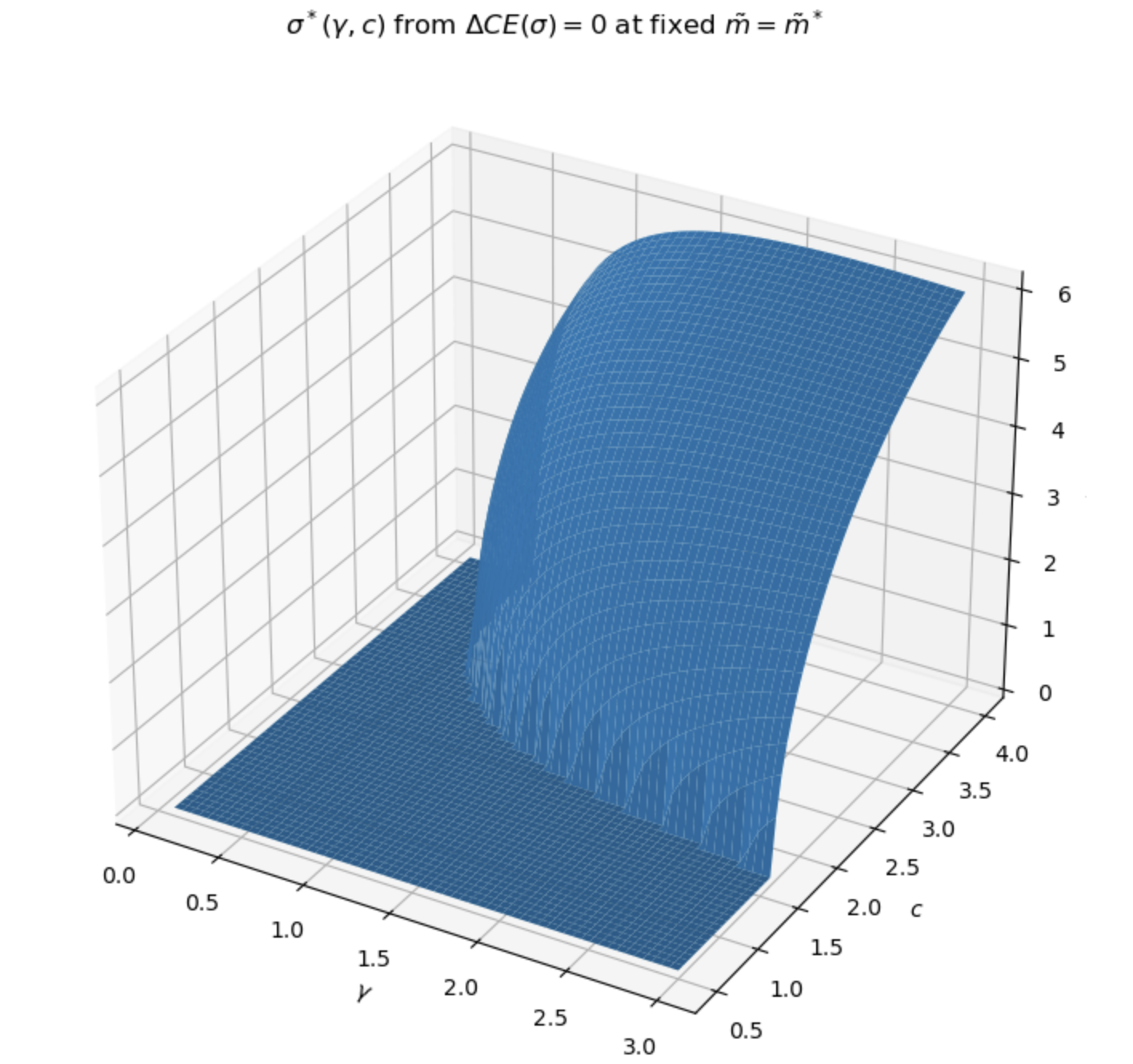}
\caption{Critical pre-screening noise level $\sigma^*(\gamma,c)$ solving 
$\Delta {\rm CE}(\sigma)=0$ at fixed equilibrium mean effort $\tilde m^*$.}
\label{fig:sigmastar}
\end{figure}
The surface $\sigma^*(\gamma,c)$ represents the maximum level of pre-screening noise under which agents of type $(\gamma,c)$ are still willing to participate in the contract. 
Agents with higher risk aversion $\gamma$ tolerate larger noise levels, as they value insurance more strongly. 
Similarly, agents with higher effort cost $c$ also exhibit greater willingness to participate, since their outside option becomes less attractive. 
For low values of $(\gamma,c)$, the threshold collapses to zero, indicating that such agents would not participate even under perfect screening. 

\subsubsection{Risk-Aversion Heterogeneity and Participation}

We now turn to the role of heterogeneity in risk aversion and examine how different population compositions affect equilibrium participation. The following examples compare several constructions of low- and high-risk-aversion populations in order to isolate how shifts in the distribution of $\gamma$ affect the equilibrium outcome.

Figure~\ref{fig:rhovsgamma1} compares equilibrium participation rates under three different constructions of low- and high-risk-aversion populations. For all three constructions below, we fix $N=5000$, $c=2.2$, and $\sigma=2.0$, and vary only the construction of the risk-aversion distribution:

\begin{itemize}
    \item \textbf{Homogeneous low/high risk baseline:} We consider two separate economies: one where all agents have a strictly low risk aversion $\gamma_L = 0.4$, and another where all have a high risk aversion $\gamma_H = 2.4$. Comparing the resulting equilibrium participation rates $\rho^*(\gamma_L)$ and $\rho^*(\gamma_H)$ isolates the pure effect of risk aversion on contract demand.
    
    \item \textbf{Threshold split within one distribution:} Starting from a single underlying distribution of $\gamma$ drawn from a truncated normal distribution $\mathcal N(1.5, 0.25)$ on $[\gamma_{\min},\gamma_{\max}]=[10^{-4},3.0]$, we partition the population at its empirical mean $\bar{\gamma}$. We then solve the mean-field equilibrium independently for the lower support $[\gamma_{\min},\bar{\gamma}]$ and the upper support $(\bar{\gamma},\gamma_{\max}]$. This approach preserves within-group heterogeneity while cleanly separating the population by risk-aversion strata.
    
    \item \textbf{Different distributions on fixed bounds:} We construct two distinct $\gamma$-distributions that share the same support interval $[\gamma_{\min},\gamma_{\max}]=[10^{-4},3.0]$ but possess different means, $0.6$ and $2.2$. By solving the equilibrium for each distribution separately, we isolate how shifts in the probability mass itself affect $\rho^*$.
\end{itemize}

Across all three constructions, higher average risk aversion leads to higher equilibrium participation, consistent with the insurance contract reducing residual risk exposure.

\begin{figure}[h!]
    \centering
    \includegraphics[width=0.95\textwidth]{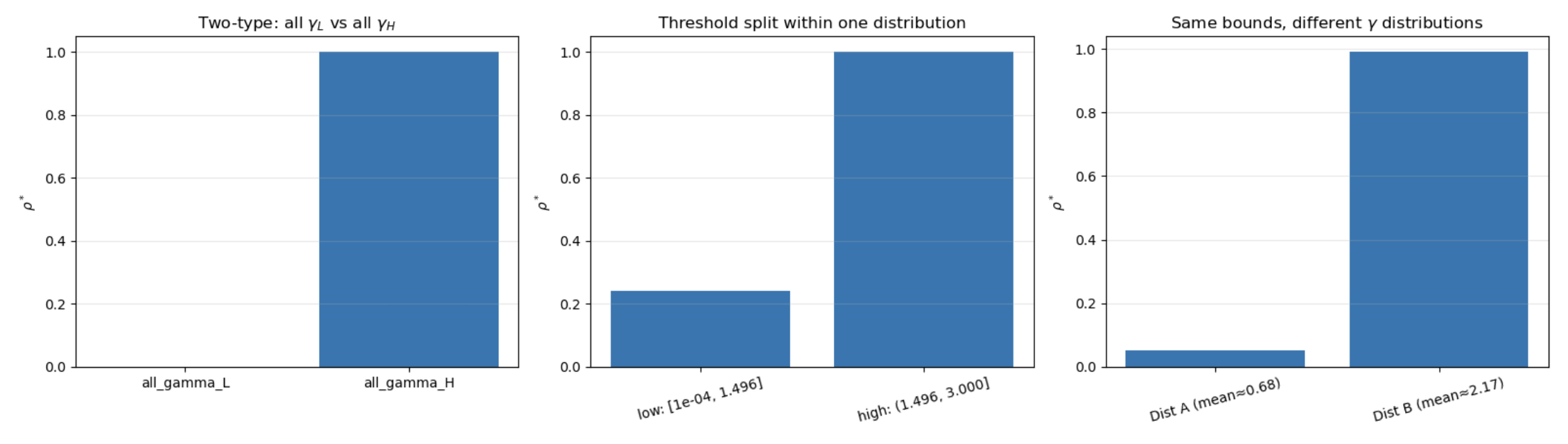}
    \caption{Equilibrium participation rates under three constructions of low and high $\gamma$ populations.}
    \label{fig:rhovsgamma1}
\end{figure}

Figure~\ref{fig:rhovsgamma2} extends the threshold-split construction. In the previous comparison, the low- and high-$\gamma$ subpopulations were treated as separate economies, and equilibria were computed independently. Here, both groups coexist in a single mean-field environment and jointly determine the fixed point $\tilde m^*$. Participation decisions are still computed at the individual level, but the aggregate effort entering the mean-field is shared.

Here we again fix $N=5000$, $c=2.2$, and $\sigma=2.0$, while drawing $\gamma$ from a truncated distribution on $10^{-4},5.0$ with mean $1.5$ and standard deviation $0.9$.

The three bars report, respectively, the total equilibrium participation rate $\rho_{\mathrm{tot}}$, the participation rate of the low-$\gamma$ group $\rho_{g1}$, and the participation rate of the high-$\gamma$ group $\rho_{g2}$, all evaluated at the common equilibrium $\tilde m^*$. The figure illustrates that while the overall equilibrium is jointly determined, participation behavior remains sharply heterogeneous across risk-aversion groups.

\begin{figure}[htbp]
    \centering
    \includegraphics[width=0.75\textwidth]{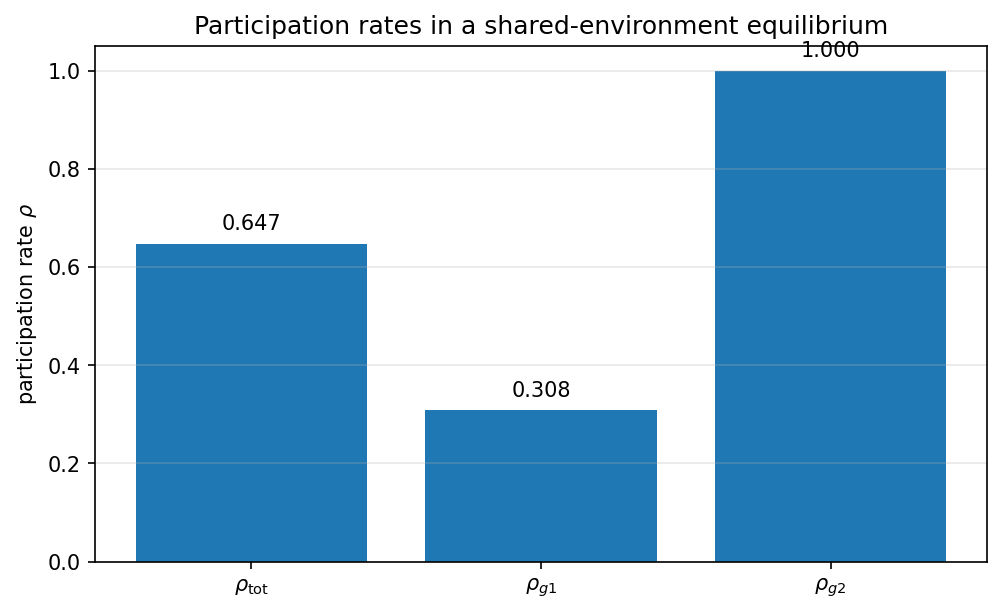}
    \caption{Participation rates in a shared-environment equilibrium with low- and high-$\gamma$ groups interacting through the mean-field.}
    \label{fig:rhovsgamma2}
\end{figure}

Across all constructions, the high-$\gamma$ group exhibits full participation. Under alternative parameter values, the low-$\gamma$ group may have zero participation while the high-$\gamma$ group may only have partial participation. These patterns suggest that a contract menu can be used to screen by risk aversion. For example, the insurer may offer a more expensive, higher-coverage contract that is attractive to high-$\gamma$ agents, and a cheaper, lower-coverage contract targeted at low-$\gamma$ agents, allowing self-selection across types.

\section{Upper-level Insurer's Contract-Design Problem} \label{sec:insurer}

We next formalize the insurer's problem as the upper-level problem in a Stackelberg game. The insurer acts as the leader by choosing contract $(p, \alpha, \beta)$ that maximizes his payoff, anticipating that the population of agents responds through a mean-field Nash equilibrium as characterized in \Cref{sec:MFG}. Therefore, the insurer does not optimize against arbitrary effort and participation, but against the equilibrium outcome induced by the chosen contract.

By designing the contract terms $(p, \alpha, \beta)$, the insurer does not merely price the risk; the principal actively manages the network's vulnerability. The coverage levels act as a direct lever for security investment. By adjusting the coverage ratio $\beta$ or the effort incentive term governed by $\alpha$, the insurer dictates the agents' `skin in the game.' When the interdependence parameter $x$  is high, the optimal contract must restrict coverage to force agents to internalize the negative externalities of their under-investment in security, thereby using the contract to elevate the baseline effort of the entire network.

For a given contract $k = (p, \alpha, \beta)$, let $\mathcal M(k)$ denote the set of mean-field Nash equilibria induced by $k$. Each equilibrium $M\in\mathcal M(k)$ is represented by
\[
M=\Bigl(\tilde m, \{e^{\mathrm{in}}(\theta),e^{\mathrm{out}}(\theta), q(\theta)\}\Bigr),
\]
where $\tilde m$ is the equilibrium mean effort, $e^{\mathrm{in}}(\theta)$ and $e^{\mathrm{out}}(\theta)$ are the equilibrium effort levels under participation and non-participation, respectively, and $q(\theta)$ is the participation probability.

For a type $\theta$ that accepts the contract and chooses effort $e^{\mathrm{in}}(\theta)$, the insurer's random payoff is
\[
p - \alpha S(e^{\mathrm{in}}(\theta); \theta) - \beta L(e^{\mathrm{in}}(\theta); \tilde{m}),
\]
where $S(e; \theta) = e + W(\theta)$, $W(\theta) \sim \mathcal{N}(0, \sigma^2(\theta))$, and $L(e; \tilde{m}) \sim \mathcal{N}\left( \mu(e + x \tilde{m}),\; \lambda(e + x \tilde{m}) \right)$. The insurer's expected profit is then
\[
p - \alpha e^{\mathrm{in}}(\theta) - \beta \mu(e^{\mathrm{in}}(\theta) + x \tilde{m}).
\]

Given a mean-field equilibrium $M = (\tilde m,\{e^{\mathrm{in}}(\theta),e^{\mathrm{out}}(\theta), q(\theta)\}) \in \mathcal{M}(k)$, the insurer's per-capita expected profit is
\begin{align*}
    J(k; M) = \int_{\Theta} \Big(p - \alpha e^{\mathrm{in}}(\theta)- \beta \mu(e^{\mathrm{in}}(\theta) + x \tilde{m})\Big) q(\theta) d\pi(\theta).
\end{align*}
\color{black}
When the lower-level mean-field Nash equilibrium is not unique, a given contract $k$ may induce multiple equilibrium payoffs through different equilibria $M\in\mathcal M(k)$. Thus the insurer’s payoff is not automatically single-valued as a function of the contract alone. Before formulating the upper-level optimization problem, we first specify the admissible contract set and recall that each admissible contract induces at least one lower-level equilibrium.

\begin{assumption} \label{assum_stack}
\leavevmode
\begin{enumerate}
    \item[(I1)] The insurer chooses the contract from a compact set
    \[
    \mathcal K:=[0,p_{\max}]\times[0,\alpha_{\max}]\times[0,\beta_{\max}],
    \]
    where $0 \leq p_{\max} < \infty$, $0 \leq  \alpha_{\max} < \infty$ and $0 \leq \beta_{\max} < 1$.

    \item[(I2)] \Cref{assum_NE} holds for every admissible contract $k = (p, \alpha, \beta) \in \mathcal K$. Consequently, the lower-level mean-field game among agents admits at least one Nash equilibrium. Denote by $\mathcal{M}(k)$ the set of all such equilibria.
\end{enumerate}
\end{assumption}

Under $\Cref{assum_stack}$, the equilibrium payoff set is nonempty for every admissible contract. To turn this payoff set into a unique choice for the insurer, we introduce a performance criterion, denoted by $\mathcal J(k)$. This criterion represents how the insurer evaluates a contract when multiple lower-level equilibria are possible.

\begin{definition}[insurer performance criterion] \label{defn: perform criterion}
    For each admissible contract $k\in\mathcal K$, let $\mathcal M(k)$ denote the set of lower-level mean-field Nash equilibria induced by $k$. An insurer performance criterion is a measurable function $\mathcal J: \mathcal K \to \mathbb R$ that assigns to each contract $k$ a single payoff value used by the insurer to evaluate that contract, possibly by aggregating or selecting from the equilibrium payoff set
    \[
    \mathbb J(k) := \{J(k; M): M \in \mathcal M(k)\}.
    \]
    We say that $\mathcal J$ is well defined if $\mathcal J(k)$ is a finite real number for every $k\in\mathcal K$.
\end{definition}
Under a chosen criterion $\mathcal J$, the insurer’s upper-level problem is to choose an admissible contract so as to maximize $\mathcal J(k)$ over $k\in\mathcal K$. In general, however, without stronger compactness and continuity assumptions on the equilibrium correspondence, this maximization problem need not admit an exact optimizer. We therefore formulate the general upper-level existence result in terms of $\varepsilon$-optimal contracts under the chosen criterion.

The following remark explains why the exact Stackelberg existence argument is delicate in the present model and why the $\varepsilon$-optimal formulation is useful.
\begin{remark}
    Although \Cref{assum_NE} guarantees that the agents' lower-level mean-field game admits at least one equilibrium for each fixed contract, this pointwise existence result is not by itself sufficient to establish the existence of an exact Stackelberg equilibrium for the insurer's upper-level problem. To obtain an exact optimizer, one would need additional compactness and closedness properties of the equilibrium correspondence across contracts, together with enough continuity of the insurer's payoff along equilibrium objects. These properties are delicate in the present model because the lower-level equilibrium is constructed through a relaxed participation and expected-effort correspondence.

    The main difficulty comes from the indifference set $\{\Delta {\rm CE}=0\}$. In the proof of \Cref{thm_NE}, the fixed point is formulated in terms of a feasible expected effort selection
    \[
        e(\theta)\in \mathcal E(\theta;\tilde m),
        \qquad
        \tilde m=\int_\Theta e(\theta)\,d\pi(\theta),
    \]
    where, on the indifference set, $\mathcal E(\theta;\tilde m)$ is the convex hull generated by inside and outside best responses. Each equilibrium therefore corresponds to some decomposition of this expected effort into $q(\theta)e^{\mathrm{in}}(\theta)+(1-q(\theta))e^{\mathrm{out}}(\theta)$. However, when $\Delta {\rm CE}=0$, the mixed participation rule does not pin down the value of $q(\theta)$, and different admissible choices may lead to different participation rates, mean efforts, and insurer payoffs. Moreover, as the contract varies, the selections of $e(\theta)$ on the indifference set need not vary continuously, and this lack of continuity can affect both the induced participation rule $q$ and the aggregate mean-field $\tilde m$.

    For this reason, instead of imposing strong global regularity assumptions on the full equilibrium correspondence in order to prove the existence of an exact Stackelberg equilibrium, we work with an upper-level $\varepsilon$-equilibrium notion. The lower-level agents' equilibrium remains exact for each fixed contract, while the insurer's contract is required to be $\varepsilon$-optimal with respect to the relevant equilibrium-performance criterion. This formulation avoids requiring a continuous selection of equilibrium objects across contracts.
\end{remark}

\begin{theorem} \label{thm: insurer}
    Under \Cref{assum_stack}, let $\mathcal J: \mathcal K \to \mathbb R$ be an insurer performance criterion. Suppose that $\mathcal J$ is well defined and bounded above on $\mathcal K$, so that
    \[
    \sup_{k \in \mathcal K} \mathcal J(k) < \infty.
    \]
    Then, for every $\varepsilon>0$, there exists an admissible contract $k_\varepsilon\in\mathcal K$ such that
    \[
    \mathcal J(k_{\varepsilon}) \geq \sup_{k \in \mathcal K} \mathcal J(k) - \varepsilon.
    \]
    We call such a contract an {\rm $\varepsilon$-optimal insurer contract} under the performance criterion $\mathcal J$.

    The approximation is only at the insurer’s upper-level optimization stage; the lower-level response remains an exact mean-field Nash equilibrium.
\end{theorem}

\begin{proof} Since $\mathcal K$ is nonempty and $\mathcal J$ is well-defined and bounded above, $\sup_{k\in\mathcal K}\mathcal J(k)<\infty$.
By the definition of the supremum, for every $\varepsilon > 0$, there exists at least one admissible contract $k_{\varepsilon} \in \mathcal K$ such that
\[
\mathcal J(k_{\varepsilon}) \geq \sup_{k \in \mathcal K} \mathcal J(k) - \varepsilon.
\]
Therefore, $k_\varepsilon$ is an $\varepsilon$-optimal insurer contract under the performance criterion $\mathcal J$. 
\end{proof}
The $\varepsilon$-approximation occurs only at the insurer’s upper-level optimization stage. For each chosen contract $k_\varepsilon$, the relevant lower-level responses are still exact mean-field Nash equilibria in $\mathcal M(k_\varepsilon)$. Hence, this result establishes the existence of an upper-level $\varepsilon$-optimal contract, not an approximate lower-level equilibrium.

\subsection{Examples of Insurer Performance Criteria}

The theorem above is stated for a general insurer performance criterion $\mathcal J$. We now record two consequences of this formulation. First, when the lower-level equilibrium is not unique, the insurer can evaluate a contract through worst- and best-equilibrium payoff envelopes, which summarize conservative and optimistic performance under equilibrium multiplicity. Second, under the additional uniqueness assumptions of $\Cref{assum_unique}$, the equilibrium response becomes single-valued, and the upper-level problem admits an exact Stackelberg equilibrium rather than only an $\varepsilon$-optimal contract.

\subsubsection{Worst- and Best-Equilibrium Performance} \label{sec: worst best}

When the lower-level equilibrium is not unique, a contract $k$ does not generally induce a single insurer payoff. Instead, it induces the payoff set
$
\mathbb J(k)
$. 
Two natural ways to evaluate this set are the worst-equilibrium and best-equilibrium payoff envelopes:
\[
\mathcal J^{\mathrm{worst}}(k) := \inf_{M \in \mathcal M(k)} J(k; M), \qquad \mathcal J^{\mathrm{best}}(k) := \sup_{M \in \mathcal M(k)} J(k; M).
\]
The worst-equilibrium criterion describes the insurer’s conservative performance under unfavorable equilibrium selection. It is useful when the insurer wants robustness against the possibility that agents coordinate on an equilibrium that is least favorable to the insurer. The best-equilibrium criterion describes the most favorable payoff compatible with the lower-level equilibrium set. It can be interpreted as an optimistic benchmark or as the attainable payoff when equilibrium selection is favorable to the insurer.

\begin{corollary} \label{cor: worst best}
    Under $\Cref{assum_stack}$, consider the worst- and best-equilibrium performance criteria defined by
    \[
    \mathcal J^{\mathrm{worst}}(k) := \inf_{M \in \mathcal M(k)} J(k; M), \qquad \mathcal J^{\mathrm{best}}(k) := \sup_{M \in \mathcal M(k)} J(k; M).
    \]
    Then both $\mathcal J^{\mathrm{worst}}$ and $\mathcal J^{\mathrm{best}}$ are well-defined insurer performance criteria defined in \Cref{defn: perform criterion} and are bounded above on $\mathcal K$. Consequently, for every $\varepsilon>0$, there exist admissible contracts $k_\varepsilon^{\mathrm{worst}},k_\varepsilon^{\mathrm{best}}\in\mathcal K$ such that
    \[
    \mathcal J^{\mathrm{worst}}(k_\varepsilon^{\mathrm{worst}}) \geq \sup_{k \in \mathcal K} \mathcal J^{\mathrm{worst}}(k) - \varepsilon, \qquad \mathcal J^{\mathrm{best}}(k_\varepsilon^{\mathrm{best}}) \geq \sup_{k \in \mathcal K} \mathcal J^{\mathrm{best}}(k) - \varepsilon.
    \]
    Thus the insurer admits $\varepsilon$-optimal contracts under both the conservative worst-equilibrium criterion and the optimistic best-equilibrium criterion.
\end{corollary}

\begin{proof}[Sketch of the proof]
The lower-level equilibrium set $\mathcal M(k)$ is nonempty for every admissible contract, so the payoff set $\mathbb J(k)$ is nonempty. The boundedness of admissible contracts, efforts, participation probabilities, and the continuity of $\mu$ on a compact domain imply that $J(k;M)$ is uniformly bounded over all $k \in \mathcal K$ and $M \in \mathcal M(k)$. Therefore, the worst- and best-equilibrium criteria are finite-valued and bounded above on $\mathcal K$. Applying $\Cref{thm: insurer}$ to each criterion gives the desired $\varepsilon$-optimal contracts. The full proof is deferred to \Cref{sec: proof of cor}.
\end{proof}

\subsubsection{Unique Equilibrium}

The general $\varepsilon$-optimal result above does not require uniqueness of the lower-level mean-field Nash equilibrium and therefore yields an $\varepsilon$-optimal insurer contract under any well-defined bounded-above performance criterion. We now show that, under the uniqueness assumptions in \Cref{assum_unique}, the insurer's upper level problem admits an exact Stackelberg equilibrium.

For the unique-equilibrium benchmark, we restrict attention to the smaller admissible set
\[
\mathcal K^{\mathrm{uniq}}
:=
[0,p_{\max}]\times[\alpha_{\min},\alpha_{\max}]\times[0,\beta_{\max}],
\qquad
0<\alpha_{\min}\leq \alpha_{\max}<\infty.
\]
This restriction keeps the effort-incentive coefficient bounded away from zero. It is needed for the threshold argument in the uniqueness proof, since the participation gap is strictly increasing in the pre-screening noise only when $\alpha>0$.

Under \Cref{assum_unique}, for any admissible contract $k \in \mathcal K^{\mathrm{uniq}}$, the lower-level mean-field game admits a unique equilibrium $M(k) \in \mathcal M(k)$ up to $\pi$-a.e.~equality. In particular, the equilibrium mean-field, the effort choices, and aggregate equilibrium quantities are unique, while the participation probability is uniquely determined for $\pi$-a.e.~$\theta$. In this case, the insurer's payoff is then single-valued as a function of contract and can be written as
\[
\mathcal J^{\mathrm{uniq}}(k) := J(k; M(k)).
\]
This unique equilibrium response varies continuously with the contract. Together with compactness of $\mathcal K^{\mathrm{uniq}}$, we can prove the existence of an exact optimal contract rather than an $\varepsilon$-optimal one.

\begin{proposition} \label{prop: uniq equilibrium}
    Assume \Cref{assum_unique} holds uniformly over $\mathcal K^{\mathrm{uniq}}$, in particular, denote $C_q(k)$ as the constant $C_q$ in $\Cref{assum_unique}$ evaluated at contract $k$, then
    \[
    \sup_{k\in\mathcal K^{\mathrm{uniq}}}
\bigl(3x+E_{\max}C_q(k)\bigr)<1.
    \]
    Then, for any admissible contract $k \in \mathcal K^{\mathrm{uniq}}$, the lower-level mean-field game admits a unique equilibrium up to $\pi$-a.e.~equality. For each $k \in \mathcal K^{\mathrm{uniq}}$, denote $M(k) \in \mathcal M(k)$ as this unique equilibrium object, choosing any representative of $q(\cdot; k)$ within its $\pi$-a.e.~equivalence class.

    Then the insurer's payoff induced by the unique lower-level equilibrium,
    \[
    \mathcal J^{\mathrm{uniq}}(k) := J(k; M(k)),
    \]
    is well-defined and continuous on $\mathcal K^{\mathrm{uniq}}$. Consequently, there exists an optimal contract $k^* \in \mathcal K^{\mathrm{uniq}}$ such that
    \[
    \mathcal J^{\mathrm{uniq}}(k^*) = \max_{k \in \mathcal K^{\mathrm{uniq}}} \mathcal J^{\mathrm{uniq}}(k).
    \]
    Thus the pair $(k^*, M(k^*))$ constitutes an exact Stackelberg equilibrium under the unique lower-level equilibrium response.
\end{proposition}

\begin{proof}[Sketch of the proof]
The proof proceeds in four steps. First, under $\Cref{assum_unique}$, the lower-level equilibrium induced by each admissible contract is unique up to $\pi$-a.e.~equality, so the payoff $\mathcal J^{\mathrm{uniq}}(k)=J(k;M(k))$ is well defined. Second, the unique equilibrium response depends continuously on $k$: the contraction characterization gives continuity of the equilibrium mean-field, Berge's Maximum Theorem gives continuity of the unique effort choices, and the threshold representation together with the bounded density of $\pi$ gives $L^1(\pi)$-continuity of $q$. Third, these continuity properties imply that $\mathcal J^{\mathrm{uniq}}$ is continuous on $\mathcal K^{\mathrm{uniq}}$. Finally, since $\mathcal K^{\mathrm{uniq}}$ is compact, Weierstrass yields an exact maximizer $k^*$, and the induced pair $(k^*,M(k^*))$ is an exact Stackelberg equilibrium. The full proof is deferred to \Cref{sec: proof of uniq equilibrium insurer}.
\end{proof}

\subsection{Numerics}

This subsection illustrates the insurer’s upper-level contract-design problem through two sets of numerical experiments. We first present a baseline upper-level optimization in which the uniqueness of the lower-level equilibrium is numerically checked for each candidate contract and the selected equilibrium is used to evaluate the insurer’s payoff. We then consider a parameter region in which some contracts induce multiple lower-level fixed points, and compare the worst- and best-equilibrium performance criteria introduced above.

\subsubsection{Baseline Upper-Level Optimization}

In the baseline experiment, each candidate contract is evaluated by solving the induced lower-level fixed-point equation and selecting the numerically detected equilibrium. In the reported parameter configurations, this procedure yields a single lower-level fixed point for each candidate contract. In the upper-level numerical optimization reported here, we identify a profit-maximizing contract from the finite contract grid for each parameter configuration. We therefore plot the resulting graphs of optimal premium $p^*$, effort incentive $\alpha^*$ and coverage level $\beta^*$ accordingly.

In the numerical experiments below, we use a calibration consistent with the fixed-point illustration in \Cref{sec:3.6}. Specifically, we set the interdependence parameter to $x=1.0$, the baseline effort cost to $c=2.0$, and the baseline pre-screening noise $\sigma=1.0$. The baseline population risk aversion is generated as $\gamma=\max{10^{-4},1.0+\varepsilon}$ with $\varepsilon\sim\mathcal N(0,0.5^2)$; in the comparative statics below, we vary the mean risk-aversion level when studying \Cref{fig:insurer_gamma}. For the upper-level optimization, we use a simulated population of $N=5000$ agents. For each candidate contract $(p,\alpha,\beta)$, we solve the induced lower-level equilibrium and evaluate the insurer’s expected profit at the selected equilibrium.

Given the numerically unique lower-level equilibrium selected for each contract, the insurer chooses the contract parameters $(p,\alpha,\beta)$ to maximize expected profit. To understand how optimal contract design responds to population characteristics, we examine in \Cref{fig:insurer_gamma} how the optimal premium $p^*$, effort incentive $\alpha^*$, and coverage level $\beta^*$ vary with the average risk aversion level $\bar{\gamma}$ of the agent distribution. In this experiment, risk aversion is generated as $\gamma=\max{10^{-4},\bar{\gamma}+\varepsilon}$ with $\varepsilon\sim\mathcal N(0,0.5^2)$, while $\bar{\gamma}$ is varied along the horizontal axis.

\begin{figure}[htbp]
\centering
\includegraphics[width=0.85\textwidth]{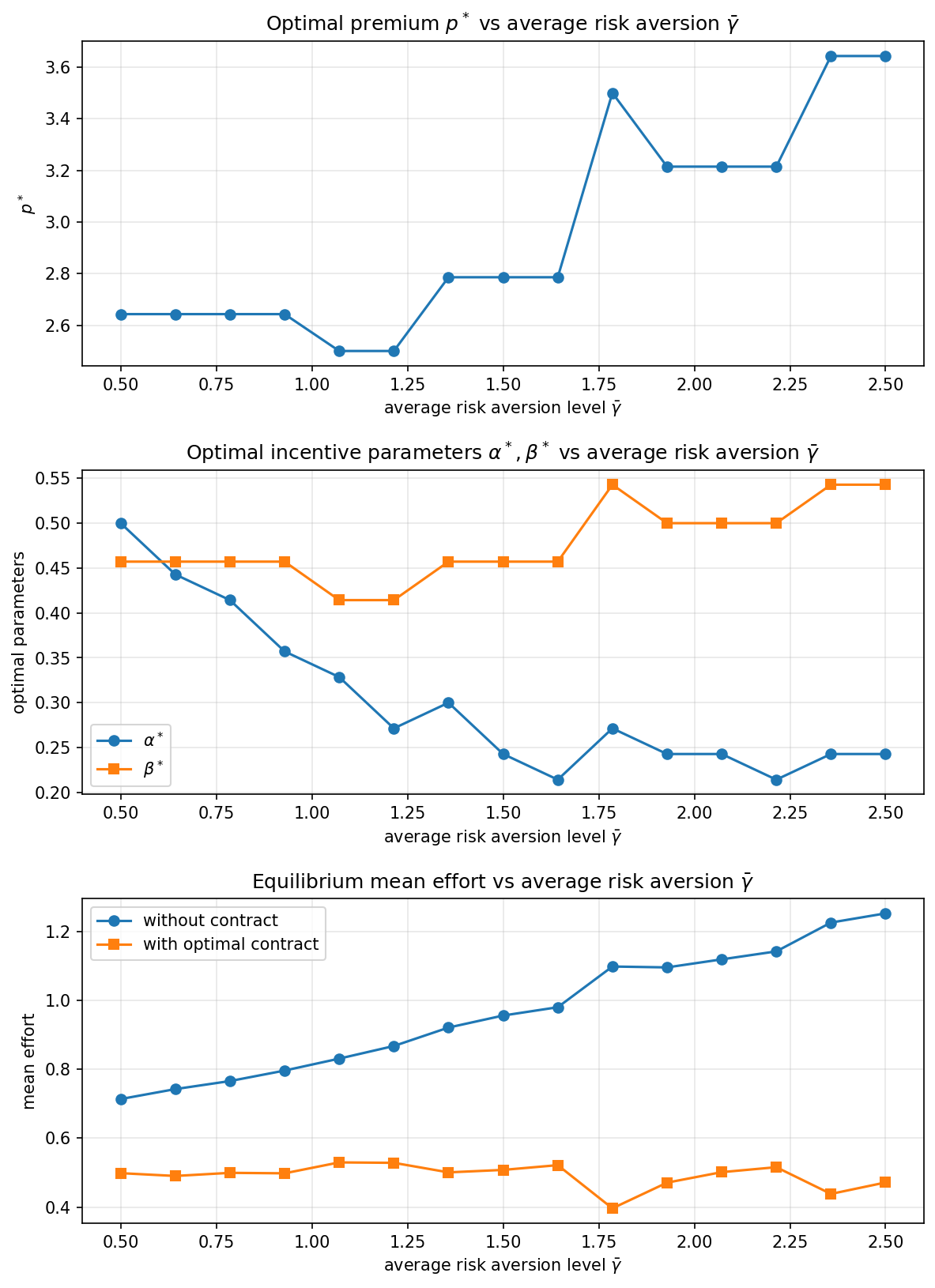}
\caption{Optimal contract and effort responses as functions of the average risk aversion level $\bar{\gamma}$.}
\label{fig:insurer_gamma}
\end{figure}

The top panel reports the optimal premium $p^*$.  As $\bar{\gamma}$ increases, the insurer is generally able to charge a higher premium. 
More risk-averse agents place a higher value on insurance, increasing their willingness to pay.  The upward shift in $p^*$ reflects the insurer’s ability to extract greater surplus when the population becomes more risk-averse.

The middle panel shows the optimal incentive parameters as $\bar{\gamma}$ varies. 
As average risk aversion increases, the insurer tends to increase the coverage level $\beta^*$, reflecting stronger demand for risk-sharing among more risk-averse agents. 
In contrast, the effort incentive $\alpha^*$ does not exhibit a monotone pattern and generally remains moderate, suggesting that as risk-sharing becomes more valuable, the insurer relies relatively more on insurance (coverage) than on effort-based incentives.

The bottom panel compares the equilibrium mean effort with and without the optimal contract. 
Without a contract, effort increases in $\bar{\gamma}$, as more risk-averse agents internalize risk more strongly. 
Under the optimal contract, however, effort remains significantly lower and relatively stable across $\bar{\gamma}$, indicating that insurance dampens precautionary effort incentives.

We next examine in \Cref{fig:insurer_sigma} how the insurer’s optimal outcomes respond to the level of pre-screening noise. 

\begin{figure}[htbp]
\centering
\includegraphics[width=0.85\textwidth]{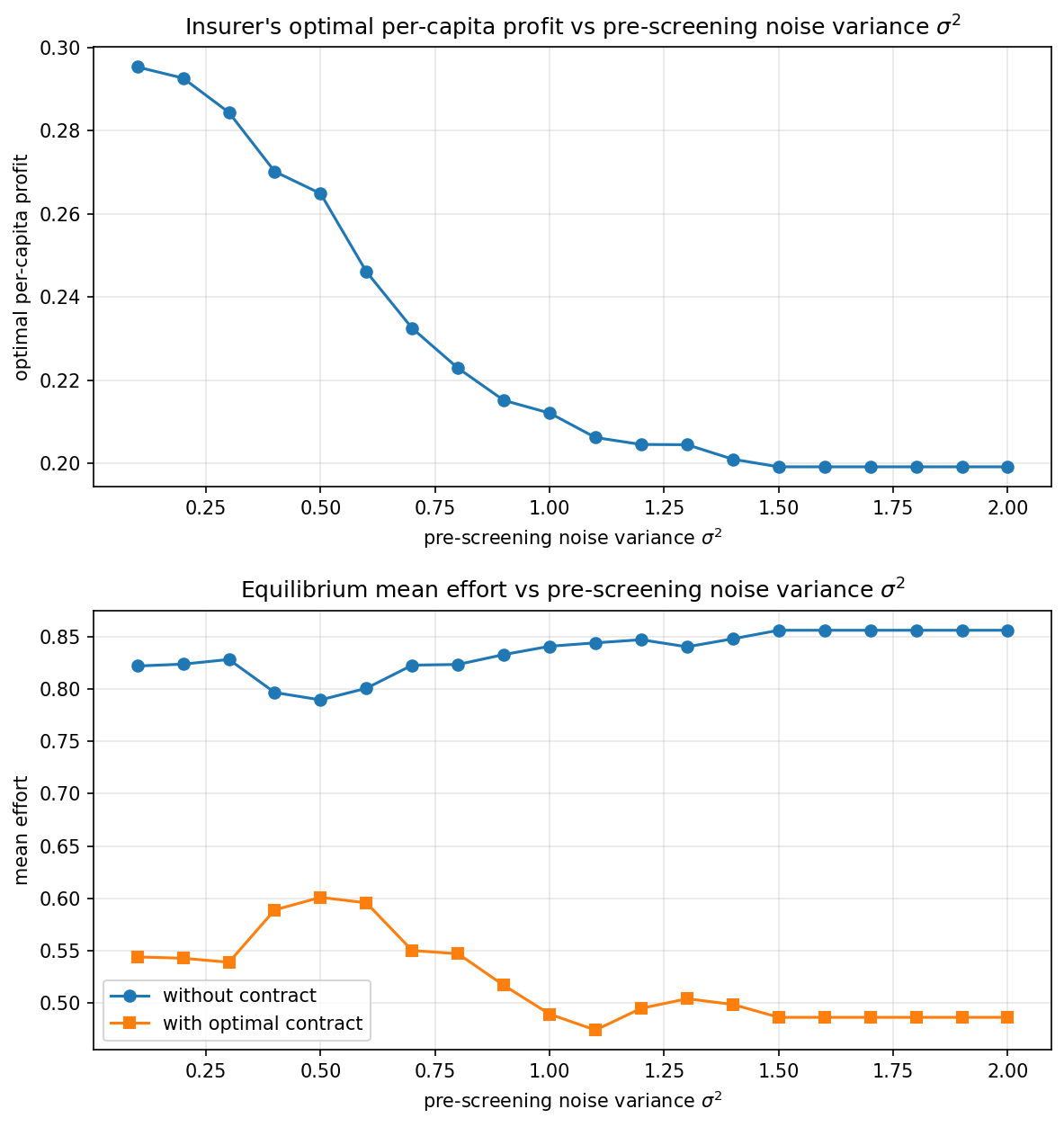}
\caption{Optimal profit and effort responses as functions of the pre-screening noise variance $\sigma^2$}
\label{fig:insurer_sigma}
\end{figure}

The top panel shows that the insurer’s optimal profit decreases monotonically in $\sigma^2$. Higher pre-screening noise weakens the insurer’s ability to design effective incentive and risk-sharing schemes, reducing surplus extraction. The bottom panel illustrates that as $\sigma^2$ increases, equilibrium effort under the optimal contract declines, while effort without a contract remains relatively stable.

\subsubsection{Numerical Examples of Worst- and Best-Equilibrium Criteria} \label{sec: numeric of worst and best}

We now provide a numerical illustration of the worst- and best-equilibrium performance criteria introduced in \Cref{sec: worst best}. The purpose of this experiment is to show how the two criteria behave in a region where the lower-level fixed-point equation admits multiple equilibria for some contracts. This also clarifies that the distinction between worst- and best-equilibrium evaluation is a contract-level distinction: for a fixed contract $k=(p,\alpha,\beta)$, different lower-level equilibria can imply different insurer payoffs.

For the numerical illustration, we consider a population calibration with $N=5000$, $c_i=2.0$, $x=2.0$, and $\sigma_i=0.75$. Risk aversion is generated by $\gamma_i=\max{10^{-4},1.2+\varepsilon_i}$, where the $\varepsilon_i$ are drawn once from $\mathcal N(0,0.5^2)$. We fix the coverage parameter at $\beta=0.75$, then examine the contract grid $p\in[4.8,4.88]$ and $\alpha\in[0.2,0.5]$. Economically, varying $p$ changes the price of participation, while varying $\alpha$ changes the effort-incentive intensity; we keep $\beta$ fixed to isolate these two channels, since varying coverage would jointly affect risk sharing and moral-hazard incentives.

The multiplicity pattern is shown in \Cref{fig:worst_best_multiplicity_map}. Each point in the $(p,\alpha)$ grid corresponds to a fixed contract, and the color reports the number of lower-level fixed points detected for that contract. A large portion of the displayed contract region admits more than one lower-level equilibrium, which makes the region suitable for illustrating the difference between worst- and best-equilibrium evaluation.
\begin{figure}[htbp]
    \centering
    \includegraphics[width=0.72\textwidth]{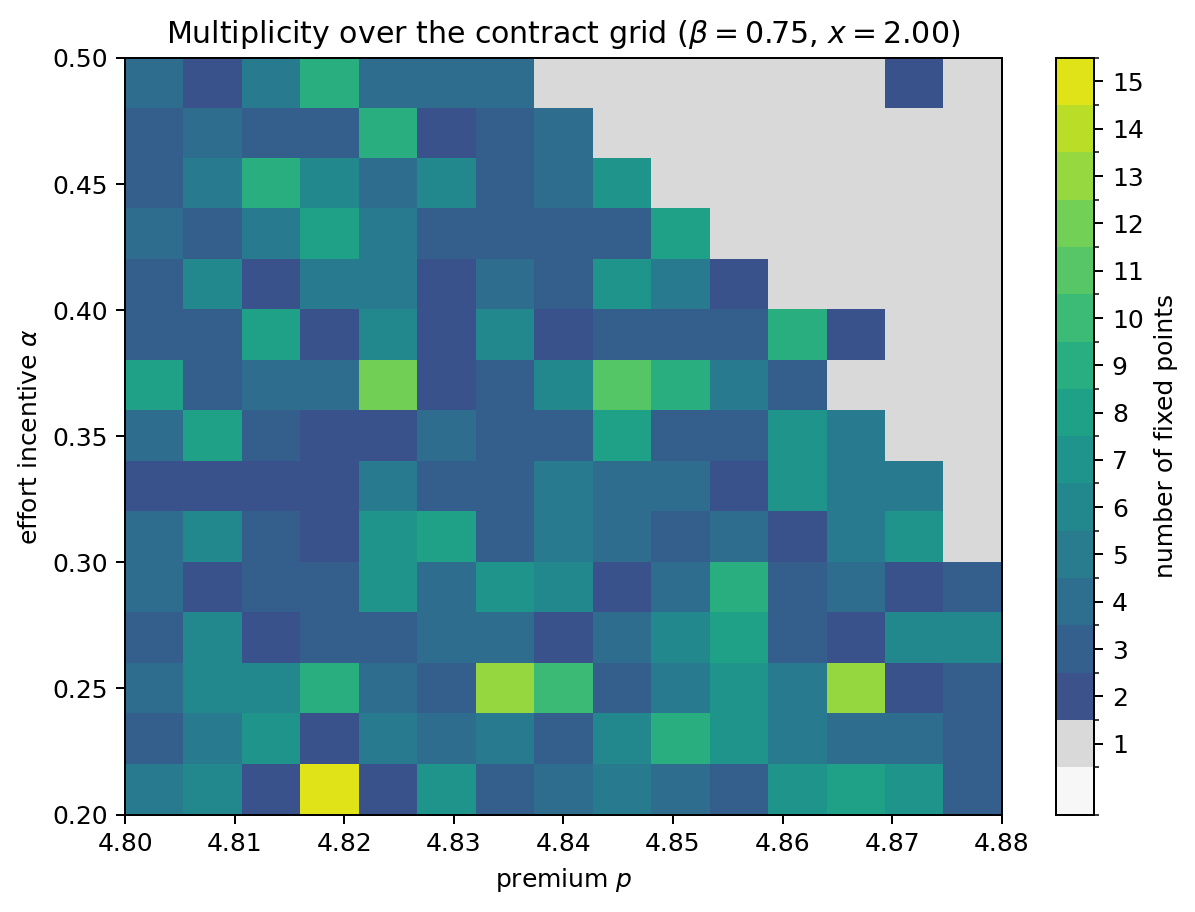}
    \caption{Multiplicity of lower-level fixed points over the contract grid.}
    \label{fig:worst_best_multiplicity_map}
\end{figure}
At the upper-right corner of \Cref{fig:worst_best_multiplicity_map}, the detected multiplicity should be interpreted with some numerical caution. The two fixed points there are very close to each other; since the algorithm counts numerical clusters of roots of $\Phi(\tilde m) - \tilde m$, near-flat residuals and threshold-induced small jumps in $\Phi$ may lead the root scan to report distinct detected clusters even when the associated equilibrium outcomes are nearly indistinguishable.

We then examine two one-dimensional slices through this region. \Cref{fig:worst_best_vary_p} fixes $\alpha=0.2$, $\beta=0.75$, and varies the premium $p$. Along this slice, both the worst- and best-equilibrium payoff criteria are evaluated for the same fixed contract at each value of $p$. The payoff gap remains substantial over the multiplicity region, showing that equilibrium selection can materially affect the insurer’s value from a given contract.
\begin{figure}[htbp]
    \centering
    \includegraphics[width=0.95\textwidth]{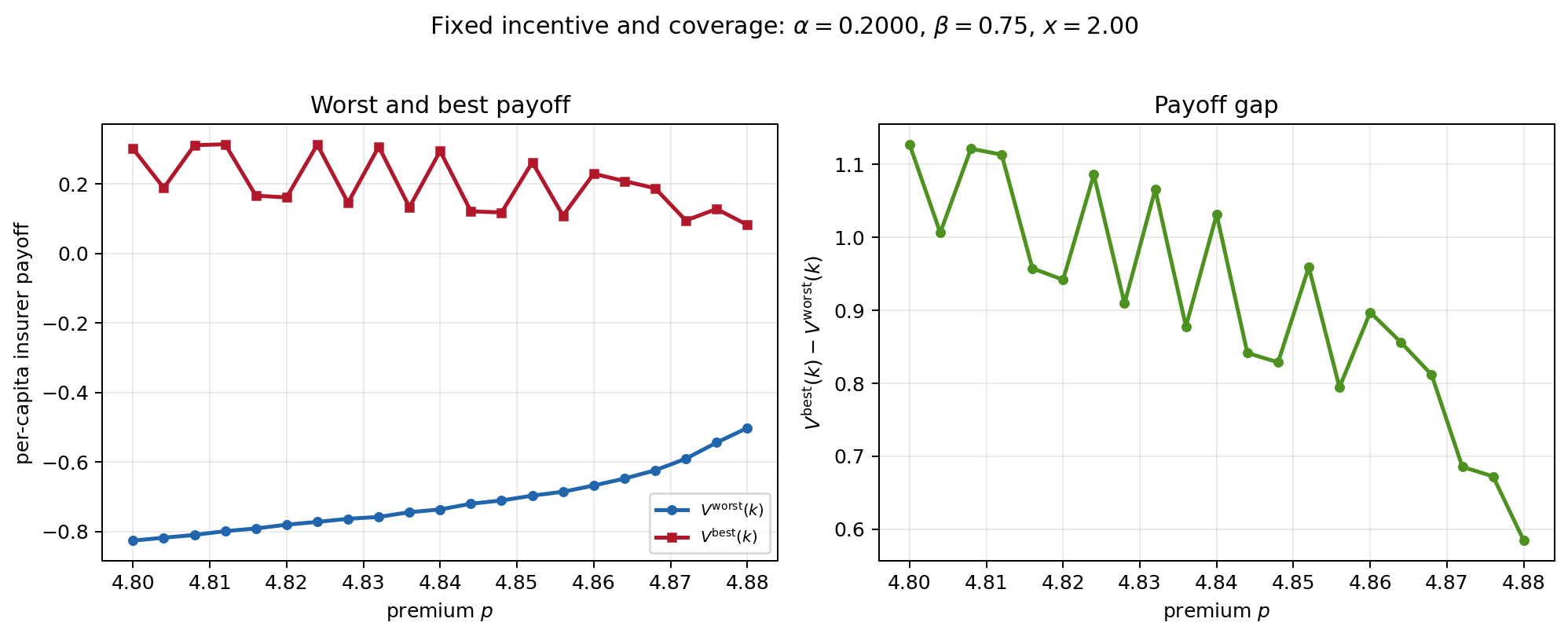}
    \caption{Worst- and best-equilibrium insurer payoffs when the premium $p$ varies}
    \label{fig:worst_best_vary_p}
\end{figure}

\Cref{fig:worst_best_vary_alpha} fixes $p=4.8$, $\beta=0.75$, and varies the effort-incentive parameter $\alpha$. As in the premium slice, the gap between the worst- and best-equilibrium payoff criteria remains substantial, showing that equilibrium selection can also matter along the incentive dimension.
\begin{figure}[htbp]
    \centering
    \includegraphics[width=0.95\textwidth]{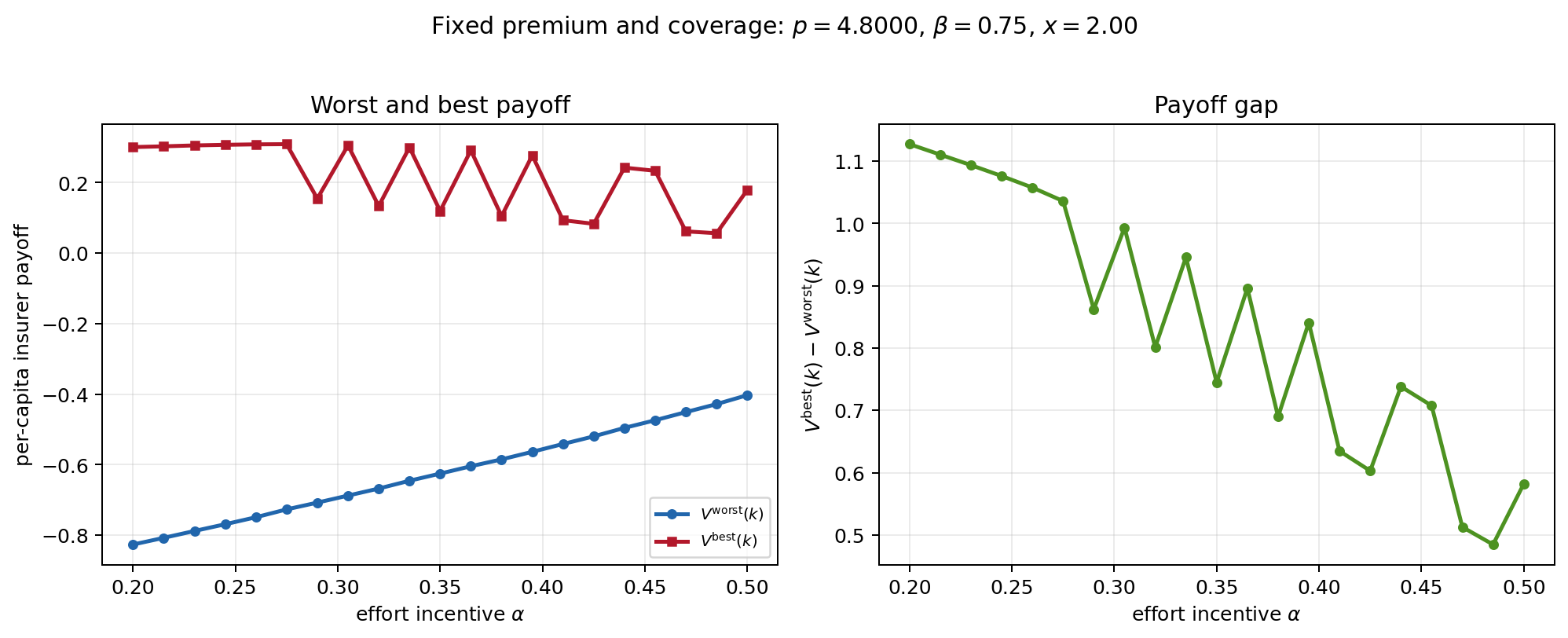}
    \caption{Worst- and best-equilibrium insurer payoffs when the effort-incentive parameter $\alpha$ varies.}
    \label{fig:worst_best_vary_alpha}
\end{figure}

One noticeable feature is that, in both graphs, the worst-equilibrium payoff varies relatively smoothly, while the best-equilibrium payoff is less smooth. This suggests that the worst-equilibrium criterion may be following a relatively stable equilibrium branch, whereas the best-equilibrium criterion can switch across different payoff-maximizing equilibrium branches as the contract parameters change. Thus, the irregularity of the best-payoff curve reflects the fact that best-equilibrium evaluation is an upper envelope over multiple lower-level equilibria.

Another feature is that the worst-equilibrium payoff is negative for most contracts in this experiment. This occurs because the displayed contracts are not restricted to those that guarantee the insurer a nonnegative payoff under every induced equilibrium. For some contracts, the participation pattern generated by an unfavorable lower-level equilibrium can make expected reimbursement and incentive payments exceed premium revenue. This further highlights why the choice between worst- and best-equilibrium performance criteria matters for the insurer's evaluation of contracts.

\section{Contract Menu} \label{sec: menu}

In this section, we extend the model from one insurance contract to a {\it menu of contracts}. The main idea is that, once agents differ in their characteristics, for example in their degree of risk-aversion, the insurer may improve performance by offering multiple contracts rather than a single policy. A menu allows agents to self-select across contracts with different prices and coverage levels, thereby generating endogenous screening. The menu-driven screening mechanism analyzed here represents a practical structural innovation for insurance operations. By endogenously aligning coverage tiers with unobservable risk-aversion profiles, the optimized menu mitigates the adverse selection inherent in interdependent networks, providing insurers with a scalable, data-driven blueprint for product design.

We begin in \Cref{sec: two contracts} with the two-contract case, which is the simplest environment in which screening can arise and be analyzed numerically. \Cref{sec: NE in two} characterizes the lower-level mean-field equilibrium among agents, \Cref{sec: NE numerics in two} presents the corresponding lower-level numerical illustrations, and \Cref{sec: Insurer in two} studies the corresponding upper-level insurer problem together with numerical illustrations in \Cref{sec: Insurer numerics in two}. We then extend the analysis to a general multi-contract menu in \Cref{sec: multi contracts}. \Cref{sec: NE in multi} and \Cref{sec: Insurer in multi} formulate the corresponding agents' equilibrium and the insurer's problem, respectively, and \Cref{sec: Insurer numerics in multi} presents numerical comparisons across menus of different sizes.

As in the previous sections, each agent has type $\theta=(\gamma,c,\sigma)$ drawn from a distribution $\pi$ on $\Theta$. Throughout this section, we simplify the model by fixing $c\equiv \bar c$ and $\sigma\equiv \bar\sigma$ as constants shared by all agents. Hence, heterogeneity comes only from $\gamma$, and $\pi$ reduces to a distribution over $\gamma$.

\subsection{Two Contracts} \label{sec: two contracts}

The insurer offers two contracts
$C_A=(p_A,\alpha_A,\beta_A),\;
C_B=(p_B,\alpha_B,\beta_B)$. Each agent chooses among the outside option and the two contracts through a mixed choice vector
$q(\gamma;\tilde m)=(q^{\mathrm{out}}(\gamma;\tilde m),q^A(\gamma;\tilde m),q^B(\gamma;\tilde m))\in [0,1]^3$, where $q^s(\gamma;\tilde m)$ is the probability that type $\gamma$ selects option $s$ given the mean-field level $\tilde m$, satisfying $q^{\mathrm{out}}+q^A+q^B=1$. For screening purposes, it is natural to impose an ordering so that one contract provides stronger risk sharing but at a higher price.
\begin{itemize}
    \item $p_B \ge p_A$ (contract B is more expensive),
	\item $\beta_B \ge \beta_A$ (contract B offers weakly higher coverage),
    \item $\alpha_B \ge \alpha_A$ (contract B offers stronger effort-based incentives)
\end{itemize}
Contract B mainly reduces residual risk more strongly, making it relatively more attractive to agents with higher $\gamma$.

In what follows, we first characterize the agents' lower-level mean-field equilibrium under the menu, then formulate the corresponding upper-level insurer problem, and finally present numerical illustrations of the induced sorting patterns and insurer performance.

\subsubsection{The Agents' Nash Equilibrium} \label{sec: NE in two}
If the agent does not purchase insurance, the certainty equivalent is
\[
{\rm CE}^{\mathrm{out}}(e;\gamma, \tilde m) = \mu(e + x\tilde{m}) + \bar{c} \, e + \frac{1}{2} \, \gamma \, \lambda(e + x\tilde{m}).
\]
Under contract A, the certainty equivalent is
\[
{\rm CE}^{A}(e;\gamma, \tilde m) = p_A + (\bar c - \alpha_A)\, e + \frac{1}{2}\,\gamma\,\alpha_A^2 \bar\sigma^2 + (1-\beta_A)\,\mu(e + x\tilde{m}) + \frac{1}{2}\,\gamma\,(1-\beta_A)^2 \lambda(e + x\tilde{m}).
\]
Similarly, under contract B, the certainty equivalent is
\[
{\rm CE}^{B}(e;\gamma, \tilde m) = p_B + (\bar c - \alpha_B)\, e + \frac{1}{2}\,\gamma\,\alpha_B^2 \bar\sigma^2 + (1-\beta_B)\,\mu(e + x\tilde{m}) + \frac{1}{2}\,\gamma\,(1-\beta_B)^2 \lambda(e + x\tilde{m}).
\]
Under each option (outside the contract, contract A, contract B), the agent chooses effort to minimize the corresponding certainty equivalent, formally, 
\[
e^{s}(\gamma; \tilde m) \in \argmin_{e \ge 0} {\rm CE}^{s}(e;\gamma, \tilde m).
\]
The menu choice is then represented by a mixed choice vector $q(\gamma;\tilde m) = \big(q^{\mathrm{out}}(\gamma;\tilde m),q^A(\gamma;\tilde m),q^B(\gamma;\tilde m)\big)$ over the three options ${\mathrm{out},A,B}$. This mixed formulation is the direct analogue of the participation probability in the one-contract model: probability is assigned only to options that attain the lowest minimized certainty equivalent. The resulting mean-field equilibrium requires the aggregate expected effort induced by these mixed choices to be consistent with the conjectured mean-field level $\tilde m$.

\begin{definition}[Mean-Field Equilibrium with Contract Menu]
    Given the contracts $(C_A, C_B)$ and the population distribution $\pi$ over $\gamma$, a tuple
    \[
    \Big(\tilde m, \{e^{\mathrm{out}}(\cdot; \tilde{m}), e^{A}(\cdot; \tilde{m}), e^{B}(\cdot; \tilde{m})\}, \{q^{\mathrm{out}}(\cdot; \tilde{m}), q^{A}(\cdot; \tilde{m}), q^{B}(\cdot; \tilde{m})\}\Big)
    \]
    forms a mean-field Nash equilibrium if the following conditions hold:
    \begin{enumerate}
        \item[(1)] For each risk-aversion level $\gamma$ and given mean-field level $\tilde{m}$, the effort choices satisfy
        \[
        e^{s}(\gamma; \tilde m) \in \argmin_{e \ge 0} {\rm CE}^{s}(e;\gamma, \tilde m), \qquad s \in \{\mathrm{out}, A, B\}.
        \]
        
        \item[(2)] The mixed menu choice vector
        \[
        q(\gamma;\tilde m)
=
\big(q^{\mathrm{out}}(\gamma;\tilde m),q^A(\gamma;\tilde m),q^B(\gamma;\tilde m)\big)
        \]
        is measurable in $\gamma$ and satisfies
        \[
q^{\mathrm{out}}(\gamma;\tilde m)+q^A(\gamma;\tilde m)+q^B(\gamma;\tilde m)=1,
\qquad
q^s(\gamma;\tilde m)\ge 0, \qquad s\in\{\mathrm{out},A,B\}.
\]  
Moreover, positive probability is assigned only to options that minimize the certainty equivalent; equivalently, 
\[
q^s(\gamma;\tilde m) = 0, \qquad \forall s\notin\argmin_{r\in\{A,B,\mathrm{out}\}} \; {\rm CE}^{r}\big(e^{r}(\gamma; \tilde m);\gamma, \tilde m\big).
\]
If one option is the unique minimizer, then the mixed choice vector assigns probability $1$ to that option; if several options attain the same minimized certainty equivalent, then the agent may randomize among these tied options.

        \item[(3)] The aggregate expected effort induced by $\tilde m$ is
\[
\bar e(\tilde{m}) = \int
\Big[
q^{\mathrm{out}}(\gamma;\tilde m)e^{\mathrm{out}}(\gamma;\tilde m)
+
q^A(\gamma;\tilde m)e^{A}(\gamma;\tilde m)
+
q^B(\gamma;\tilde m)e^{B}(\gamma;\tilde m)
\Big]\,d\pi(\gamma).
\]
        The equilibrium mean effort $\tilde{m}$ satisfies the fixed point condition
        \[
        \tilde m = \bar e(\tilde m).
        \]
    \end{enumerate}
\end{definition}

We do not provide a separate theorem for the contract menu model, since the argument is essentially a finite-action extension of the one-contract case. Under a mild extension of \Cref{assum_NE}, the binary participation rule is replaced by a mixed choice vector over the finite option set ${\mathrm{out},A,B}$. The corresponding expected-effort correspondence is then obtained by taking convex combinations of the optimal efforts associated with the options. With this modification, the existence of a mean-field Nash equilibrium follows from the same measurable-selection and Kakutani fixed-point argument used in the one-contract model. For this reason, we focus here on the equilibrium formulation and the associated numerical analysis, and omit a separate theorem and proof.

\subsubsection{The Agents' Nash Equilibrium - Numerics}  \label{sec: NE numerics in two}

In the two-contract numerical experiments, we use a simulated population of $N=5000$ agents. We draw the risk-aversion parameter from a truncated normal distribution. Specifically, we draw
$\gamma \sim \mathcal{N}(1.5,\, 0.9^2)$ truncated to the interval $[\gamma_{\min}, \gamma_{\max}] = [10^{-4},\, 5]$. 
$c \equiv 2.2$ and $\sigma \equiv 2.0$ are fixed among agents, and the interdependence parameter is set to $x=1.0$. For the lower-level menu illustrations in Figures~\ref{fig:choice_shares} and \ref{fig:distribution_with_choices}, we use the two contracts $C_A=(3.5,0.02,0.52)$ and $C_B=(6.6,0.05,0.87)$.

As in the previous numerical sections, the two-contract equilibria reported below are numerically verified to be unique. The reported choice share of each option $s\in{\mathrm{out},A,B}$ is computed as
\[
\int q^s(\gamma;\tilde m^*)\,d\pi(\gamma),
\]
where $\tilde m^*$ is the equilibrium mean-field level. In the numerical implementation, the indifference sets are negligible, so the mixed choice vector is effectively pure for almost every $\gamma$.

\Cref{fig:choice_shares} illustrates the lower-level mean-field Nash equilibrium choice shares under a two-contract menu. The insurer offers a menu of two contracts $C_A$ and $C_B$, satisfying $p_B > p_A, \alpha_B > \alpha_A$, and $\beta_B > \beta_A$, so that contract B provides stronger risk sharing at a higher price. Under this menu, the equilibrium mixed choice vector is effectively pure almost everywhere: lower-$\gamma$ agents are assigned probability one to the outside option, intermediate agents are assigned probability one to contract $A$, and sufficiently high-$\gamma$ agents are assigned probability one to contract $B$. This illustrates that a properly designed menu can generate partial screening, with higher–$\gamma$ agents more likely to select the stronger contract.
\begin{figure}[htbp]
\centering
\includegraphics[width=0.8\textwidth]{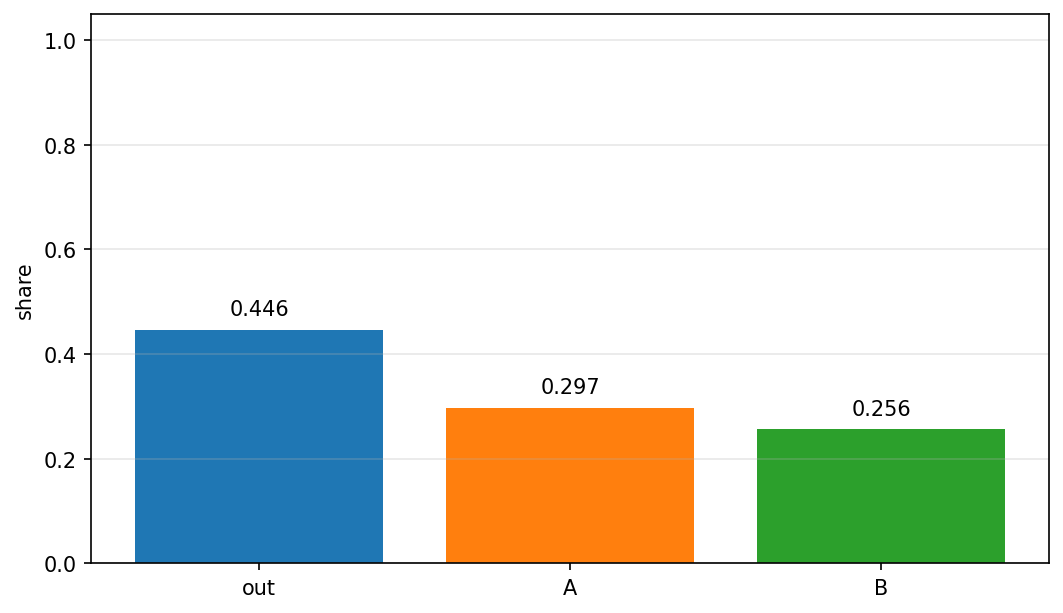}
\caption{Equilibrium choice shares under the contract menu.}
\label{fig:choice_shares}
\end{figure}

To further visualize the screening effect of the contract menu in the presence of heterogeneity, \Cref{fig:distribution_with_choices} overlays the equilibrium option assigned by the mixed choice vector on the population distribution of $\gamma$. The black curve represents the truncated normal density $\pi(\gamma)$, while the shaded regions indicate the option receiving probability $1$ for each value of $\gamma$. The two vertical thresholds $\gamma_1$ and $\gamma_2$ separate the regions corresponding to opting out, choosing contract $A$, and choosing contract $B$. Since in the present setup $c$ and $\sigma$ are fixed across agents, heterogeneity arises only from $\gamma$, so the equilibrium choice is determined solely by $\gamma$ and the thresholds $\gamma_1$ and $\gamma_2$ can be computed explicitly.
\begin{figure}[ht]
\centering
\includegraphics[width=0.8\textwidth]{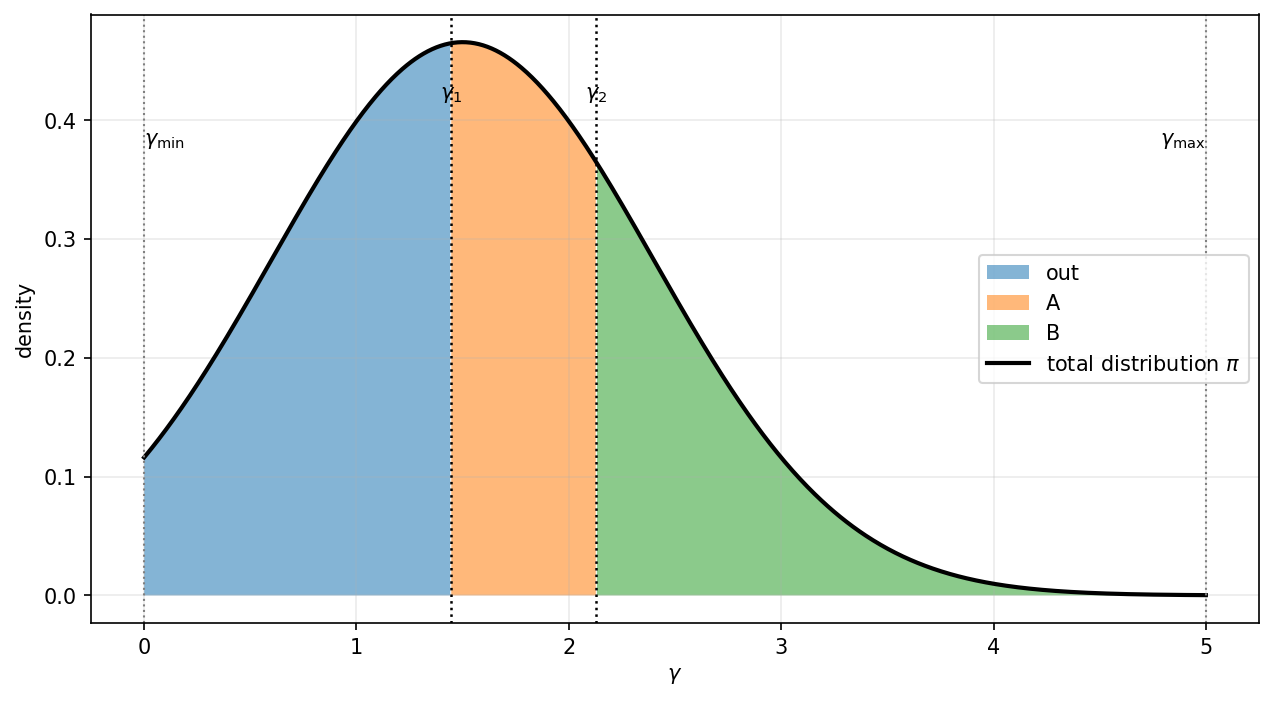}
\caption{Population distribution of $\gamma$ with the equilibrium contract choices under the menu.}
\label{fig:distribution_with_choices}
\end{figure}

\begin{remark}

\begin{enumerate}
\item[1.] In the current setup, we fix $c$ and $\sigma$ and allow heterogeneity only in $\gamma$. More general specifications are possible, for example, one may change the underlying $\gamma$-distributions or introduce heterogeneity in $c$ and $\sigma$. In such cases, the equilibrium mixed choice vector may generate more complex sorting patterns: some low-$\gamma$ agents may be assigned to contract $B$, while some high-$\gamma$ agents may be assigned to the outside option.

\item[2.] The equilibrium choice shares are highly sensitive to the parameters of $C_A$ and $C_B$. Small changes in prices or coverage levels can lead to very different outcomes: in some cases all agents participate; in others, total participation is below one but all participating agents are assigned to contract $B$, indicating that contract $B$ strictly dominates $A$ under that parameter configuration. Consequently, in the insurer’s problem, the joint design of $(C_A, C_B)$ is nontrivial. Beyond increasing overall participation, the insurer may also need to control how participating agents are allocated across contracts $A$ and $B$.
\end{enumerate}
\end{remark}

\subsubsection{The Insurer's Problem} \label{sec: Insurer in two}

For a two-contract menu $(C_A,C_B)$ with $C_s=(p_s,\alpha_s,\beta_s)$, $s\in{A,B}$, fix a lower-level mean-field effort level $\tilde m$. For each contract $s\in{A,B}$, the insurer's random payoff from an agent with risk-aversion $\gamma$ under contract $s$ is
\[
\Pi^s(\gamma;\tilde m)
=
p_s-\alpha_s S\!\big(e^{s}(\gamma;\tilde m)\big)
-\beta_s L\!\big(e^{s}(\gamma;\tilde m);\tilde m\big),
\qquad s\in\{A,B\}.
\]
Taking expectations, the insurer's expected payoff is
\[
V^s(C_s;\gamma,\tilde m)
=
p_s-\alpha_s e^{s}(\gamma;\tilde m)
-\beta_s\mu\!\big(e^{s}(\gamma;\tilde m)+x\tilde m\big),
\qquad s\in\{A,B\}.
\]
If the agent opts out, the insurer's payoff is zero.

For a given menu $(C_A,C_B)$, let
\begin{align*}
\mathcal M(C_A,C_B)
:=
\Big\{
&\big(
\tilde m,
\{e^{\mathrm{out}}(\cdot;\tilde m),e^{A}(\cdot;\tilde m),e^{B}(\cdot;\tilde m)\},
\{q^{\mathrm{out}}(\cdot; \tilde{m}), q^{A}(\cdot; \tilde{m}), q^{B}(\cdot; \tilde{m})\}
\big)
\\
&\text{ satisfying the equilibrium conditions}
\Big\}
\end{align*}
denote the set of mean-field equilibria induced by the menu.

Given a lower-level mean-field equilibrium $M\in\mathcal M(C_A,C_B)$, the insurer's per-capita
expected profit is
\begin{align*}
J(C_A,C_B;M)
= \int
\Big[
V^{A}(C_A;\gamma,\tilde m) q^A(\gamma;\tilde m)
+
V^{B}(C_B;\gamma,\tilde m) q^B(\gamma;\tilde m)
\Big]\,d\pi(\gamma),
\end{align*}
Here $q^A(\gamma;\tilde m)$ and $q^B(\gamma;\tilde m)$ are the equilibrium probabilities assigned by a type-$\gamma$ agent to contracts $A$ and $B$, respectively.

The insurer's menu-design problem is therefore
\[
\sup_{C_A,C_B} J^{\mathrm{menu}}(C_A,C_B),
\qquad
J^{\mathrm{menu}}(C_A,C_B)
:=
\sup_{M\in\mathcal M(C_A,C_B)} J(C_A,C_B;M).
\]

This formulation also provides the comparison with the one-contract benchmark.
Indeed, the one-contract problem is nested in the menu problem by imposing
$C_A=C_B=C$. In that case, the two contract labels are payoff-equivalent, and the insurer's payoff depends only on the total probability $q^A(\gamma;\tilde m)+q^B(\gamma;\tilde m)$ assigned to the identical contract. Therefore,
\[
\sup_{C_A,C_B} J^{\mathrm{menu}}(C_A,C_B)
\ge
\sup_C J^{\mathrm{one}}(C).
\]
Hence, after optimization, two contracts can never perform worse than one
contract. If the optimal menu satisfies $C_A^*=C_B^*$, then the model
endogenously collapses to the one-contract benchmark and the second contract has
no additional value. By contrast, if the optimizer yields $C_A^*\neq C_B^*$ and
a strictly larger value, then menu-based screening creates additional insurer
value.

\subsubsection{Numerical Optimization of the Two-Contract Menu} \label{sec: Insurer numerics in two}

We conclude with a very simple numerical example that compares the insurer's performance under a single-contract offer and under a two-contract menu. In particular, we compare contract $A$ only with contract $(A,B)$, and contract $B$ only with contract $(A,B)$. 
The simple comparison experiment uses the same population calibration as above, with $N=5000$, $\gamma\sim\mathcal N(1.5,0.9^2)$ truncated to $10^{-4},5$, $c\equiv 2.2$, $\sigma\equiv 2.0$, and $x=1.0$. The contracts are $C_A=(3.5,0.02,0.52)$ and $C_B=(6.6,0.05,0.87)$.

As in the previous numerical exercise, the indifference sets are numerically negligible, and the fixed-point equation is numerically checked to have a unique solution in all reported two-contract experiments. Hence the mixed choice vector is effectively pure for almost every $\gamma$, and the reported shares can be interpreted as equilibrium choice shares.

Figure~\ref{fig:A-vs-menu} and Figure~\ref{fig:B-vs-menu} reports the lower-level Nash equilibrium outcomes. In each figure, the left panel shows the equilibrium choice shares, where the participation share is computed as the total probability assigned to contracts $A$ and $B$, and the right panel shows the comparison between the insurer's per-capita profits. Relative to offering only contract $A$ or contract $B$, the two-contract menu splits insured agents across two contracts and yields a higher per-capita profit.

This is only a stylized example, not an optimized exercise. It's simply used to illustrate that, even in a basic calibration, allowing the insurer to offer two contracts can improve performance relative to offering only one contract. In this sense, the figure provides preliminary numerical support for the idea that menu-based screening can create additional insurer value.

\begin{figure}[ht]
    \centering
    \includegraphics[width=0.8\textwidth]{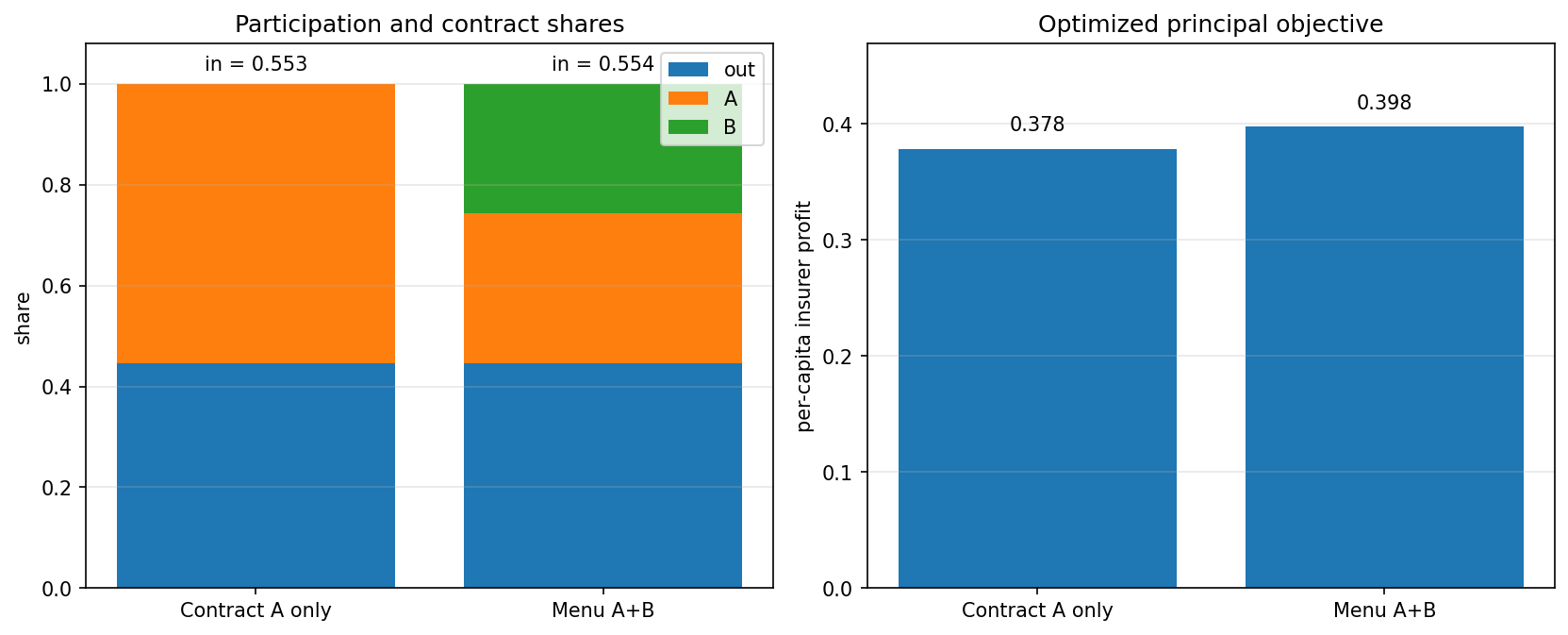}
    \caption{A simple comparison between offering only contract $A$ and offering the menu $(A,B)$.}
    \label{fig:A-vs-menu}
\end{figure}

\begin{figure}[ht]
    \centering
    \includegraphics[width=0.8\textwidth]{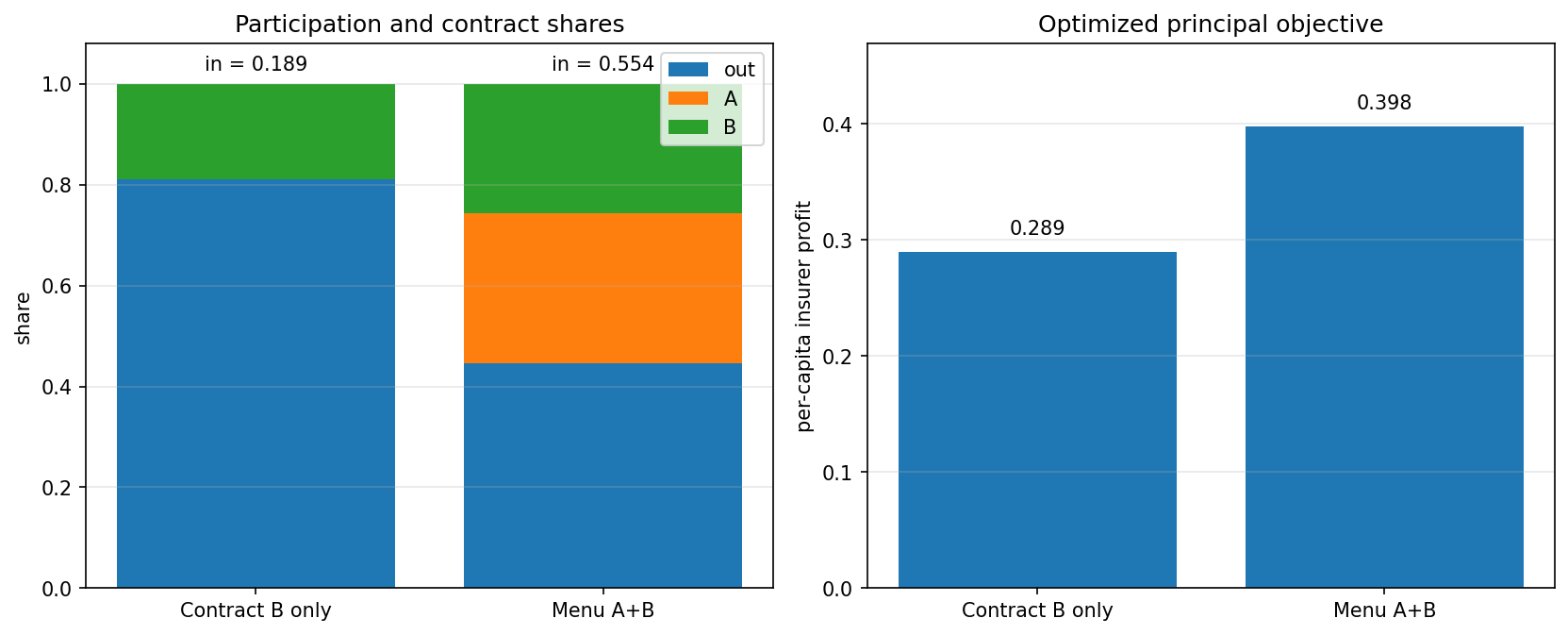}
    \caption{A simple comparison between offering only contract $B$ and offering the menu $(A,B)$.}
    \label{fig:B-vs-menu}
\end{figure}

Furthermore, we solve the insurer's problem numerically and compare the optimal one-contract benchmark with the optimal two-contract menu. 

Figure~\ref{fig:optimized-problem-comparison} reports the outcome of this optimization exercise. For the optimized one-contract problem, the best contract is
\[
C^{*,\mathrm{one}}=(p,\alpha,\beta)=(2.9,\,0.06,\,0.45),
\]
which yields an insurer per-capita profit of approximately $0.4026$ and a lower-level equilibrium participation share of approximately $0.6476$. For the optimized two-contract problem, the best menu is
\[
C_A^*=(2.2,\,0.04,\,0.35), \qquad
C_B^*=(7.1,\,0.04,\,0.94).
\]
Under this optimized menu, the insurer's per-capita profit rises to approximately $0.4635$. The total participation share is approximately $0.6374$, with equilibrium choice shares of about $0.3980$ for contract $A$ and about $0.2394$ for contract $B$.

Several features are worth emphasizing. First, the optimized one-contract benchmark generates a higher total participation share than the optimized two-contract menu: participation is about $64.76\%$ under the single contract, compared with about $63.74\%$ under the two-contract menu. However, despite this lower participation share, the optimized two-contract menu delivers a substantially higher per-capita profit. In other words, the menu improves insurer value not by maximizing total participation, but by reallocating equilibrium choice probabilities across contracts in a more profitable way. Second, the optimized menu has a natural screening interpretation: contract $A$ is slightly cheaper and offers weaker coverage, while contract $B$ is substantially more expensive and provides much stronger coverage. Thus, the numerical optimizer selects a menu with a clear low-tier/high-tier structure.

A related observation is that the optimized one-contract benchmark still achieves a slightly higher participation share than the optimized two-contract menu, even though contract $A$ in the menu is cheaper. One possible explanation is that contract $A$ is not only cheaper, but also provides weaker coverage: the optimized one-contract benchmark is $(2.9,0.06,0.45)$, whereas the low-tier menu contract is $(2.2,0.04,0.35)$. Thus, for some types, the lower premium may not fully compensate for the lower insurance value. As a useful complement, the agents' equilibrium effort levels are
\[
\bar e^{\,\mathrm{one}} = 0.5755, \qquad
\bar e^{\,\mathrm{one}}_{\mathrm{insured}} = 0.4255, \qquad
\bar e^{\,\mathrm{one}}_{\mathrm{out}} = 0.8510,
\]
under the optimized one-contract benchmark, and
\[
\bar e^{\,\mathrm{menu}} = 0.5517, \qquad
\bar e^{\,A} = 0.5849, \qquad
\bar e^{\,B} = 0, \qquad
\bar e^{\,\mathrm{out}} = 0.8796,
\]
under the optimized two-contract menu. These statistics suggest that the menu changes not only contract choice, but also the composition of effort across types.

The outperformance of the optimized menu stems directly from its ability to align incentives across heterogeneous agents. Rather than offering a blunt, one-size-fits-all policy, the menu leverages differential coverage to safely screen highly risk-averse agents while maintaining strict security-investment incentives for lower-cost agents. This mitigates the moral hazard that would otherwise be amplified by the network's domino effect.

Finally, we note that the present optimization is still only approximate. Numerically, we use a relatively coarse grid search around the relevant parameter region. The one-contract problem and two-contract menu are built from the same underlying grid structure, but we restrict attention to different parameter regions in order to focus the search on the most relevant candidates and to reduce computation. In the final search reported here, the one-contract problem is solved over $10\times 5\times 9=450$ candidate contracts, while the two-contract problem is solved over ordered pairs of contracts, with $432$ candidates for contract $A$ and $252$ candidates for contract $B$, for a total of up to $108{,}864$ menu candidates. Even with this grid, the computation time is already noticeable, so the optimizer should still be interpreted as an approximate numerical solution. A substantially finer search would likely improve accuracy, but would also require much longer running time.

\begin{figure}[ht]
    \centering
    \includegraphics[width=0.9\textwidth]{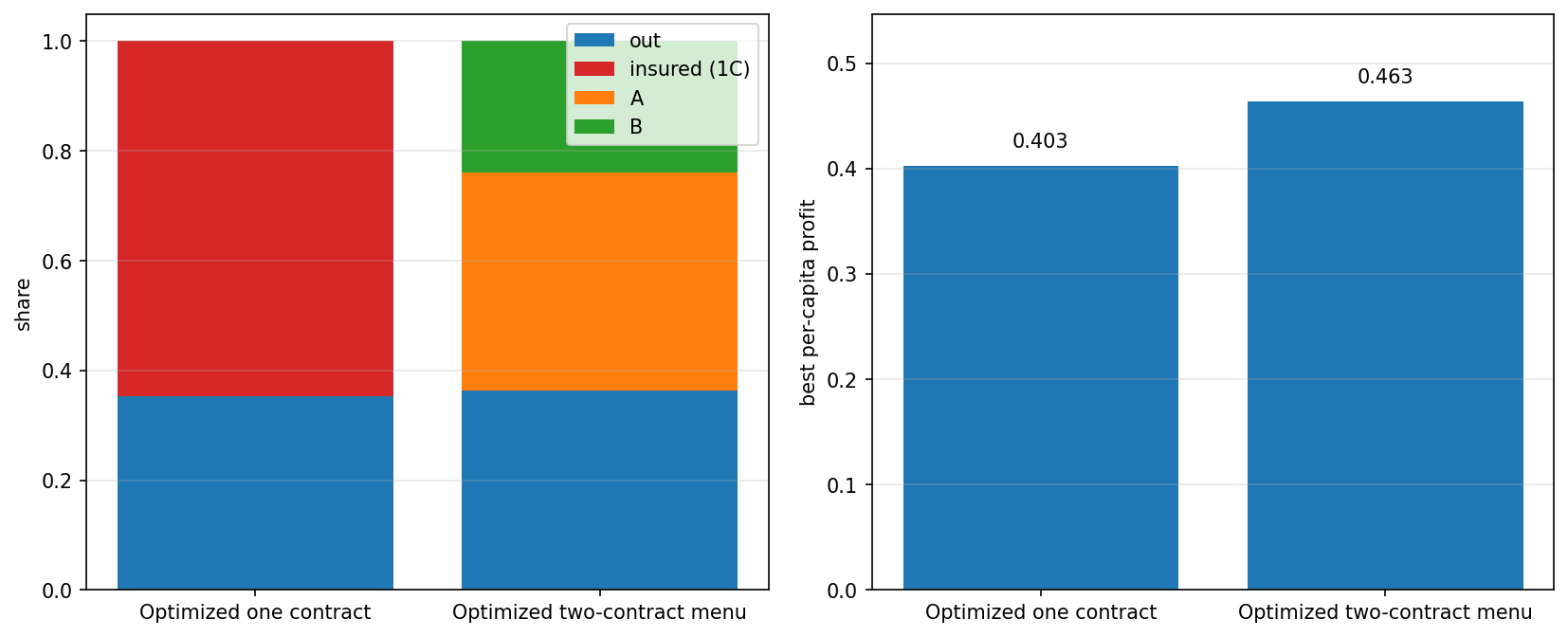}
    \caption{Optimized one-contract problem versus optimized two-contract problem.}
    \label{fig:optimized-problem-comparison}
\end{figure}

\subsection{Multi Contracts} \label{sec: multi contracts}

Before moving to a full multi-contract model, it is useful to clarify the main modeling choices. The basic idea is that we keep a single underlying type distribution for the population, but instead of offering only one contract or a two-contract menu, the insurer offers a larger menu
\[
\mathscr{C} = \{C_1,\dots,C_K\}, \qquad C_s=(p_s,\alpha_s,\beta_s),
\]
and each agent assigns probabilities to the outside option and to all contracts in the menu, with positive probability only on options that minimize the certainty equivalent.

At this point, there are two possible ways to formulate the model. The first approach is the direct extension of our current two-contract setup. Namely, we fix a large but finite collection of candidate contracts in some prescribed parameter region, and then optimize over menus formed by selecting several contracts from that collection. This is conceptually the closest to what we already do in the two-contract case, and it is also the most straightforward path for numerical implementation. For this reason, we plan to start from this approach. The main drawback is computational: once we move from two contracts to three or more, the search space grows very quickly, so the numerical optimization will be substantially more expensive than in the two-contract case.

A second approach would be to let agents choose contract features more continuously. For example, one could allow the agent to choose $(\alpha,\beta)$ and let the insurer determine the premium through a pricing rule
\[
p = P(\alpha,\beta),
\]
or impose some other functional relationship among the contract variables. This would turn the menu-design problem into the design of a pricing schedule or mechanism, rather than the choice of finitely many contracts. Such a formulation may be more flexible and perhaps more elegant from a screening perspective, but it is also significantly harder to formulate and solve. In particular, the insurer would no longer optimize over finitely many parameters, but over a function, which makes both the theory and the numerics much more complicated.

For now, we therefore focus on the first route: a finite multi-contract menu built from a large set of candidate contracts, leaving the more general functional-design formulation for future work.

\subsubsection{The Agents' Nash Equilibrium} \label{sec: NE in multi}

The two-contract framework extends naturally to a menu with $K \geq 3$ contracts. The insurer offers
\[
\mathscr{C} = \{C_1,\dots,C_K\}, \qquad C_s=(p_s,\alpha_s,\beta_s), \quad s=1,\dots,K,
\]
together with the outside option. 
As before, for a given mean-field level $\tilde m$, each agent compares the certainty equivalents generated by the outside option and by every contract in the menu. Since the certainty-equivalent expressions are exactly the same as in the two-contract case, we only change the indexing.

For each $s=1,\dots,K$, let ${\rm CE}^{s}(e;\gamma,\tilde m)$ denote the certainty equivalent under contract $C_s$, and let
\[
e^{s}(\gamma;\tilde m)\in\argmin_{e\ge 0} {\rm CE}^{s}(e;\gamma,\tilde m)
\]
be the associated optimal effort. The outside option is treated in the same way, with certainty equivalent ${\rm CE}^{\mathrm{out}}(e;\gamma,\tilde m)$ and optimal effort $e^{\mathrm{out}}(\gamma;\tilde m)$.

Given these optimal efforts, the menu choice is represented by a mixed choice vector $q(\gamma;\tilde m)
=
\big(q^{\mathrm{out}}(\gamma;\tilde m),q^1(\gamma;\tilde m),\dots,q^K(\gamma;\tilde m)\big)$, with
\[
q^{\mathrm{out}}(\gamma;\tilde m)+\sum_{s=1}^K q^s(\gamma;\tilde m)=1,
\qquad
q^{\mathrm{out}}\ge 0, \quad q^s(\gamma;\tilde m)\ge 0 \text{ for } s\in\{1,\dots,K\},
\]
and positive probability only on minimized CE options:
\[
q^s(\gamma;\tilde m)=0,
\qquad \forall s\notin
\argmin_{r\in\{\mathrm{out},1,\dots,K\}}
{\rm CE}^{r}\big(e^{r}(\gamma;\tilde m);\gamma,\tilde m\big).
\]
Thus, the finite menu induces an endogenous allocation of choice probabilities across the outside option and the $K$ contracts.

The aggregate mean-field level is then determined directly from the population distribution:
\[
\tilde m
=
\int
\left(
e^{\mathrm{out}}(\gamma;\tilde m) q^{\mathrm{out}}(\gamma;\tilde m)
+
\sum_{s=1}^K e^{s}(\gamma;\tilde m) q^s(\gamma;\tilde m)
\right)
\,d\pi(\gamma).
\]

Therefore, relative to the two-contract case, the Nash equilibrium definition is unchanged in structure: the only difference is that the mixed choice vector now has $K+1$ components.

\subsubsection{The Insurer's Problem} \label{sec: Insurer in multi}

The insurer's objective also extends directly. For each contract $C_s = (p_s, \alpha_s, \beta_s)$, the insurer's expected payoff from a type-$\gamma$ agent under contract $C_s$ is
\[
V^s(C_s;\gamma,\tilde m)
=
p_s-\alpha_s e^{s}(\gamma;\tilde m)
-\beta_s \mu\big(e^{s}(\gamma;\tilde m)+x\tilde m\big),
\qquad s=1,\dots,K,
\]
while the payoff is zero for agents under the outside option.

For a given menu $\mathscr{C}=\{C_1,\dots,C_K\}$ and a corresponding lower-level Nash equilibrium $M$, the insurer's per-capita expected profit is
\[
J(\mathscr{C};M)
=
\int
\left[
\sum_{s=1}^K V^s(C_s;\gamma,\tilde m) q^s(\gamma;\tilde m)
\right]
\,d\pi(\gamma).
\]

The insurer's multi-contract design problem is then
\[
\sup_{\mathscr{C}} J^{\mathrm{menu}}(\mathscr{C}),
\qquad
J^{\mathrm{menu}}(\mathscr{C})
:=
\sup_{M\in\mathcal M(\mathscr{C})} J(\mathscr{C};M),
\]
where $\mathcal M(\mathscr{C})$ denotes the set of mean-field equilibria induced by the menu $\mathscr{C}$.

Again, the key difference from the two-contract case is not conceptual but combinatorial. The lower-level Nash equilibrium and profit formulas extend in a straightforward way, but the insurer must now optimize jointly over a larger menu. As a result, even for a finite candidate set of contracts, the number of possible menus grows rapidly with $K$, making numerical implementation substantially more expensive.

\subsubsection{Numerical Optimization of the Multi-Contract Menu} \label{sec: Insurer numerics in multi}

A full brute-force optimization in the multi-contract problem becomes infeasible very quickly. If the insurer chooses each contract from a finite candidate set of size $N$, then an exhaustive search over all $K$ contract menus requires evaluating through $N^K$ possibilities. This is manageable for the one-contract and two-conract problems, but it becomes prohibitively expensive for $\geq 3$ contracts.

For this reason, we do not use global brute-force optimization in the multi-contract problem. Instead, we use a layered search procedure. Starting from the optimized two-contract benchmarks, we enlarge the menu step by step: first we search for a promising third contract conditioned on an optimized two-contract menu, then search for a promising fourth contract, and so on. Rather than enumerating all possible $K$-tuples, we build the menu sequentially and keep only the most promising intermediate candidates at each stage.

The multi-contract experiments use the same population calibration as the optimized two-contract benchmark: $N=5000$, $\gamma\sim\mathcal N(1.5,0.9^2)$ truncated to $10^{-4},5$, $c\equiv 2.2$, $\sigma\equiv 2.0$, and $x=1.0$. We start from the optimized two-contract menu $C_A^*=(2.2,0.04,0.35)$ and $C_B^*=(7.1,0.04,0.94)$, and search over a finite candidate set for the additional contract.

Before reporting the multi-contract results, we use the same interpretation of choice shares as in the two-contract case. In the reported examples, the indifference sets are $\pi$-null, so the mixed choice vector is effectively pure for almost every $\gamma$. We also checked the fixed-point equation numerically for the reported optimized multi-contract menu, and the updated parameter choice yields a unique equilibrium mean-field level.

First we consider a constrained three-contract extension built directly on top of the optimized two-contract baseline, which is $C_A^*=(2.2,\,0.04,\,0.35), \ C_B^*=(7.1,\,0.04,\,0.94)$, and search only for a third contract $C$. The best third contract found is $C_C^* = (2.0, 0.074, 0.32)$. The total equilibrium participation share is $0.6520$, with equilibrium choice shares $0.3480$ for the outside option, $0.2576$ for contract $A$, $0.2356$ for contract $B$, and $0.1588$ for contract $C$. Under the menu $(A, B, C)$, the insurer's per-capita profit slightly increases from 0.4635 to 0.4668. Thus, allowing one additional contract does improve profit, but only by a small percentage. The resulting comparison is shown in \Cref{fig:restrained-third-contract}.

\begin{figure}[ht]
    \centering
    \includegraphics[width=0.9\textwidth]{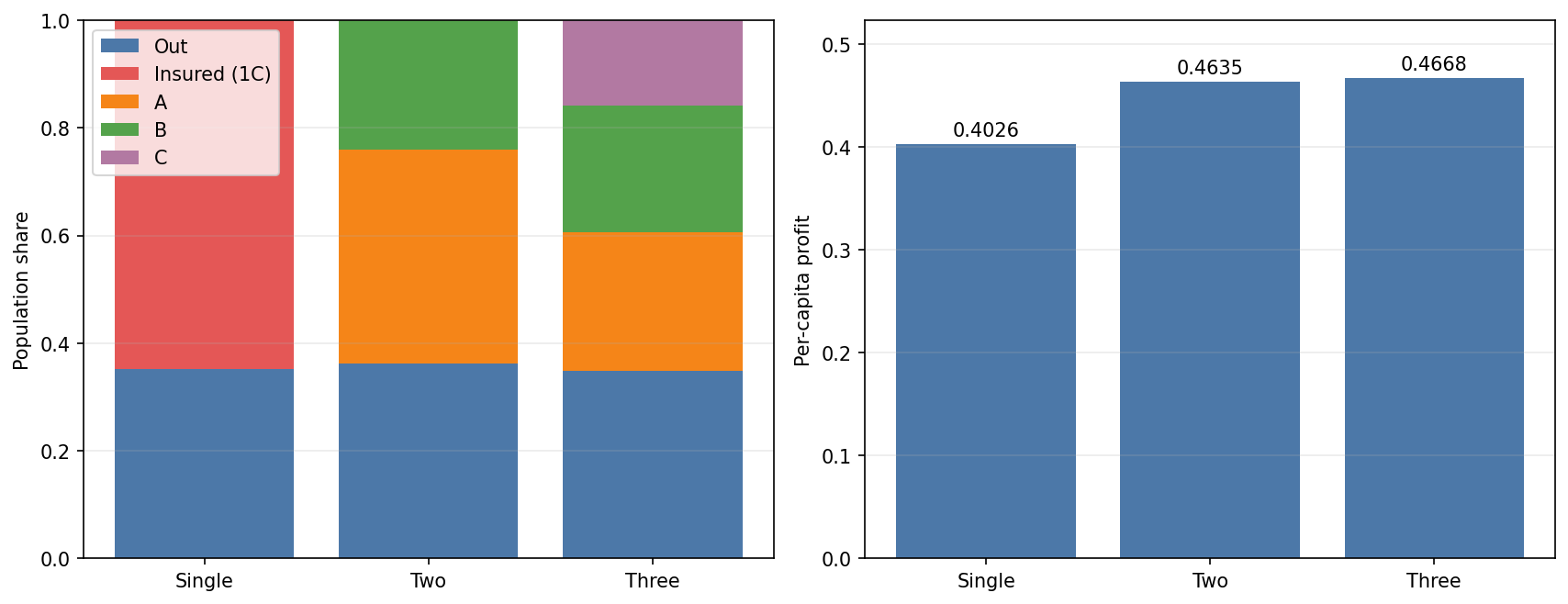}
    \caption{Fixed Two-Contract Baseline with an Additional Third Contract.}
    \label{fig:restrained-third-contract}
\end{figure}

We also repeat the three-contract experiment starting from the top 10 two-contract menus, rather than fixing only the single best two-contract. For each of these top 10 menus, we search over the candidate set for an additional third contract $C$, and compare the resulting three-contract menus. The outcome is the same as in the experiment using fixed best $A, B$, which is $C_A^*=(2.2,\,0.04,\,0.35), \ C_B^*=(7.1,\,0.04,\,0.94), C_C^* = (2.0, 0.074, 0.32)$. Overall, in our optimization experiments, the optimized three-contract menu only achieves slight improvement over the optimized two-contract menu.

\section{Conclusions and Future Outlook}\label{sec:6}

This paper studies insurance contract design under moral hazard, endogenous participation, and strategic risk interdependence, motivated by cyber-insurance markets in which an agent's loss depends not only on individual prevention effort but also on the aggregate effort of the insured population. To address the curse of dimensionality inherent in the finite $N$-player game, we developed a heterogeneous mean-field framework. For a fixed contract $(p, \alpha, \beta)$, we established the existence of a mean-field Nash equilibrium via measurable-selection and Kakutani fixed-point arguments, and proved uniqueness under a contraction condition on the interdependence parameter $x$. We then embedded this equilibrium characterization into the insurer's upper-level Stackelberg problem, proving the existence of upper-level $\varepsilon$-optimal contracts under equilibrium multiplicity, and an exact Stackelberg equilibrium under unique lower-level equilibrium. Extending the model to a contract menu, we showed numerically that menu-based screening improves insurer per-capita profit relative to the single-contract benchmark by inducing self-selection across risk-aversion types.

Ultimately, the framework presented in this manuscript offers a comprehensive and mathematically rigorous resolution to the challenge of pricing insurance under endogenous participation and strategic risk interdependence.  By embedding the agents' heterogeneous mean-field equilibrium directly into the insurer's Stackelberg optimization, this work provides a robust, self-contained foundation for future research in systemic risk management. In particular, the mean-field equilibrium characterization developed here makes analytically tractable the coupling between individual contract incentives and population-level risk exposure—a coupling that is central to understanding systemic risk in interconnected insurance markets.

Beyond these theoretical contributions, our framework offers direct managerial insights for the underwriting of interconnected risks. The equilibrium analysis shows how contract terms shape participation and prevention effort, and how the insurer can use this response in contract design.  The contract-menu extension then shows how these instruments can be organized into tiered products that screen heterogeneous agents across risk-aversion types. We demonstrate that by operationalizing pre-screening technologies within a multi-tier contract menu, insurers can effectively sever the contagion of moral hazard, aligning individual incentives and restoring profitability in markets that might otherwise suffer from systemic failure.

A natural direction for future research is to enrich the structure of risk interdependence. In the present model, interaction is summarized by the scalar interdependence parameter $x$, so each agent is affected by the same aggregate mean-field effort. In many interconnected-risk environments, however, exposure is heterogeneous: one firm may be strongly linked to a small number of suppliers, while another may be weakly exposed to a broad set of counterparties. This suggests replacing the scalar interaction parameter with pairwise intensities $x_{ij}$, or with a graphon representation of the interaction network in large populations. Such an extension would allow the model to capture community structure, heterogeneous network positions, and richer channels of localized contagion in cyber-insurance markets.

Other extensions concern the design of richer contract instruments and the treatment of equilibrium multiplicity. The finite-menu analysis in this paper could be extended toward continuous pricing or coverage schedules, where the insurer designs a menu over a continuum of contract characteristics rather than selecting a finite set of contracts. In addition, while the present paper handles lower-level multiplicity through performance criteria and $\varepsilon$-optimal contracts, a sharper theory of equilibrium selection and exact upper-level existence under multiplicity would further strengthen the Stackelberg contract-design framework.

Together, these extensions suggest that the heterogeneous mean-field approach to insurance contract design developed in this paper is not only a tractable solution to the specific problem studied here, but also a flexible analytical framework adaptable to richer network structures, more complex contract instruments, and a broader class of systemic risk environments.

\bibliographystyle{abbrvnat}
\bibliography{references}

\appendix
\section{Appendix} \label{sec: appendix}

In the following appendices, we present the proofs of the main results of the paper.

\subsection{Proof of \Cref{thm_NE}} \label{sec: proof of existence}
To establish existence of equilibrium, we define a best-response correspondence $\Gamma$ at the population level. Recall that $E^{\mathrm{in}}(\theta;\tilde m)$ and $E^{\mathrm{out}}(\theta;\tilde m)$ are the optimal effort correspondences in \eqref{eq: optimal effort correspondences}. First, we define the optimal certainty-equivalent gap
\begin{equation} \label{eq:CEgap}
\Delta {\rm CE}(\theta;\tilde m)
:= \min_{e\in E^{\mathrm{in}}(\theta;\tilde m)} {\rm CE}^{\mathrm{in}}(e;\theta,\tilde m)
   - \min_{e\in E^{\mathrm{out}}(\theta;\tilde m)} {\rm CE}^{\mathrm{out}}(e;\theta,\tilde m).
\end{equation}
Next, define the participation correspondence
\[
Q(\theta;\tilde m)
:=
\begin{cases}
\{1\}
& \text{if } \Delta {\rm CE}(\theta;\tilde m) < 0, \\[6pt]
\{0\},
& \text{if } \Delta {\rm CE}(\theta;\tilde m) > 0, \\[6pt]
[0, 1],
& \text{if } \Delta {\rm CE}(\theta;\tilde m)=0.
\end{cases}
\]
This captures the mixed participation rule: a type strictly enters if the inside option is strictly better, strictly stays out if the outside option is strictly better, and may mix when indifferent.

Then we define the feasible expected-effort correspondence
\[
\mathcal E(\theta; \tilde m) := \{q a + (1-q) b: a \in E^{\mathrm{in}}(\theta;\tilde m), b \in E^{\mathrm{out}}(\theta;\tilde m), q \in Q(\theta; \tilde m)\},
\]
or equivalently,
\[
\mathcal E(\theta;\tilde m)
=
\begin{cases}
E^{\mathrm{in}}(\theta;\tilde m),
& \text{if } \Delta {\rm CE}(\theta;\tilde m) < 0, \\[6pt]
E^{\mathrm{out}}(\theta;\tilde m),
& \text{if } \Delta {\rm CE}(\theta;\tilde m) > 0, \\[6pt]
\operatorname{conv}\!\Big(
E^{\mathrm{in}}(\theta;\tilde m)
\cup
E^{\mathrm{out}}(\theta;\tilde m)
\Big),
& \text{if } \Delta {\rm CE}(\theta;\tilde m)=0.
\end{cases}
\]
Thus $\mathcal E(\theta;\tilde m)$ represents the set of feasible expected efforts generated by pure or mixed participation decisions.

For each $\tilde m \in [0,E_{\max}]$, define $\Gamma(\tilde m)$ as the set of all feasible population-average expected efforts generated by measurable selections of individual best responses:
\[
\Gamma(\tilde m)
:=
\left\{
\int_\Theta e(\theta)\, d\pi(\theta)
:
e(\theta)\in \mathcal E(\theta;\tilde m)
\text{ for $\pi$-a.e.~}\theta
\right\},
\]

Next, we will apply the Kakutani fixed-point theorem to $\Gamma$. To do so, we verify that for each $\tilde{m} \in [0, E_{\max}]$, the set $\Gamma({\tilde{m}})$ is nonempty, convex, and contained in $[0,E_{\max}]$, and $\Gamma$ has a closed graph.

\vspace{5pt}\paragraph{\textbf{Step 1. $\Gamma(\tilde{m})$ is nonempty for every $\tilde{m} \in [0, E_{\max}]$.}} 

Fix an arbitrary $\tilde{m} \in [0, E_{\max}]$. We show that $\Gamma(\tilde{m}) \neq \emptyset$, i.e., there exists a measurable function $e(\cdot)$ such that $e(\theta) \in \mathcal{E}(\theta; \tilde{m})$ for $\pi$-a.e.~$\theta$, and hence $\int_\Theta e(\theta)\, d\pi(\theta) \in \Gamma(\tilde{m})$.

Under Assumptions (A2) and (A3), $\mu(z)$ and $\lambda(z)$ are continuous on $[0, (1+x)E_{\max}]$, $\gamma(\theta), c(\theta), \sigma(\theta)$ are measurable and uniformly bounded, then the certainty equivalent functions
\begin{align*}
&{\rm CE}^{\mathrm{out}}(e;\theta,\tilde m) = \mu(e+x\tilde m)+c(\theta)e+\tfrac12\gamma(\theta)\lambda(e+x\tilde m),\\
&{\rm CE}^{\mathrm{in}}(e;\theta,\tilde m) = p+(c(\theta)-\alpha)e+\tfrac12\gamma(\theta)\alpha^2\sigma(\theta)^2
+(1-\beta)\mu(e+x\tilde m)+\tfrac12\gamma(\theta)(1-\beta)^2\lambda(e+x\tilde m)
\end{align*}
are continuous on $[0, E_{\max}]$. Therefore, by the Weierstrass Extreme Value Theorem, each minimization problem admits at least one minimizer, thus
\[
E^{\mathrm{out}}(\theta;\tilde m)=\argmin_{e\in[0,E_{\max}]} {\rm CE}^{\mathrm{out}}(e;\theta,\tilde m),\qquad E^{\mathrm{in}}(\theta;\tilde m)=\argmin_{e\in[0,E_{\max}]} {\rm CE}^{\mathrm{in}}(e;\theta,\tilde m)
\]
are nonempty compact subsets of $[0, E_{\max}]$.

Next we prove that $\Delta {\rm CE}(\theta; \tilde{m})$ is measurable in $\theta$. Since ${\rm CE}^{\mathrm{in}}(e;\theta,\tilde m)$ and ${\rm CE}^{\mathrm{out}}(e;\theta,\tilde m)$ are continuous with respect to $e$ and measurable with respect to $\theta$ (by Assumption (A2) and (A3), $\gamma, c, \sigma$ are measurable and $\mu, \lambda$ are continuous), $(\theta,e)\mapsto {\rm CE}^{\mathrm{out}}(e;\theta,\tilde m)$ and $(\theta,e)\mapsto {\rm CE}^{\mathrm{in}}(e;\theta,\tilde m)$ are Carathéodory. By the Measurable Maximum Theorem \citep[Theorem 18.19]{aliprantis2006infinite}, $\theta \mapsto \min_{e\in[0,E_{\max}]} {\rm CE}^{\mathrm{out}}(e;\theta,\tilde m)$, $\theta \mapsto \min_{e\in[0,E_{\max}]} {\rm CE}^{\mathrm{in}}(e;\theta,\tilde m)$ are measurable, thus $\theta \mapsto \Delta {\rm CE}(\theta; \tilde{m})$ is measurable. Therefore, the regions $\{\Delta {\rm CE}<0\}, \{\Delta {\rm CE}>0\}, \{\Delta {\rm CE}=0\}$ are measurable subsets of $\Theta$.

By construction,
\begin{itemize}
\item if $\Delta {\rm CE}(\theta;\tilde m)<0$, then $\mathcal E(\theta;\tilde m)=E^{\mathrm{in}}(\theta;\tilde m)$;
\item if $\Delta {\rm CE}(\theta;\tilde m)>0$, then $\mathcal E(\theta;\tilde m)=E^{\mathrm{out}}(\theta;\tilde m)$;
\item if $\Delta {\rm CE}(\theta;\tilde m)=0$, then $\mathcal E(\theta;\tilde m)=\operatorname{conv}\big(E^{\mathrm{in}}(\theta;\tilde m)\cup E^{\mathrm{out}}(\theta;\tilde m)\big)$.
\end{itemize}
Because $E^{\mathrm{in}}(\theta;\tilde m)$ and $E^{\mathrm{out}}(\theta;\tilde m)$ are nonempty compact subsets of $[0, E_{\max}]$, and the convex hull of a union of compact sets in finite dimensions is compact, $\mathcal E(\theta;\tilde m)$ is a nonempty compact subset of $[0, E_{\max}]$ for every $\theta$.

Furthermore, the Measurable Maximum Theorem applied previously guarantees that the argmin correspondences $E^{\mathrm{in}}(\cdot; \tilde m)$ and $E^{\mathrm{out}}(\cdot; \tilde m)$ are weakly measurable. Because the partition sets based on $\Delta {\rm CE}(\cdot; \tilde{m})$ are measurable, and the operations of finite union and taking the convex hull preserve the weak measurability of correspondences (see \citep[Proposition 2.3 and Theorem 9.1]{him1975}), the piecewise correspondence $\theta \mapsto \mathcal E(\theta; \tilde{m})$ is also weakly measurable.

Therefore, by the Kuratowski--Ryll-Nardzewski measurable selection theorem (\citep[Theorem 18.13]{aliprantis2006infinite}), there exists a measurable function $e(\theta)$ such that $e(\theta) \in \mathcal E(\theta; \tilde{m})$ for $\pi$-a.e.~$\theta$. By the construction of $\mathcal E(\theta; \tilde m)$, this selection $e(\theta)$ can be interpreted as the expected effort. Because $0 \leq e(\theta) \leq E_{\max}$, the Lebesgue integral $\int_{\Theta} e(\theta) d\pi(\theta) \in [0, E_{\max}]$ is well-defined. Thus,
\[
\int_{\Theta} e(\theta) d\pi(\theta) \in \Gamma(\tilde{m}),
\]
which proves that $\Gamma(\tilde{m}) \neq \emptyset$ for all $\tilde{m} \in [0, E_{\max}]$.

\vspace{5pt}\paragraph{\textbf{Step 2: $\Gamma(\tilde{m})$ is convex for every $\tilde{m} \in [0, E_{\max}]$.}}

Under Assumption (A2), $\mu(z)$ and $\lambda(z)$ are convex in $z$. Since $e \mapsto e + x\tilde{m}$ is affine,
\[
e \mapsto \mu(e + x\tilde{m}), \quad e \mapsto \lambda(e + x\tilde{m})
\]
are convex on $[0, E_{\max}]$. Since $\gamma(\theta) \geq 0$ and $1-\beta \geq 0$,
\[
e \mapsto {\rm CE}^{\mathrm{out}}(e; \theta; \tilde{m}), \quad e \mapsto {\rm CE}^{\mathrm{in}}(e; \theta; \tilde{m})
\]
are convex on $[0, E_{\max}]$.

Therefore, the sets of minimizers of ${\rm CE}^{\mathrm{in/out}}$
\[
E^{\mathrm{out}}(\theta;\tilde m)=\argmin_{e\in[0,E_{\max}]} {\rm CE}^{\mathrm{out}}(e;\theta,\tilde m),\qquad E^{\mathrm{in}}(\theta;\tilde m)=\argmin_{e\in[0,E_{\max}]} {\rm CE}^{\mathrm{in}}(e;\theta,\tilde m)
\]
are convex subsets of $[0, E_{\max}]$. By definition, $\mathcal E(\theta; \tilde{m})$ equals $E^{\mathrm{out}}(\theta;\tilde m)$, or $E^{\mathrm{in}}(\theta;\tilde m)$, or $\mathrm{conv} \big(E^{\mathrm{in}}(\theta;\tilde m) \cup E^{\mathrm{out}}(\theta;\tilde m)\big)$. Thus $\mathcal E(\theta; \tilde{m})$ is also convex.

Take any $y_1, y_2 \in \Gamma(\tilde{m})$. By definition, there exist measurable expected-effort selections $e_1(\cdot)$ and $e_2(\cdot)$ such that $e_k(\theta) \in \mathcal E(\theta; \tilde{m})$ for $\pi$-a.e.~$\theta$ and $y_k = \int_{\Theta} e_k(\theta) d\pi(\theta)$, $k = 1, 2$.

For any $t \in [0, 1]$, define $e_t(\theta) := t e_1(\theta) + (1-t) e_2(\theta)$, measurable with respect to $\theta$. Since $\mathcal E(\theta; \tilde{m})$ is convex for $\pi$-a.e.~$\theta$, we have $e_t(\theta) \in \mathcal E(\theta; \tilde{m})$. Thus
\[
\int_{\Theta} e_t(\theta) d\pi(\theta) \in \Gamma(\tilde{m}).
\]
Then by linearity,
\[
\int_{\Theta} e_t(\theta) d\pi(\theta) = t\int_{\Theta} e_1(\theta) d\pi(\theta) + (1-t)\int_{\Theta} e_2(\theta) d\pi(\theta) = ty_1 + (1-t)y_2 \in \Gamma(\tilde{m}),
\]
which shows that $\Gamma(\tilde{m})$ is convex.

\vspace{5pt}\paragraph{\textbf{Step 3. $\Gamma(\tilde{m}) \subset [0, E_{\max}]$.}} For any $y \in \Gamma(\tilde{m})$, there exists a measurable expected-effort selection $e(\cdot)$ such that
\[
y = \int_{\Theta} e(\theta) d\pi(\theta).
\]
Under Assumption (A1), the feasible effort set is $[0, E_{\max}]$. Then $E^{\mathrm{in}}(\theta; \tilde{m}), E^{\mathrm{out}}(\theta; \tilde{m}) \subset [0, E_{\max}]$, $\mathrm{conv}(E^{\mathrm{in}}(\theta; \tilde{m}) \cup E^{\mathrm{out}}(\theta; \tilde{m})) \subset [0, E_{\max}]$. Thus $\mathcal E(\theta; \tilde{m}) \subset [0, E_{\max}]$. Hence for $\pi$-a.e.~$\theta$,
\[
0 \leq e(\theta) \leq E_{\max}.
\]
By integrating,
\[
0 \leq \int_{\Theta} e(\theta) d\pi(\theta) \leq \int_{\Theta} E_{\max} d\pi(\theta) = E_{\max}.
\]
Thus $0 \leq y \leq E_{\max}$, which implies $\Gamma(\tilde m)\subset[0,E_{\max}]$.

\vspace{5pt}\paragraph{\textbf{Step 4. $\Gamma$ has a closed graph.}}

First, we show that $\tilde{m} \mapsto \mathcal E(\theta;\tilde m)$ has a closed graph. For this, we just need to prove that if $\tilde{m}_n \to \tilde{m}$, $e_n \to e$, $e_n \in \mathcal{E}(\theta; \tilde{m}_n)$, then $e \in \mathcal{E}(\theta; \tilde{m})$.

Fix $\theta$. Under Assumption (A2), $\mu$ and $\lambda$ are continuous, then ${\rm CE}^{\mathrm{in}}(e; \theta, \tilde{m})$ and ${\rm CE}^{\mathrm{out}}(e; \theta, \tilde{m})$ are continuous in $(e, \tilde{m})$. By Berge's Maximum Theorem,
\[
E^{\mathrm{out}}(\theta;\tilde m)=\argmin_{e\in[0,E_{\max}]} {\rm CE}^{\mathrm{out}}(e;\theta,\tilde m),\qquad E^{\mathrm{in}}(\theta;\tilde m)=\argmin_{e\in[0,E_{\max}]} {\rm CE}^{\mathrm{in}}(e;\theta,\tilde m)
\]
are nonempty, compact-valued and upper hemicontinuous in $\tilde{m}$. The value functions
\[
\min_{e\in[0,E_{\max}]} {\rm CE}^{\mathrm{out}}(e;\theta,\tilde m), \qquad \min_{e\in[0,E_{\max}]} {\rm CE}^{\mathrm{in}}(e;\theta,\tilde m)
\]
are continuous in $\tilde{m}$, thus
\[
\Delta {\rm CE}(\theta;\tilde m)
= \min_{e\in E^{\mathrm{in}}(\theta;\tilde m)} {\rm CE}^{\mathrm{in}}(e;\theta,\tilde m)
   - \min_{e\in E^{\mathrm{out}}(\theta;\tilde m)} {\rm CE}^{\mathrm{out}}(e;\theta,\tilde m)
\]
is continuous in $\tilde{m}$.

Let $\Delta_n := \Delta {\rm CE}(\theta; \tilde{m}_n), \Delta := \Delta {\rm CE}(\theta; \tilde{m})$. By continuity, $\Delta_n \to \Delta$. 

If $\Delta < 0$, then there exists $N$ large enough, such that $\Delta_n < 0$ for $n \geq N$. Then for $n \geq N$, $\mathcal E(\theta; \tilde{m}_n) = E^{\mathrm{in}}(\theta; \tilde{m}_n)$, $e_n \in  E^{\mathrm{in}}(\theta; \tilde{m}_n)$. Since $E^{\mathrm{in}}$ is compact-valued and upper hemicontinuous, $\tilde{m} \mapsto E^{\mathrm{in}}(\theta; \tilde{m})$ has a closed graph. $\tilde{m}_n \to \tilde{m}$, $e_n \to e$ with $e_n \in E^{\mathrm{in}}(\theta; \tilde{m}_n)$, then $e \in E^{\mathrm{in}}(\theta; \tilde{m}) = \mathcal E(\theta; \tilde{m})$. Similarly, if $\Delta > 0$, $e \in E^{\mathrm{out}}(\theta; \tilde{m}) = \mathcal E(\theta; \tilde{m})$.

If $\Delta = 0$, consider the sequence $\Delta_n$. If there exist infinitely many $n$ satisfying $\Delta_n < 0$, we can choose a strictly negative subsequence $\{\Delta_{n_k}\}$, meaning $e_{n_k} \in E^{\mathrm{in}}(\theta; \tilde{m}_{n_k})$. Because $\tilde{m}_{n_k} \to \tilde{m}$ and $e_{n_k} \to e$, the closed graph property of $E^{\mathrm{in}}$ yields $e \in E^{\mathrm{in}}(\theta; \tilde{m}) \subset \mathcal E(\theta; \tilde{m})$. Similarly, if there exist infinitely many $n$ such that $\Delta_n > 0$, we obtain $e \in E^{\mathrm{out}}(\theta; \tilde{m}) \subset \mathcal E(\theta; \tilde{m})$. 

If neither case holds, then there must exist infinitely many $n$ such that $\Delta_n = 0$. We can then select a subsequence $e_{n_k}$ with $\Delta_{n_k} = 0$, meaning $e_{n_k} \in \mathcal E(\theta;\tilde m_{n_k}) = \operatorname{conv}\big( E^{\mathrm{in}}(\theta;\tilde m_{n_k}) \cup E^{\mathrm{out}}(\theta;\tilde m_{n_k}) \big)$. This implies we can write $e_{n_k} = t_k a_k + (1-t_k) b_k$, where $t_k \in [0, 1]$, $a_k \in E^{\mathrm{in}}(\theta;\tilde m_{n_k})$, and $b_k \in E^{\mathrm{out}}(\theta;\tilde m_{n_k})$. 

Because the sequence of tuples $(t_k, a_k, b_k)$ is bounded within the compact space $[0,1] \times [0, E_{\max}]^2$, the Bolzano-Weierstrass theorem guarantees we can pass to a common convergent subsequence (which we do not relabel) such that $t_k \to t$, $a_k \to a$, and $b_k \to b$. Furthermore, because $E^{\mathrm{in}}$ and $E^{\mathrm{out}}$ have closed graphs, it follows that $t \in [0,1]$, $a \in E^{\mathrm{in}}(\theta;\tilde m)$, and $b \in E^{\mathrm{out}}(\theta;\tilde m)$. Since $e_{n_k} \to e$, taking the limit on both sides yields $e = ta + (1-t)b$, which proves that $e \in \operatorname{conv}\big( E^{\mathrm{in}}(\theta;\tilde m) \cup E^{\mathrm{out}}(\theta;\tilde m) \big) = \mathcal E(\theta;\tilde m)$.

In all possible cases, $e \in \mathcal E(\theta;\tilde m)$. Therefore, the graph of $\tilde{m} \mapsto \mathcal E(\theta;\tilde m)$ is closed.

Since $\mathcal{E}(\theta; \tilde{m}) \subset [0, E_{\max}]$ is nonempty, closed and convex for every $(\theta, \tilde{m})$, and $\theta \mapsto \mathcal{E}(\theta; \tilde{m})$ is a measurable correspondence as shown in Step 1, the Aumann integral correspondence
\[
\Gamma(\tilde m)=\Big\{\int_\Theta e(\theta)\,d\pi(\theta):\ e(\theta)\in\mathcal E(\theta;\tilde m)\ \text{for $\pi$-a.e.~}\theta\Big\}
\]
is well-defined. We now verify the hypotheses of the closed-graph theorem for Aumann integrals. We apply the theorem on the completion of the population probability space, which does not affect any $\pi$-a.e.~selection or integral. In our setting, the parameter space is the compact metric space $[0, E_{\max}]$, and the value space is $\mathbb R$, which is a separable Banach space. As shown above, for each fixed $\theta$, the correspondence $\tilde m \mapsto \mathcal E(\theta;\tilde m)$ has a closed graph. Moreover, the correspondence $\mathcal E$ is dominated by the compact valued correspondence $K(\theta) := [0, E_{\max}]$, since for all $\theta, \tilde m$,
\[
\mathcal E(\theta;\tilde m)\subset[0,E_{\max}].
\]
The correspondence $K$ is nonempty, weakly compact-valued, and integrably bounded, since $E_{\max}<\infty$ and $\pi$ is a probability measure. Therefore, by \citet[Theorem 3.2]{yannelis1990upper}, which generalizes the classical result of \citet{aumann1976elementary}, the Aumann integral correspondence $\Gamma$ has a closed graph.

\vspace{5pt}\paragraph{\textbf{Step 5. Reconstruction of a mixed equilibrium.}}

By Step 1--4, the correspondence $\Gamma: [0, E_{\max}] \rightrightarrows [0, E_{\max}]$ has nonempty convex values and a closed graph. Since $[0,E_{\max}]$ is a nonempty, compact, and convex subset of $\mathbb R$, the conditions of Kakutani’s fixed-point theorem are satisfied. Thus, there exists a fixed point 
\[
\tilde{m} \in \Gamma(\tilde{m}).
\]
By definition of $\Gamma$, there exists a measurable function $e(\theta)$ such that $e(\theta) \in \mathcal E(\theta; \tilde m)$ for $\pi$-a.e.~$\theta$, and
\[
\tilde m = \int_{\Theta} e(\theta) d\pi(\theta).
\]
From Step 1, we know that $\{\Delta {\rm CE} < 0\}, \{\Delta {\rm CE} > 0\}$ and $\{\Delta {\rm CE} = 0\}$ are all measurable subsets of $\Theta$.

By the Kuratowski--Ryll-Nardzewski measurable selection theorem, we know that there exists measurable selections $\overline{e}^{\mathrm{in}}(\theta) \in E^{\mathrm{in}}(\theta; \tilde{m})$, $\overline{e}^{\mathrm{out}}(\theta) \in E^{\mathrm{out}}(\theta; \tilde{m})$ for $\pi$-a.e.~$\theta$.

On $\{\Delta {\rm CE} < 0\}$, since $\mathcal E(\theta; \tilde m) = E^{\mathrm{in}}(\theta; \tilde m)$, we have $e(\theta) \in E^{\mathrm{in}}(\theta; \tilde m)$. Define
\[
e^{\mathrm{in}}(\theta; \tilde m) := e(\theta), \qquad e^{\mathrm{out}}(\theta; \tilde m) := \overline{e}^{\mathrm{out}}(\theta), \qquad q(\theta; \tilde m) := 1.
\]
On $\{\Delta {\rm CE} > 0\}$, since $\mathcal E(\theta; \tilde m) = E^{\mathrm{out}}(\theta; \tilde m)$, we have $e(\theta) \in E^{\mathrm{out}}(\theta; \tilde m)$. Define
\[
e^{\mathrm{in}}(\theta; \tilde m) := \overline{e}^{\mathrm{in}}(\theta), \qquad e^{\mathrm{out}}(\theta; \tilde m) := e(\theta), \qquad q(\theta; \tilde m) := 0.
\]
On $\{\Delta {\rm CE} = 0\}$, since $E^{\mathrm{in}}(\theta; \tilde m)$ and $E^{\mathrm{out}}(\theta; \tilde m)$ are nonempty compact convex subsets of $\mathbb R$, they are closed intervals. For each $\theta \in \{\Delta {\rm CE} = 0\}$, let $u(\theta)$ be the closest point (metric projection) to $e(\theta)$ in $E^{\mathrm{in}}(\theta; \tilde m)$, and $v(\theta)$ be the closest point to $e(\theta)$ in $E^{\mathrm{out}}(\theta; \tilde m)$. Because the metric projection of a measurable function onto a weakly measurable, compact-valued correspondence is measurable \citep[see, e.g., Chapter 18 in][]{aliprantis2006infinite}, $u(\theta)$ and $v(\theta)$ are well-defined measurable functions, with $u(\theta) \in E^{\mathrm{in}}(\theta; \tilde m)$ and $v(\theta) \in E^{\mathrm{out}}(\theta; \tilde m)$. Define
\[
q(\theta; \tilde m) = \begin{cases}
    \frac{e(\theta) - v(\theta)}{u(\theta) - v(\theta)} & \text{ if } u(\theta) \neq v(\theta),\\[6pt]
    \frac{1}{2} & \text{ if } u(\theta) = v(\theta).
\end{cases}
\]
\footnote{This explicit geometric reconstruction of the mixing probability $q$ relies on the one-dimensional nature of the effort space. In $\mathbb{R}^1$, any point in the convex hull of two disjoint intervals can be represented exactly as a convex combination of its closest projections onto those intervals. In multidimensional spaces, collinearity with the closest projections is not guaranteed. Furthermore, if $u(\theta) = v(\theta)$, then $e(\theta)$ lies exactly in the intersection of both optimal sets, meaning the agent's expected effort is identical under both options. The choice of $q=1/2$ is therefore mathematically valid to ensure a measurable selection, but any $q \in [0,1]$ would suffice.}
Then $q(\theta; \tilde m)$ is measurable, and $q(\theta; \tilde m) \in [0, 1]$. Because $e(\theta) \in \operatorname{conv}\!\Big(E^{\mathrm{in}}(\theta;\tilde m) \cup E^{\mathrm{out}}(\theta;\tilde m)\Big)$, it follows algebraically that $e(\theta) = q(\theta; \tilde m) u(\theta) + (1-q(\theta; \tilde m)) v(\theta)$. We then define $e^{\mathrm{in}}(\theta; \tilde m) := u(\theta)$ and $e^{\mathrm{out}}(\theta; \tilde m) := v(\theta)$.

Combining the three regions, we obtain measurable functions
\[
e^{\mathrm{in}}(\theta; \tilde m) \in E^{\mathrm{in}}(\theta;\tilde m), \qquad e^{\mathrm{out}}(\theta; \tilde m) \in E^{\mathrm{out}}(\theta;\tilde m), \qquad q(\theta; \tilde m) \in [0, 1],
\]
and
\[
e(\theta) = q(\theta; \tilde m) e^{\mathrm{in}}(\theta; \tilde m) + (1-q(\theta; \tilde m)) e^{\mathrm{out}}(\theta; \tilde m),
\]
for $\pi$-a.e.~$\theta$. Therefore
\[
\tilde m = \int_{\Theta} \big[q(\theta; \tilde m) e^{\mathrm{in}}(\theta; \tilde m) + (1-q(\theta; \tilde m)) e^{\mathrm{out}}(\theta; \tilde m)\big] d\pi(\theta).
\]
Thus $(\tilde{m}, \{e^{\mathrm{in}}(\theta; \tilde{m}), e^{\mathrm{out}}(\theta; \tilde{m}), q(\theta; \tilde m)\}_{\theta \in \Theta})$ is a mean-field Nash equilibrium. This proves the theorem.
\qed

    \subsection{Proof of \Cref{thm_NEuniq}} \label{sec: proof of uniqueness}

We establish the uniqueness of the equilibrium by proving that the aggregate mean-field map $\Phi$ is a strict contraction on $[0, E_{\max}]$. The proof proceeds in four steps. Step 1 demonstrates that under the uniqueness assumptions, the individual best-response correspondences are single-valued and the resulting effort mappings are Lipschitz continuous in $\tilde{m}$. Step 2 shows that, for each fixed $\tilde m$, the participation probability $q(\theta; \tilde m)$ is unique for $\pi$-a.e.~$\theta$. Step 3 proves an $L^1(\pi)$-Lipschitz bound for $q(\cdot; \tilde m)$. Finally, Step 4 synthesizes these results to obtain a global Lipschitz constant for $\Phi$ that is strictly less than $1$. The Banach fixed-point theorem then guarantees the uniqueness of the equilibrium up to $\pi$-a.e.~equality.

    \vspace{5pt}\paragraph{\textbf{Step 1. Single-valuedness and Lipschitz continuity of $e^{\mathrm{in}}$, $e^{\mathrm{out}}$ in $\tilde{m}$.}}

    Recall that the best response functions are defined as correspondences:
    \begin{align*}
        E^{\mathrm{out}}(\theta;\tilde m) &= \argmin_{e\in[0,E_{\max}]} {\rm CE}^{\mathrm{out}}(e;\theta,\tilde m),\\
        E^{\mathrm{in}}(\theta;\tilde m) &= \argmin_{e\in[0,E_{\max}]} {\rm CE}^{\mathrm{in}}(e;\theta,\tilde m).
    \end{align*}
    
To isolate the optimization over the combined state, we apply the change of variables $z = e + x\tilde{m}$. Note that substituting $e = z - x\tilde{m}$ into the certainty equivalents yields some terms that depend solely on $\tilde{m}$ and $\theta$ (e.g., $-(c(\theta)-\alpha)x\tilde{m}$). Because these terms do not affect the minimizer with respect to $z$, we omit them and define the following surrogate objective functions:
\begin{align*}
    &f^{\mathrm{out}}(z; \theta) := \mu(z)+\tfrac12\gamma(\theta)\lambda(z)+c(\theta)z,\\
    &f^{\mathrm{in}}(z; \theta) := (c(\theta)-\alpha)z+(1-\beta)\mu(z)+\tfrac12\,\gamma(\theta)(1-\beta)^2\lambda(z).
\end{align*}
Under Assumption \ref{assum_unique}((U2), (U3)), $\mu$ and $\lambda$ are strictly convex on the compact interval $[0, (1+x) E_{\max}]$, and $\gamma(\theta) \geq \gamma_{\min} > 0$. Hence, both $f^{\mathrm{in}}(z; \theta)$ and $f^{\mathrm{out}}(z; \theta)$ are strictly convex on $[0, (1+x)E_{\max}]$. Therefore, for each $\theta$, the unconstrained minimizers over this encompassing domain,
\[
z^{*, \mathrm{in/out}}(\theta) := \argmin_{z \in [0, (1+x) E_{\max}]} f^{\mathrm{in/out}}(z; \theta),
\]
exist and are unique.

Because $f^{\mathrm{in/out}}$ are strictly convex, their constrained minimizer over any sub-interval is simply the metric projection of the unconstrained minimizer $z^{*, \mathrm{in/out}}(\theta)$ onto that sub-interval. The feasible effort domain $e \in [0, E_{\max}]$ corresponds precisely to the moving sub-interval $z \in [x\tilde{m}, E_{\max} + x\tilde{m}]$. By shifting the domain back to the effort space, the optimal efforts can be expressed using the shift-invariance of the projection operator:
    \[
    e^{\mathrm{in/out}}(\theta; \tilde{m}) = \operatorname{Proj}_{[0, E_{\max}]} \big( z^{*, \mathrm{in/out}}(\theta) - x\tilde{m} \big) := \min\{E_{\max}, (z^{*, \mathrm{in/out}}(\theta) - x\tilde{m})^+\}.
    \]
    Hence the correspondences $E^{\mathrm{in}}$ and $E^{\mathrm{out}}$ reduce to singletons, and we denote their unique elements by $e^{\mathrm{in}}(\theta; \tilde{m})$ and $e^{\mathrm{out}}(\theta; \tilde{m})$.

    To prove Lipschitz continuity, the projection operator $\operatorname{Proj}_{[0, E_{\max}]}(\cdot)$ onto a convex set $[0, E_{\max}]$ is $1$-Lipschitz. Since the minimizer $z^{*, \mathrm{in/out}}(\theta)$ is independent of $\tilde{m}$, for any $\tilde{m}_1, \tilde{m}_2 \in [0, E_{\max}]$, we have:
    \begin{align*}
    |e^{\mathrm{in/out}}(\theta; \tilde{m}_1) - e^{\mathrm{in/out}}(\theta; \tilde{m}_2)| &= \big| \operatorname{Proj}_{[0, E_{\max}]} (z^{*, \mathrm{in/out}} - x\tilde{m}_1) - \operatorname{Proj}_{[0, E_{\max}]} (z^{*, \mathrm{in/out}} - x\tilde{m}_2) \big| \\
    &\leq \big| (z^{*, \mathrm{in/out}} - x\tilde{m}_1) - (z^{*, \mathrm{in/out}} - x\tilde{m}_2) \big| \\
    &= x |\tilde{m}_1 - \tilde{m}_2|.
    \end{align*}
    Thus, the individual best response $e^{\mathrm{in/out}}$ is $x$-Lipschitz continuous in $\tilde{m}$ for all $\theta \in \Theta$.
    \vspace{5pt}\paragraph{\textbf{Step 2. For a fixed $\tilde m$, $q(\theta; \tilde m)$ is unique for $\pi$-a.e.~$\theta$.}}

    Fix $\tilde m \in [0, E_{\max}]$. Recall that in \eqref{eq: participation probability} the participation probability $q(\theta; \tilde m)$ satisfies the mixed participation rule:
    \[
    q(\theta; \tilde{m}) = \begin{cases}
        1 & \text{ if }\;\; \Delta{\rm CE}(\theta; \tilde m) < 0,\\
        0 & \text{ if }\;\; \Delta{\rm CE}(\theta; \tilde m) > 0,\\
        \text{any value in } [0, 1] & \text{ if }\;\; \Delta{\rm CE}(\theta; \tilde m) = 0.
    \end{cases}
    \]
    Define
    \begin{align*}
    &z^{\mathrm{in}}(\theta; \tilde{m}) := e^{\mathrm{in}}(\theta; \tilde{m}) + x\tilde{m} = \operatorname{Proj}_{[x \tilde{m}, E_{\max} + x\tilde{m}]} \big( z^{*, \mathrm{in}}(\theta)\big),\\
    &z^{\mathrm{out}}(\theta; \tilde{m}) := e^{\mathrm{out}}(\theta; \tilde{m}) + x\tilde{m} = \operatorname{Proj}_{[x \tilde{m}, E_{\max} + x\tilde{m}]} \big( z^{*, \mathrm{out}}(\theta)\big).
    \end{align*}
    Using this, rewrite the definition of the optimal certainty-equivalent gap in \eqref{eq:CEgap},
    \begin{align*}
    \Delta {\rm CE}(\theta;\tilde m)
    &= {\rm CE}^{\mathrm{in}}(e^{\mathrm{in}}(\theta; \tilde m); \theta, \tilde m) - {\rm CE}^{\mathrm{out}}(e^{\mathrm{out}}(\theta; \tilde m);\theta, \tilde m) \\[4pt]
    &= p + (c-\alpha)e^{\mathrm{in}}(\theta; \tilde{m}) + \tfrac12\gamma\alpha^2\sigma^2 + (1-\beta)\mu(z^{\mathrm{in}}(\theta; \tilde{m})) + \tfrac12\gamma(1-\beta)^2\lambda(z^{\mathrm{in}}(\theta; \tilde{m})) \\
    &\quad -\Big[ \mu(z^{\mathrm{out}}(\theta; \tilde{m})) + ce^{\mathrm{out}}(\theta; \tilde{m}) + \tfrac12\gamma\lambda(z^{\mathrm{out}}(\theta; \tilde{m}))\Big].
    \end{align*}
    where $\theta = (\gamma, c, \sigma)$. Define
    \[
    g(\sigma, \tilde{m}; \gamma, c) := \Delta {\rm CE}(\gamma, c, \sigma; \tilde{m}).
    \]
    Because $z^{\mathrm{in}}(\theta; \tilde m)$ and $z^{\mathrm{out}}(\theta; \tilde m)$ are independent of $\sigma$, the only $\sigma$-dependent term in $g$ is $\tfrac12\gamma\alpha^2\sigma^2$. Then
    \begin{align*}
    \partial_{\sigma} g = \partial_\sigma \Delta {\rm CE}(\theta;\tilde m) = \gamma \alpha^2 \sigma \geq \gamma_{\min} \alpha^2 \sigma_{\min} =: c_{\sigma} > 0.
    \end{align*}
    Therefore, for each fixed $(\gamma, c, \tilde m)$, the map $\sigma \mapsto g(\sigma, \tilde m; \gamma, c)$ is strictly increasing. Thus the equation
    \[
    \Delta {\rm CE}(\gamma, c, \sigma; \tilde{m}) = 0
    \]
    has at most one solution in $\sigma$.

    Now define the indifference set
    \[
    I(\tilde m) := \{\theta \in \Theta: \Delta {\rm CE} (\theta; \tilde m) = 0\}.
    \]
    As shown in \Cref{sec: proof of existence}, $\Delta {\rm CE}(\cdot; \tilde m)$ is measurable, hence $I(\tilde m)$ is measurable. For each fixed $(\gamma, c)$, define
    \[
    I(\tilde m)_{\gamma, c} := \{\sigma \in [\sigma_{\min}, \sigma_{\max}]: (\gamma, c, \sigma) \in I(\tilde m)\},
    \]
    and this set contains at most one point, therefore has one-dimensional Lebesgue measure $0$. By Fubini's theorem, $I(\tilde m)$ has Lebesgue measure $0$. By \Cref{assum_unique}(U5), $\pi$ admits a bounded density $f$, thus
    \[
    \pi(I(\tilde m)) = \int_{I(\tilde m)} f(\theta) d\theta \leq \|f\|_{\infty} \cdot 0 = 0.
    \]
    Thus $I(\tilde m)$ is a $\pi$-null set. Let $q_1(\theta; \tilde m)$ and $q_2(\theta; \tilde m)$ be two measurable participation probabilities satisfying the mixed participation rule, i.e.
    \begin{alignat*}{2}
        &q_1(\theta; \tilde m) = q_2(\theta; \tilde m) = 1, & \qquad \text{ on } \{\Delta {\rm CE} (\theta; \tilde m) < 0\},\\
        &q_1(\theta; \tilde m) = q_2(\theta; \tilde m) = 0, & \qquad \text{ on } \{\Delta {\rm CE} (\theta; \tilde m) > 0\},\\
        &q_1(\theta; \tilde m) \in [0, 1], \ q_2(\theta; \tilde m) \in [0, 1], & \qquad \text{ on } \{\Delta {\rm CE} (\theta; \tilde m) = 0\}.
    \end{alignat*}
    Then $q_1$ and $q_2$ can differ only on $\{\Delta {\rm CE} (\theta; \tilde m) = 0\} = I(\tilde m)$, which is $\pi$-null. Thus $q_1(\theta; \tilde m) = q_2(\theta; \tilde m)$ for $\pi$-a.e.~$\theta$. 
    
    So for each fixed $\tilde m$, the participation probability $q(\theta; \tilde m)$ is unique for $\pi$-a.e.~$\theta$.

    \vspace{5pt}\paragraph{\textbf{Step 3. $q(\cdot; \tilde m)$ is Lipschitz in $L^1(\pi)$.}}
    
    We first establish that the certainty-equivalent gap $\Delta {\rm CE}(\theta; \tilde{m})$ is Lipschitz in $\tilde{m}$.
    
    Recall from Step 2 that
    \begin{align*}
    &z^{\mathrm{in}}(\theta; \tilde{m}) := e^{\mathrm{in}}(\theta; \tilde{m}) + x\tilde{m} = \operatorname{Proj}_{[x \tilde{m}, E_{\max} + x\tilde{m}]} \big( z^{*, \mathrm{in}}(\theta)\big),\\
    &z^{\mathrm{out}}(\theta; \tilde{m}) := e^{\mathrm{out}}(\theta; \tilde{m}) + x\tilde{m} = \operatorname{Proj}_{[x \tilde{m}, E_{\max} + x\tilde{m}]} \big( z^{*, \mathrm{out}}(\theta)\big).
    \end{align*}
    Because the unconstrained minimizer $z^{*, \mathrm{in/out}}(\theta)$ is independent of $\tilde{m}$, and the projection operator is 1-Lipschitz, the shifted argument $z^{\mathrm{in/out}}(\theta; \tilde{m})$ satisfies the bound
    \[
    |z^{\mathrm{in/out}}(\theta; \tilde{m}_1) - z^{\mathrm{in/out}}(\theta; \tilde{m}_2)| \leq x |\tilde{m}_1 - \tilde{m}_2|.
    \]
    Under \Cref{assum_unique}((U1), (U2)), $\mu, \lambda$ are Lipschitz on $[0, (1+x) E_{\max}]$, and $e \in [0, E_{\max}]$, which means $z^{\mathrm{in}}(\theta; \tilde{m}), z^{\mathrm{out}}(\theta; \tilde{m}) \in [0, (1+x) E_{\max}]$ are bounded. Thus
    \begin{align*}
    &|\mu(z^{\mathrm{in}}(\theta; \tilde{m}_1)) - \mu(z^{\mathrm{in}}(\theta; \tilde{m}_2))| \leq L_{\mu} \cdot |z^{\mathrm{in}}(\theta; \tilde{m}_1) - z^{\mathrm{in}}(\theta; \tilde{m}_2)| \leq L_{\mu} \cdot x |\tilde{m}_1 - \tilde{m}_2|,\\
    &|\lambda(z^{\mathrm{in}}(\theta; \tilde{m}_1)) - \lambda(z^{\mathrm{in}}(\theta; \tilde{m}_2))| \leq L_{\lambda} \cdot |z^{\mathrm{in}}(\theta; \tilde{m}_1) - z^{\mathrm{in}}(\theta; \tilde{m}_2)| \leq L_{\lambda} \cdot x |\tilde{m}_1 - \tilde{m}_2|.
    \end{align*}
    Plug these equations back into $\Delta {\rm CE}(\theta;\tilde{m}_1) - \Delta {\rm CE}(\theta;\tilde{m}_2)$, we have
    \begin{align*}
    &|\Delta {\rm CE}(\theta;\tilde{m}_1) - \Delta {\rm CE}(\theta;\tilde{m}_2)|\\
    &\quad \leq (c_{\max} + \alpha) x |\tilde{m}_1 - \tilde{m}_2| + (1-\beta) L_{\mu} \cdot x |\tilde{m}_1 - \tilde{m}_2| + \frac{1}{2} \gamma_{\max} (1-\beta)^2 L_{\lambda} \cdot x |\tilde{m}_1 - \tilde{m}_2|\\
    &\qquad + L_{\mu} \cdot x |\tilde{m}_1 - \tilde{m}_2| + c_{\max} x |\tilde{m}_1 - \tilde{m}_2| + \frac{1}{2} \gamma_{\max} L_{\lambda} \cdot x |\tilde{m}_1 - \tilde{m}_2|\\
    &\quad = x \Big( 2c_{\max} + \alpha + (2-\beta) L_{\mu} + \frac{1}{2} \gamma_{\max} \big((1-\beta)^2+1\big) L_{\lambda}\Big) \cdot |\tilde{m}_1 - \tilde{m}_2|.
    \end{align*}
for any $\theta \in \Theta$. Let $C_m := x \Big( 2c_{\max} + \alpha + (2-\beta) L_{\mu} + \frac{1}{2} \gamma_{\max} \big((1-\beta)^2+1\big) L_{\lambda}\Big)$,
\[
|\Delta {\rm CE}(\theta;\tilde{m}_1) - \Delta {\rm CE}(\theta;\tilde{m}_2)| \leq C_m \cdot |\tilde{m}_1 - \tilde{m}_2|, \qquad \forall \theta \in \Theta.
\]

We next introduce the participation threshold. From Step 2, we know that $\sigma \mapsto \Delta {\rm CE}(\gamma, c, \sigma; \tilde m)$ is continuous and strictly increasing, and so whenever $\Delta {\rm CE}(\gamma, c, \sigma; \tilde m) = 0$ has a solution, it is unique. Hence whenever $\Delta {\rm CE}(\gamma,c,\sigma_{\min};\tilde m)<0<\Delta {\rm CE}(\gamma,c,\sigma_{\max};\tilde m)$, there exists $\sigma^*(\gamma, c; \tilde{m}) \in [\sigma_{\min}, \sigma_{\max}]$ such that
\[
g(\sigma^*(\gamma, c; \tilde{m}), \tilde{m}; \gamma, c) = \Delta {\rm CE}(\gamma, c, \sigma^*(\gamma, c; \tilde{m}); \tilde{m}) = 0.
\]
To cover the cases that the solution lies outside the feasible interval $[\sigma_{\min}, \sigma_{\max}]$, we define
\begin{equation} \label{eq: truncated sigma}
\hat{\sigma}^*(\gamma,c;\tilde m) = 
\begin{cases}
\sigma_{\min}, & \text{if } \Delta {\rm CE}(\gamma,c,\sigma_{\min};\tilde m)\ge 0,\\[4pt]
\sigma^*(\gamma, c; \tilde{m}), & \text{if } \Delta {\rm CE}(\gamma,c,\sigma_{\min};\tilde m)<0<\Delta {\rm CE}(\gamma,c,\sigma_{\max};\tilde m),\\[4pt]
\sigma_{\max}, & \text{if } \Delta {\rm CE}(\gamma,c,\sigma_{\max};\tilde m)\le 0.
\end{cases}
\end{equation}
Fixed $\gamma, c$, for any given $\tilde{m}_1, \tilde{m}_2$, let $\sigma_1 = \hat \sigma^*(\gamma, c; \tilde{m}_1)$ and $\sigma_2 = \hat \sigma^*(\gamma, c; \tilde{m}_2)$. WLOG, we assume $\sigma_1 \geq \sigma_2$. If both $\sigma_1, \sigma_2 \in (\sigma_{\min}, \sigma_{\max})$,
\begin{align*}
    0 &= \Delta {\rm CE}(\gamma, c, \sigma_1; \tilde{m}_1) - \Delta {\rm CE}(\gamma, c, \sigma_2; \tilde{m}_2)\\
    &= \big[\Delta {\rm CE}(\gamma, c, \sigma_1; \tilde{m}_1) - \Delta {\rm CE}(\gamma, c, \sigma_1; \tilde{m}_2)\big] + \big[\Delta {\rm CE}(\gamma, c, \sigma_1; \tilde{m}_2) - \Delta {\rm CE}(\gamma, c, \sigma_2; \tilde{m}_2)\big],
\end{align*}
where $|\Delta {\rm CE}(\gamma, c, \sigma_1; \tilde{m}_1) - \Delta {\rm CE}(\gamma, c, \sigma_1; \tilde{m}_2)| \leq C_m |\tilde{m}_1 - \tilde{m}_2|$ and $|\Delta {\rm CE}(\gamma, c, \sigma_1; \tilde{m}_2) - \Delta {\rm CE}(\gamma, c, \sigma_2; \tilde{m}_2)| \geq c_{\sigma}(\sigma_1 - \sigma_2)$. Thus
\[
0 \geq -C_m |\tilde{m}_1 - \tilde{m}_2| + c_{\sigma} (\sigma_1 - \sigma_2),
\]
we obtain $\sigma_1-\sigma_2 \le (C_m/c_\sigma)|\tilde m_1-\tilde m_2|$.

If at least one of $\sigma_1,\sigma_2$ lies on the boundary, we argue similarly but replace the equalities by the corresponding sign conditions. For instance, if $\sigma_1=\sigma_{\max}$, then $\Delta {\rm CE}(\gamma,c,\sigma_{\max};\tilde m_1)\le 0$, whereas $\Delta {\rm CE}(\gamma,c,\sigma_2;\tilde m_2)\ge 0$ when $\sigma_2=\sigma_{\min}$, and $\Delta {\rm CE}(\gamma,c,\sigma_2;\tilde m_2)=0$ when $\sigma_2\in(\sigma_{\min},\sigma_{\max})$. In all cases we have
\begin{align*}
0 &\ge \Delta {\rm CE}(\gamma,c,\sigma_1;\tilde m_1)-\Delta {\rm CE}(\gamma,c,\sigma_2;\tilde m_2)\\
&=\big[\Delta {\rm CE}(\gamma,c,\sigma_1;\tilde m_1)-\Delta {\rm CE}(\gamma,c,\sigma_1;\tilde m_2)\big]
+\big[\Delta {\rm CE}(\gamma,c,\sigma_1;\tilde m_2)-\Delta {\rm CE}(\gamma,c,\sigma_2;\tilde m_2)\big],
\end{align*}
and the same two bounds apply because $\sigma_1,\sigma_2\in[\sigma_{\min},\sigma_{\max}]$. This again yields $\sigma_1-\sigma_2 \le \frac{C_m}{c_\sigma}\,|\tilde m_1-\tilde m_2|$.

Therefore,
\begin{equation} \label{eq:Csigma}
|\hat{\sigma}^*(\gamma, c; \tilde{m}_1) - \hat{\sigma}^*(\gamma, c; \tilde{m}_2)| = \sigma_1 - \sigma_2 \leq \frac{C_m}{c_{\sigma}} |\tilde{m}_1 - \tilde{m}_2|.
\end{equation}

Now pick $q(\theta; \tilde m_i)$ as any measurable participation probability satisfying the mixed participation rule, $i = 1, 2$. For any fixed $(\gamma, c)$, $q(\theta; \tilde m_i)$ can only change value at the threshold $\hat{\sigma}^*(\gamma, c; \tilde m_i)$, up to the single indifference point. Therefore, $q(\gamma, c, \sigma; \tilde m_1)$ and $q(\gamma, c, \sigma; \tilde m_2)$ can differ only when $\sigma$ lies between $\hat{\sigma}^*(\gamma, c; \tilde m_1)$ and $\hat{\sigma}^*(\gamma, c; \tilde m_2)$, and $|q(\gamma, c, \sigma; \tilde m_1) - q(\gamma, c, \sigma; \tilde m_2)| \leq 1$. Thus
\[
\int_{\sigma_{\min}}^{\sigma_{\max}} |q(\gamma, c, \sigma; \tilde m_1) - q(\gamma, c, \sigma; \tilde m_2)| f(\gamma, c, \sigma) d\sigma \leq \|f\|_{\infty} |\hat{\sigma}^*(\gamma, c; \tilde m_1) - \hat{\sigma}^*(\gamma, c; \tilde m_2)|.
\]
Write $\mathcal G = [\gamma_{\min}, \gamma_{\max}], \mathcal C = [0, c_{\max}]$ as the feasible region of $\gamma$ and $c$. Integrating over $\gamma$ and $c$,
\begin{align*}
    \int_{\Theta} |q(\theta; \tilde m_1) - q(\theta; \tilde m_2)| d\pi(\theta) &\leq \int_{\mathcal G}\int_{\mathcal C} \|f\|_{\infty} \ |\hat{\sigma}^*(\gamma,c;\tilde{m}_1) - \hat{\sigma}^*(\gamma,c;\tilde{m}_2)|\ dc\,d\gamma\\
    &\le \int_{\mathcal G}\int_{\mathcal C} C_f \cdot \frac{C_m}{c_\sigma} \ |\tilde{m}_1 - \tilde{m}_2|\ dc\,d\gamma\\
    &= \frac{C_f \cdot C_m}{c_\sigma} |\mathcal G| \cdot |\mathcal{C}| \cdot |\tilde{m}_1 - \tilde{m}_2|,
\end{align*}
by \Cref{assum_unique}((U5)), $\|f\|_{\infty}\le C_f$. Here $|\mathcal G| = \gamma_{\max} - \gamma_{\min}$, $|\mathcal{C}| = c_{\max} - 0$ are the Lebesgue measure of the support set for $\gamma$ and $c$. Define
\begin{equation}\label{eq:Crho}
C_q := \frac{C_f \cdot C_m}{c_\sigma} |\mathcal G| \cdot |\mathcal{C}|,
\end{equation}
then
\begin{equation} \label{eq: q lip}
\int_{\Theta} |q(\theta; \tilde m_1) - q(\theta; \tilde m_2)| d\pi(\theta) \le C_q|\tilde m_1-\tilde m_2|.
\end{equation}

    \vspace{5pt}\paragraph{\textbf{Step 4. Contraction property of the mean-field map $\Phi$.}}

    In this step we show that the aggregate response map $\Phi: [0, E_{\max}] \to [0, E_{\max}]$ is Lipschitz continuous with constant strictly smaller than $1$. This establishes that $\Phi$ is a contraction, which yields uniqueness of the fixed point.
    
    Recall from \eqref{eq: Phi},
    \[
    \Phi(\tilde{m}) = \int_{\Theta} \Big[ q(\theta; \tilde m) e^{\mathrm{in}}(\theta; \tilde m) + (1-q(\theta; \tilde m)) e^{\mathrm{out}}(\theta; \tilde m) \big] d\pi(\theta).
    \]
    By Step 2, this map is well defined, since different admissible versions of $q(\theta; \tilde m)$ can differ only on a $\pi$-null set.
    
    Let $\tilde{m}_1, \tilde{m}_2 \in [0, E_{\max}]$. Then,
\begin{align*}
&\Phi(\tilde{m}_1) - \Phi(\tilde{m}_2)\\
&\quad = \int_\Theta (e^{\mathrm{out}}(\theta;\tilde{m}_1) - e^{\mathrm{out}}(\theta;\tilde{m}_2)) d\pi(\theta)\\
&\qquad + \int_\Theta \Big[\big(e^{\mathrm{in}}(\theta;\tilde{m}_1) - e^{\mathrm{out}}(\theta;\tilde{m}_1)\big)q(\theta;\tilde{m}_1) - \big(e^{\mathrm{in}}(\theta;\tilde{m}_2) - e^{\mathrm{out}}(\theta;\tilde{m}_2)\big)q(\theta;\tilde{m}_2)\Big] d\pi(\theta)\\
&\quad =: I_1 + I_2.
\end{align*}

For $I_1$, from Step 1, we have
\[
|e^{\mathrm{out}}(\theta;\tilde{m}_1) - e^{\mathrm{out}}(\theta;\tilde{m}_2)| \leq x|\tilde{m}_1 - \tilde{m}_2|.
\]
Then
\[
|I_1|\le \int_\Theta x|\tilde m_1-\tilde m_2|\,d\pi(\theta)=x|\tilde m_1-\tilde m_2|.
\]

For $I_2$, let $\Delta e(\theta; \tilde{m}) := e^{\mathrm{in}}(\theta;\tilde{m}) - e^{\mathrm{out}}(\theta;\tilde{m})$. Decompose $I_2$ as
\begin{align*}
    I_2 = &\int_\Theta \Big[\Delta e(\theta; \tilde{m}_1) q(\theta;\tilde{m}_1) - \Delta e(\theta; \tilde{m}_2) q(\theta;\tilde{m}_2)\Big] d\pi(\theta)\\
    = &\int_\Theta \Big[\big(\Delta e(\theta; \tilde{m}_1) - \Delta e(\theta; \tilde{m}_2)\big) q(\theta;\tilde{m}_1)\Big] d\pi(\theta)\\
    &+ \int_\Theta \Big[\Delta e(\theta; \tilde{m}_2) \big(q(\theta;\tilde{m}_1) - q(\theta;\tilde{m}_2)\big)\Big] d\pi(\theta)\\
    =: & I_{2,1} + I_{2,2}.
\end{align*}
For $I_{2,1}$, by Step 1,
\begin{align*}
|\Delta e(\theta; \tilde{m}_1) - \Delta e(\theta; \tilde{m}_2)| &\leq |e^{\mathrm{in}}(\theta; \tilde{m}_1) - e^{\mathrm{in}}(\theta; \tilde{m}_2)| + |e^{\mathrm{out}}(\theta; \tilde{m}_1) - e^{\mathrm{out}}(\theta; \tilde{m}_2)|\\
&\leq 2x |\tilde{m}_1 - \tilde{m}_2|.
\end{align*}
Since $0 \leq q(\theta;\tilde{m}_1) \leq 1$,
\[
|I_{2,1}| \leq \int_{\Theta} 2x |\tilde{m}_1 - \tilde{m}_2| q(\theta;\tilde{m}_1) d\pi(\theta) \leq 2x |\tilde{m}_1 - \tilde{m}_2|.
\]
For $I_{2,2}$, by the definition of effort levels, $|\Delta e(\theta; \tilde{m}_2)| = |e^{\mathrm{in}}(\theta; \tilde{m}_2) - e^{\mathrm{out}}(\theta; \tilde{m}_2)| \leq E_{\max}$. Therefore,
\begin{align*}
|I_{2,2}| &\leq \int_\Theta |\Delta e(\theta; \tilde{m}_2)| \cdot |q(\theta; \tilde{m}_1) - q(\theta; \tilde{m}_2)| d\pi(\theta) \\
&\leq E_{\max} \int_\Theta |q(\theta; \tilde{m}_1) - q(\theta; \tilde{m}_2)| d\pi(\theta)\\
&\leq E_{\max} \cdot C_q |\tilde m_1 - \tilde m_2|,
\end{align*}
where the last inequality follows from \eqref{eq: q lip}.
Combining the bounds for $I_1$, $I_{2,1}$ and $I_{2,2}$, we have
\begin{align*}
    |\Phi(\tilde{m}_1) - \Phi(\tilde{m}_2)| &\leq |I_1| + |I_{2,1}| + |I_{2,2}|\\
    &\leq x |\tilde{m}_1 - \tilde{m}_2| + 2x |\tilde{m}_1 - \tilde{m}_2| + E_{\max} \cdot C_q |\tilde{m}_1 - \tilde{m}_2|\\
    &= (3x + E_{\max} \cdot C_q) |\tilde{m}_1 - \tilde{m}_2|.
\end{align*}
Under \Cref{assum_unique}((U6)),
\[
3x + E_{\max} \cdot C_q < 1.
\]
Therefore, $\Phi$ is a contraction mapping on the complete metric space $[0, E_{\max}]$. By the Banach fixed point theorem, $\Phi$ admits a unique fixed point. By Step 1, the optimal efforts $e^{\mathrm{in}}(\theta; \tilde m^*)$ and $e^{\mathrm{out}}(\theta; \tilde m^*)$ are uniquely defined. By Step 2, the participation probability $q(\theta; \tilde m^*)$ is unique for $\pi$-a.e.~$\theta$. Therefore, the associated mean-field equilibrium is unique up to $\pi$-a.e.~equality. In particular, all aggregate equilibrium quantities are unique.

\qed

\subsection{Proof of \Cref{cor: worst best}} \label{sec: proof of cor}

Fix any admissible contract $k \in \mathcal K$. By \Cref{assum_stack}(I2), the lower-level mean-field game admits at least one mean-field Nash equilibrium, so $\mathcal M(k)$ is nonempty. Hence the equilibrium payoff set $\mathbb J(k) = \{J(k;M): M \in \mathcal M(k)\}$ is nonempty for any $k \in \mathcal K$.

We first show that the insurer payoff is uniformly bounded. Let $M = (\tilde m, \{e^{\mathrm{in}}(\theta), e^{\mathrm{out}}(\theta), q(\theta)\}_{\theta \in \Theta}) \in \mathcal M(k)$. By the bounded-effort assumption, we have $0 \leq e^{\mathrm{in}}(\theta), \tilde{m} \leq E_{\max}$. Thus
\[
0 \leq e^{\mathrm{in}}(\theta) + x\tilde{m} \leq (1+x)E_{\max}.
\]
Furthermore, the participation probability satisfies $0 \leq q(\theta) \leq 1$.

Because $k \in \mathcal K$, Assumption \ref{assum_stack}(I1) dictates that the contract parameters are bounded:
\[
0 \leq p \leq p_{\max}, \quad 0 \leq \alpha \leq \alpha_{\max}, \quad 0 \leq \beta \leq \beta_{\max}.
\]
Because the cost function $\mu$ is continuous on the compact interval $[0, (1+x) E_{\max}]$, it is bounded. We define its maximum absolute value as:
\[
M_{\mu} := \sup_{z \in [0, (1+x) E_{\max}]} |\mu(z)| < \infty.
\]
Applying these uniform bounds to the integrand, the absolute value of the insurer's payoff satisfies:
\begin{align*}
    |J(k; M)| &= \Big|\int_{\Theta} \Big(p - \alpha e^{\mathrm{in}}(\theta)- \beta \mu(e^{\mathrm{in}}(\theta) + x \tilde{m})\Big) q(\theta) d\pi(\theta)\Big|\\
    &\leq \int_{\Theta} (p_{\max} + \alpha_{\max} E_{\max} + \beta_{\max} M_{\mu}) q(\theta) d\pi(\theta)\\
    &\leq p_{\max} + \alpha_{\max} E_{\max} + \beta_{\max} M_{\mu}\\
    &=: C_J.
\end{align*}
Thus the insurer payoff $J(k; M)$ is uniformly bounded over all admissible contracts and all lower-level equilibria. For every $k\in\mathcal K$, the set $\mathbb J(k) = \{J(k;M):M\in\mathcal M(k)\}$ is nonempty and contained in $[-C_J, C_J]$. Therefore,
\[
\mathcal J^{\mathrm{worst}}(k)
=
\inf_{M\in\mathcal M(k)} J(k;M),
\qquad
\mathcal J^{\mathrm{best}}(k)
=
\sup_{M\in\mathcal M(k)} J(k;M)
\]
are finite real numbers for every $k\in\mathcal K$, and both are bounded above by $C_J$. Hence the hypotheses of \Cref{thm: insurer} are satisfied for both criteria. Applying \Cref{thm: insurer} with $\mathcal J=\mathcal J^{\mathrm{worst}}$ and $\mathcal J=\mathcal J^{\mathrm{best}}$ gives the desired $\varepsilon$-optimal contracts.
\qed

\subsection{Proof of \Cref{prop: uniq equilibrium}} \label{sec: proof of uniq equilibrium insurer}

To prove the existence of an optimal contract under the unique lower-level equilibrium response, we proceed in four steps. Step 1 demonstrates that under the uniqueness assumptions, the insurer's payoff evaluates to a well-defined, single-valued function on the admissible contract space. Step 2 proves that the lower-level equilibrium objects (the mean-field, individual efforts, and participation probabilities) are continuous with respect to the contract. Step 3 utilizes these results to establish the continuity of the insurer's objective function. Finally, Step 4 invokes the Weierstrass extreme value theorem to guarantee the existence of an optimal contract.
\vspace{5pt}\paragraph{\textbf{Step 1. The unique insurer payoff is well-defined.}}

Fix any admissible contract $k = (p, \alpha, \beta) \in \mathcal K^{\mathrm{uniq}}$. Since \Cref{assum_unique} holds for every admissible $k$, thus by \Cref{thm_NEuniq} the lower-level mean-field game among agents admits a unique mean-field Nash equilibrium up to $\pi$-a.e.~equality. We denote this equilibrium by
\[
M(k) = (\tilde m(k), \{e^{\mathrm{in}}(\theta; k), e^{\mathrm{out}}(\theta; k), q(\theta;k)\}_{\theta \in \Theta}).
\]
Here we choose any representative of $q(\cdot; k)$ within its $\pi$-a.e.~equivalence class, thus this equilibrium is unique up to $\pi$-a.e.~equality. This ambiguity does not affect any payoff-relevant object, because the insurer payoff depends on $q$ only through the integral
\[
J(k;M(k))
=
\int_\Theta
\Big(
p-\alpha e^{\mathrm{in}}(\theta;k)
-\beta\mu(e^{\mathrm{in}}(\theta;k)+x\tilde m(k))
\Big)
q(\theta;k)\,d\pi(\theta).
\]
Hence changing $q(\cdot;k)$ on a $\pi$-null set leaves $J(k;M(k))$ unchanged. The insurer's payoff evaluated at the equilibrium can be represented as
\[
\mathcal J^{\mathrm{uniq}}(k) = J(k; M(k)),
\]
which is single-valued and well-defined for any $k \in \mathcal K^{\mathrm{uniq}}$.

\vspace{5pt}\paragraph{\textbf{Step 2. The lower-level equilibrium objects are continuous on $k \in \mathcal K^{\mathrm{uniq}}$.}}

Fix a sequence $k_n := (p_n, \alpha_n, \beta_n) \in \mathcal K^{\mathrm{uniq}}$, $k_0 := (p_0, \alpha_0, \beta_0) \in \mathcal K^{\mathrm{uniq}}$, and suppose that $k_n \to k_0$ as $n \to \infty$. For each $n$, write
\[
M(k_n) = \Big(\tilde{m}(k_n), \{e^{\mathrm{in}}(\theta; k_n), e^{\mathrm{out}}(\theta; k_n), q(\theta; k_n)\}_{\theta \in \Theta}\Big),
\]
and similar for $M(k_0)$. We now show that $\tilde{m}(k_n) \to \tilde{m}(k_0)$, $e^{\cdot}(\theta; k_n) \to e^{\cdot}(\theta; k_0)$, and $q(\theta; k)$ is $L^1$-continuous on $k$.

We first prove the continuity of the equilibrium mean effort $k \mapsto \tilde{m}(k)$. For each fixed admissible contract $k \in \mathcal K^{\mathrm{uniq}}$, denote $\Phi_{k}: [0, E_{\max}] \to [0, E_{\max}]$ as the aggregate response map of the lower-level mean-field game,
\begin{equation} \label{eq:Phik}
\Phi_k(\tilde m) = \int_{\Theta} \big[q(\theta; \tilde{m}, k) e^{\mathrm{in}}(\theta;\tilde{m},k) + (1 - q(\theta; \tilde{m}, k)) e^{\mathrm{out}}(\theta;\tilde{m},k)\big] d\pi(\theta).
\end{equation}
For notational clarity, when both a candidate mean-field $\tilde m$ and a contract $k$ are fixed, we write $e^{\mathrm{in}}(\theta;\tilde m,k)$, $e^{\mathrm{out}}(\theta;\tilde m,k)$, and $q(\theta;\tilde m,k)$ for the corresponding individual best-response objects and participation rule. Along the unique equilibrium induced by contract $k$, we write $\tilde m(k)$, $e^{\mathrm{in}}(\theta;k)=e^{\mathrm{in}}(\theta;\tilde m(k),k)$, $e^{\mathrm{out}}(\theta;k)=e^{\mathrm{out}}(\theta;\tilde m(k),k)$, and $q(\theta;k)=q(\theta;\tilde m(k),k)$.

By \Cref{thm_NEuniq}, the equilibrium effort is the unique fixed point of $\Phi_{k}$, i.e.
\[
\tilde{m}(k) = \Phi_{k}\big(\tilde{m}(k)\big).
\]
Let
\[
\tilde{m}_n := \tilde{m}(k_n), \qquad \tilde{m}_0 := \tilde{m}(k_0).
\] 
Since the uniqueness conditions in $\Cref{assum_unique}$ hold uniformly over $\mathcal K^{\mathrm{uniq}}$, there exists a constant $C_{\mathcal K^{\mathrm{uniq}}} := \sup_{k \in \mathcal K^{\mathrm{uniq}}} (3x + E_{\max} C_q(k)) < 1$, such that for every $k \in \mathcal K^{\mathrm{uniq}}$,
\[
|\Phi_{k}(m_1) - \Phi_{k}(m_2)| \leq C_{\mathcal K^{\mathrm{uniq}}} |m_1 - m_2|, \qquad \forall m_1, m_2 \in [0, E_{\max}].
\]
Then for any $n$,
\begin{align*}
    |\tilde{m}_n - \tilde{m}_0| &= |\Phi_{k_n}(\tilde{m}_n) - \Phi_{k_0}(\tilde{m}_0)|\\
    &\leq |\Phi_{k_n}(\tilde{m}_n) - \Phi_{k_n}(\tilde{m}_0)| + |\Phi_{k_n}(\tilde{m}_0) - \Phi_{k_0}(\tilde{m}_0)|\\
    &\leq C_{\mathcal K^{\mathrm{uniq}}} |\tilde{m}_n - \tilde{m}_0| + |\Phi_{k_n}(\tilde{m}_0) - \Phi_{k_0}(\tilde{m}_0)|.
\end{align*}
Hence,
\begin{equation} \label{eq:mtophi}
(1-C_{\mathcal K^{\mathrm{uniq}}}) |\tilde{m}_n - \tilde{m}_0| \leq |\Phi_{k_n}(\tilde{m}_0) - \Phi_{k_0}(\tilde{m}_0)|.
\end{equation}
Thus it suffices to show that $\Phi_{k_n}(\tilde{m}_0) \to \Phi_{k_0}(\tilde{m}_0)$. 

Fix $\tilde m_0$. The individual certainty-equivalent function ${\rm CE}^{\mathrm{out}}(e; \theta, \tilde{m}_0)$ does not depend on the contract, and ${\rm CE}^{\mathrm{in}}(e; \theta, \tilde{m}_0, k)$ depends continuously on $(k, e)$ for any fixed $\theta$. By Berge's Maximum Theorem \citep[Theorem 17.31]{aliprantis2006infinite}, the argmin correspondence of ${\rm CE}^{\mathrm{in}}$ is upper hemicontinuous. Under the strictly convexity assumption in \Cref{assum_unique}, the minimizer is unique, the argmin correspondence is single-valued, hence,
\begin{equation} \label{eq: e convergence}
e^{\mathrm{in}}(\theta;\tilde{m}_0,k_n) \to e^{\mathrm{in}}(\theta;\tilde{m}_0,k_0), \qquad e^{\mathrm{out}}(\theta;\tilde{m}_0, k) \text{ does not depend on the contract } k
\end{equation}
for $\pi$-a.e.~$\theta$.

Next we show that
\[
\int_{\Theta} |q(\theta; \tilde{m}_0, k_n) - q(\theta; \tilde{m}_0, k_0)| d\pi(\theta) \to 0.
\]
We can deduce directly from \Cref{sec: proof of uniqueness} that $\Delta {\rm CE}(\theta;\tilde{m}_0,k)$ is continuous in $k$ for $\pi$-a.e.~$\theta$. Moreover, for each fixed $(\gamma, c, k)$, the map
\[
\sigma \mapsto \Delta {\rm CE}(\gamma, c, \sigma; \tilde m_0, k)
\]
is continuously and strictly increasing, as proved in \Cref{sec: proof of uniqueness}. Hence the participation region is characterized by a unique truncated threshold from \eqref{eq: truncated sigma}
\[
\hat{\sigma}^*(\gamma, c; \tilde m_0, k) \in [\sigma_{\min}, \sigma_{\max}],
\]
and this threshold depends continuously on $k$ by the same argument in \Cref{sec: proof of uniqueness}: for fixed $(\gamma, c)$, define $g(\sigma, k) := \Delta {\rm CE}(\gamma, c, \sigma; \tilde m_0, k)$. Since $g$ is continuous on $\sigma, k$, on the compact set $[\sigma_{\min}, \sigma_{\max}] \times \mathcal K^{\mathrm{uniq}}$, it is uniformly continuous. Therefore, as $k_n \to k_0$,
\[
\sup_{\sigma \in [\sigma_{\min}, \sigma_{\max}]} |g(\sigma, k_n) - g(\sigma, k_0)| \to 0.
\]
Combining this with the uniform lower bound $\partial_{\sigma} g(\sigma, k) \geq c_{\sigma}^{\mathcal K^{\mathrm{uniq}}} := \inf_{k \in \mathcal K^{\mathrm{uniq}}} c_{\sigma}(k) \geq
\gamma_{\min}\alpha_{\min}^2\sigma_{\min} > 0$, the same argument as in \Cref{sec: proof of uniqueness} yields
\[
\hat{\sigma}^*(\gamma, c; \tilde m_0, k_n) \to \hat{\sigma}^*(\gamma, c; \tilde m_0, k_0).
\]
$q(\theta; \tilde{m}_0, k_n)$ and $q(\theta; \tilde{m}_0, k_0)$ can differ only on the slab between $\hat{\sigma}^*(\gamma, c; \tilde m_0, k_n)$ and $\hat{\sigma}^*(\gamma, c; \tilde m_0, k_0)$, which is up to $\pi$-null indifference sets. Since $\pi$ admits a bounded density,
\[
\int_{\Theta} |q(\theta; \tilde{m}_0, k_n) - q(\theta; \tilde{m}_0, k_0)| d\pi(\theta) \to 0.
\]
Now write
\[
\Delta e(\theta; \tilde m, k) := e^{\mathrm{in}}(\theta; \tilde m, k) - e^{\mathrm{out}}(\theta; \tilde m, k).
\]
Recall the definition of $\Phi_k$ from \eqref{eq:Phik}, we have
\begin{align*}
    |\Phi_{k_n}(\tilde m_0) - \Phi_{k_0}(\tilde m_0)| \leq &\Big|\int_{\Theta} (e^{\mathrm{out}}(\theta; \tilde m_0, k_n) -  e^{\mathrm{out}}(\theta; \tilde m_0, k_0)) d\pi(\theta)\Big|\\
    &+ \Big|\int_{\Theta} \big[q(\theta; \tilde{m}_0, k_n) \Delta e(\theta; \tilde m_0, k_n) - q(\theta; \tilde{m}_0, k_0) \Delta e(\theta; \tilde m_0, k_0)\big] d\pi(\theta)\Big|\\
    \leq &\Big|\int_{\Theta} (e^{\mathrm{out}}(\theta; \tilde m_0, k_n) -  e^{\mathrm{out}}(\theta; \tilde m_0, k_0)) d\pi(\theta)\Big|\\
    &+ \Big|\int_{\Theta} q(\theta; \tilde{m}_0, k_n) (\Delta e(\theta; \tilde m_0, k_n) - \Delta e(\theta; \tilde m_0, k_0)) d\pi(\theta)\Big|\\
    &+ \Big|\int_{\Theta} (q(\theta; \tilde{m}_0, k_n) - q(\theta; \tilde{m}_0, k_0)) \Delta e(\theta; \tilde m_0, k_0) d\pi(\theta)\Big|.
\end{align*}
The first term is equal to $0$ by \eqref{eq: e convergence}, since for fixed $\tilde m_0$ the outside effort does not depend on $k$. For the second term, since $|q(\theta; \tilde m_0, k_n)| \leq 1$, and $\Delta e(\theta; \tilde m_0, k_n) \to \Delta e(\theta; \tilde m_0, k_0)$ for $\pi$-a.e.~$\theta$, by dominated convergence it goes to $0$ too. The third term converges to $0$ since $|\Delta e(\theta; \tilde m_0, k_0)| \leq E_{\max}$, and $\int_{\Theta} |q(\theta; \tilde{m}_0, k_n) - q(\theta; \tilde{m}_0, k_0)| d\pi(\theta) \to 0$. Thus
\[
\Phi_{k_n}(\tilde{m}_0) \to \Phi_{k_0}(\tilde{m}_0).
\]
Returning to \eqref{eq:mtophi}, we conclude that $\tilde{m}_n \to \tilde{m}_0$. Thus the equilibrium mean effort $k \mapsto \tilde{m}(k)$ is continuous on $\mathcal K^{\mathrm{uniq}}$.

We now prove continuity of the effort functions and $L^1(\pi)$-continuity of the participation probability $q$. By definition of the selected equilibrium objects,
\[
e^{\mathrm{in}}(\theta; k) = e^{\mathrm{in}}(\theta;\tilde{m}(k),k), \qquad e^{\mathrm{out}}(\theta; k) = e^{\mathrm{out}}(\theta;\tilde{m}(k),k),
\]
and
\[
q(\theta; k) = q(\theta; \tilde{m}(k), k),
\]
where $e^{\mathrm{in}}$, $e^{\mathrm{out}}$ and $q$ are the lower-level equilibrium objects introduced in \Cref{defn: equilibrium}.

Once $\tilde{m}(k)$ is continuous in $k$, it follows that the individual optimal efforts $e^{\mathrm{in}}(\theta; \tilde{m}(k),k)$ and $e^{\mathrm{out}}(\theta; \tilde{m}(k), k)$ are continuous in $k$ for $\pi$-a.e.~$\theta$, i.e.
\begin{align*}
    &e^{\mathrm{in}}(\theta; k_n) = e^{\mathrm{in}}(\theta;\tilde{m}(k_n),k_n) \to e^{\mathrm{in}}(\theta;\tilde{m}(k_0),k_0) = e^{\mathrm{in}}(\theta; k_0),\\
    &e^{\mathrm{out}}(\theta; k_n) = e^{\mathrm{out}}(\theta;\tilde{m}(k_n),k_n) \to e^{\mathrm{out}}(\theta;\tilde{m}(k_0),k_0) = e^{\mathrm{out}}(\theta; k_0).
\end{align*}
Consider the continuity of $\hat{\sigma}^*(\gamma, c; \tilde{m}(k), k)$,
\begin{align*}
    &|\hat{\sigma}^*(\gamma, c; \tilde{m}(k_n), k_n) - \hat{\sigma}^*(\gamma, c; \tilde{m}(k_0), k_0)|\\
    &\quad \leq |\hat{\sigma}^*(\gamma, c; \tilde{m}(k_n), k_n) - \hat{\sigma}^*(\gamma, c; \tilde{m}(k_n), k_0)| + |\hat{\sigma}^*(\gamma, c; \tilde{m}(k_n), k_0) - \hat{\sigma}^*(\gamma, c; \tilde{m}(k_0), k_0)|.
\end{align*}
The first term converges since $\hat{\sigma}^*(\gamma, c; \tilde{m}, k)$ is continuous in $k$ for a fixed $(\gamma, c, \tilde m)$, and the second term converges since $\hat{\sigma}^*(\gamma, c; \tilde{m}, k)$ is continuous in $\tilde{m}$ for each fixed $k$ by the argument in \Cref{sec: proof of uniqueness}. Therefore, for each fixed $(\gamma, c)$, the two participation probabilities $q(\theta; k_n)$ and $q(\theta; k_0)$ can differ only when $\sigma$ lies between the two corresponding thresholds, up to $\pi$-null indifference sets. Since $\pi$ has a bounded density, the same argument as in \Cref{sec: proof of uniqueness} gives
\[
\int_{\Theta} |q(\theta; k_n) - q(\theta; k_0)| d\pi(\theta) \to 0.
\]

\vspace{5pt}\paragraph{\textbf{Step 3. The insurer's objective function is continuous on $\mathcal K^{\mathrm{uniq}}$.}}

We now prove that the reduced insurer objective $\mathcal J^{\mathrm{uniq}}(k) = J(k; M(k))$ is continuous on $\mathcal K^{\mathrm{uniq}}$. For a sequence $k_n \to k_0 \in \mathcal K^{\mathrm{uniq}}$, we show that $\mathcal J^{\mathrm{uniq}}(k_n) \to \mathcal J^{\mathrm{uniq}}(k_0)$.

Recall that the insurer's per-capita expected profit under a contract $k$ and induced equilibrium $M(k)$ is
\begin{align*}
    \mathcal J^{\mathrm{uniq}}(k) = J(k; M(k)) &= \int_{\Theta} \Big(p - \alpha e^{\mathrm{in}}(\theta; k)- \beta \mu(e^{\mathrm{in}}(\theta; k) + x \tilde{m}(k))\Big) q(\theta; k) d\pi(\theta).
\end{align*}
For convenience, define
\begin{equation} \label{eq: H function}
H(\theta; k) := \Big(p - \alpha e^{\mathrm{in}}(\theta; k)- \beta \mu(e^{\mathrm{in}}(\theta; k) + x \tilde{m}(k))\Big)
\end{equation}
By Step 2, $\tilde{m}(k_n) \to \tilde{m}(k_0)$, and $e^{\mathrm{in}}(\theta; k_n) \to e^{\mathrm{in}}(\theta; k_0)$, since $\mu$ is continuous, it follows that
\[
\mu(e^{\mathrm{in}}(\theta; k_n) + x \tilde{m}(k_n)) \to \mu(e^{\mathrm{in}}(\theta; k_0) + x \tilde{m}(k_0)).
\]
We conclude that $H(\theta; k_n) \to H(\theta; k_0)$ for $\pi$-a.e.~$\theta$.

Next, we establish a uniform bound to justify passing the limit through the integral. Because $\{k_n\} \subset \mathcal K^{\mathrm{uniq}}$, the contract parameters are uniformly bounded: $p_n \leq p_{\max}$, $\alpha_n \leq \alpha_{\max}$, and $\beta_n \leq \beta_{\max}$. Furthermore, the bounded-effort assumption ensures $e^{\mathrm{in}}(\theta; k_n) + x\tilde{m}(k_n) \leq (1+x)E_{\max}$. Because $\mu$ is continuous on this compact interval, it is bounded by $M_{\mu} := \sup_{z} |\mu(z)| < \infty$. Therefore:
\[
|H(\theta; k_n)| \leq p_{\max} + \alpha_{\max} E_{\max} + \beta_{\max} M_{\mu} =: C_H < \infty,
\]
for all $n$ and $\pi$-a.e.~$\theta$. We now decompose
\begin{align*}
    |\mathcal J^{\mathrm{uniq}}(k_n) - \mathcal J^{\mathrm{uniq}}(k_0)| = &\Big|\int_{\Theta} H(\theta; k_n) q(\theta; k_n) d\pi(\theta) - \int_{\Theta} H(\theta; k_0) q(\theta; k_0) d\pi(\theta)\Big|\\
    \leq &\Big|\int_{\Theta} (H(\theta; k_n) - H(\theta; k_0)) q(\theta; k_n) d\pi(\theta)\Big|\\
    &+ \Big|\int_{\Theta} H(\theta; k_0) (q(\theta; k_n) - q(\theta; k_0))d\pi(\theta)\Big|.
\end{align*}
The first term converges to $0$ because $0 \leq q(\theta; k_n) \leq 1$, $H(\theta; k_n) \to H(\theta; k_0)$ pointwise, and the integrand is uniformly dominated by the integrable constant $2C_H$, the integral converges to $0$ by the Dominated Convergence Theorem. 
For the second term, we bound the integrand using the uniform constant $C_H$:
\[
\Big|\int_{\Theta} H(\theta; k_0) \big(q(\theta; k_n) - q(\theta; k_0)\big) d\pi(\theta)\Big| \leq C_H \int_{\Theta} |q(\theta; k_n) - q(\theta; k_0)| d\pi(\theta).
\]
Because $q$ converges in $L^1(\pi)$ as shown in Step 2, this term also vanishes. Thus, $\mathcal J^{\mathrm{uniq}}(k_n) \to \mathcal J^{\mathrm{uniq}}(k_0)$, proving that the objective function is continuous.
\vspace{5pt}\paragraph{\textbf{Step 4. Existence of an optimal contract.}}

By \Cref{assum_stack} the admissible contract set
\[
\mathcal K^{\mathrm{uniq}} = [0, p_{\max}] \times [\alpha_{\min}, \alpha_{\max}] \times [0, \beta_{\max}]
\]
is compact. From step 3, the insurer's reduced objective function $\mathcal J^{\mathrm{uniq}}(k) = J(k; M(k))$ is continuous on $\mathcal K^{\mathrm{uniq}}$.

Therefore, by the Weierstrass extreme value theorem, there exists a contract $k^* = (p^*, \alpha^*, \beta^*)$, such that
\[
\mathcal J^{\mathrm{uniq}}(k^*) = \max_{k \in \mathcal K^{\mathrm{uniq}}} \mathcal J^{\mathrm{uniq}}(k).
\]
By Step 1, the contract $k^*$ induces a unique lower-level mean-field Nash equilibrium $M(k^*) = M(p^*, \alpha^*, \beta^*)$, hence the pair $(k^*, M(k^*))$ satisfies that
\begin{enumerate}
    \item[(i)] $M(k^*)$ is the unique mean-field Nash equilibrium of the agents' game under the contract $k^*$;

    \item[(ii)] given the equilibrium response, the contract $k^*$ maximizes the insurer's expected profit over all admissible contracts in $\mathcal K^{\mathrm{uniq}}$.
\end{enumerate}
Therefore, $(k^*, M(k^*))$ constitutes a Stackelberg equilibrium of the insurer-agent game.

\qed

\end{document}